\documentclass[11pt,reqno]{amsart} 

\usepackage[displaymath, mathlines]{lineno}

\usepackage{graphicx}

\usepackage{microtype}

\usepackage{textcase}

\usepackage[top=1.00in,bottom=1.00in,left=1.00in,right=1.00in]{geometry} 

\usepackage{xcolor}

\usepackage{enumitem}

\usepackage[labelfont={},font=bf]{caption}

\usepackage[font={md,small},labelfont=md]{subfig}

\usepackage[noadjust,nosort]{cite}

\usepackage{amsmath, amsthm, amssymb, tikz, mathtools, array, mathrsfs, tensor}

\usepackage{polynom}

\usepackage{esint}

\usepackage{multicol}

\usepackage{dsfont}
	\newcommand{\one}{\mathds{1}}

\usepackage{framed}

\usepackage{upref}

\usepackage{stackengine}

\usepackage{hyperref}
	\hypersetup{colorlinks=true,linkcolor=blue,citecolor=blue,urlcolor=blue} 

\graphicspath{{"/Users/erik/Documents/Academic/Research/1_Postdoc/Airy/Disjoint geodesics/graphics/"}}

\allowdisplaybreaks

\numberwithin{equation}{section}


\newcommand{\eq}[1]{\begin{linenomath}\postdisplaypenalty=0\begin{align*} #1 \end{align*}\end{linenomath}}
\newcommand{\eeq}[1]{\begin{linenomath}\postdisplaypenalty=0\begin{align} \begin{split} #1 \end{split} \end{align}\end{linenomath}}

\newcommand{\stackref}[2]{\stackrel{\mbox{\footnotesize{\eqref{#1}}}}{#2}}
\newcommand{\stackrefp}[2]{\stackrel{\phantom{\mbox{\footnotesize{\eqref{#1}}}}}{#2}}

\def\eps{\varepsilon}
\def\vphi{\varphi}

\newcommand{\E}{\mathbb{E}}

\newcommand{\N}{\mathbb{N}}
\renewcommand{\P}{\mathbb{P}}
\newcommand{\Q}{\mathbb{Q}}
\newcommand{\R}{\mathbb{R}}

\newcommand{\Z}{\mathbb{Z}}

\renewcommand{\AA}{\mathcal{A}}
\newcommand{\BB}{\mathcal{B}}

\newcommand{\DD}{\mathcal{D}}
\newcommand{\EE}{\mathcal{E}}
\newcommand{\FF}{\mathcal{F}}

\newcommand{\LL}{\mathcal{L}}
\newcommand{\MM}{\mathcal{M}}

\newcommand{\OO}{\mathcal{O}}
\newcommand{\PP}{\mathcal{P}}

\newcommand{\UU}{\mathcal{U}}

\newcommand{\WW}{\mathcal{W}}
\newcommand{\XX}{\mathcal{X}}
\newcommand{\YY}{\mathcal{Y}}
\newcommand{\ZZ}{\mathcal{Z}}

\newcommand{\wt}[1]{\widetilde{#1}}

\DeclareMathOperator{\diam}{diam}
\DeclareMathOperator{\dist}{dist}

\DeclareMathOperator{\Gr}{Graph}

\newcommand{\dd}{\mathrm{d}} 

            \makeatletter
            \DeclareFontFamily{OMX}{MnSymbolE}{}
            \DeclareSymbolFont{MnLargeSymbols}{OMX}{MnSymbolE}{m}{n}
            \SetSymbolFont{MnLargeSymbols}{bold}{OMX}{MnSymbolE}{b}{n}
            \DeclareFontShape{OMX}{MnSymbolE}{m}{n}{
                <-6>  MnSymbolE5
               <6-7>  MnSymbolE6
               <7-8>  MnSymbolE7
               <8-9>  MnSymbolE8
               <9-10> MnSymbolE9
              <10-12> MnSymbolE10
              <12->   MnSymbolE12
            }{}
            \DeclareFontShape{OMX}{MnSymbolE}{b}{n}{
                <-6>  MnSymbolE-Bold5
               <6-7>  MnSymbolE-Bold6
               <7-8>  MnSymbolE-Bold7
               <8-9>  MnSymbolE-Bold8
               <9-10> MnSymbolE-Bold9
              <10-12> MnSymbolE-Bold10
              <12->   MnSymbolE-Bold12
            }{}
            
            \let\llangle\@undefined
            \let\rrangle\@undefined
            \DeclareMathDelimiter{\llangle}{\mathopen}%
                                 {MnLargeSymbols}{'164}{MnLargeSymbols}{'164}
            \DeclareMathDelimiter{\rrangle}{\mathclose}%
                                 {MnLargeSymbols}{'171}{MnLargeSymbols}{'171}
            \makeatother
            
\DeclareMathOperator{\MDP}{MaxDisjtPoly}
\DeclareMathOperator{\MDG}{MaxDisjtGeo}
\DeclareMathOperator{\NI}{NonInt}
\DeclareMathOperator{\Supp}{Supp}
\newcommand{\lt}{\mathrm{L}}
\newcommand{\rt}{\mathrm{R}}

    \DeclareFontFamily{U}{matha}{\hyphenchar\font45}
    \DeclareFontShape{U}{matha}{m}{N}{ <-6> matha5 <6-7> matha6 <7-8>
    matha7 <8-9> matha8 <9-10> matha9 <10-12> matha10 <12-> matha12 }{}
    \DeclareSymbolFont{matha}{U}{matha}{m}{N}
    \DeclareFontFamily{U}{mathx}{\hyphenchar\font45}
    \DeclareFontShape{U}{mathx}{m}{N}{ <-6> mathx5 <6-7> mathx6 <7-8>
    mathx7 <8-9> mathx8 <9-10> mathx9 <10-12> mathx10 <12-> mathx12 }{}
    \DeclareSymbolFont{mathx}{U}{mathx}{m}{N}
    
    \DeclareMathDelimiter{\llbrack} {4}{matha}{"76}{mathx}{"30}
    \DeclareMathDelimiter{\rrbrack} {5}{matha}{"77}{mathx}{"38}

\newtheorem{thm}{Theorem}[section]
\newtheorem{prop}[thm]{Proposition}
\newtheorem{cor}[thm]{Corollary}
\newtheorem{lemma}[thm]{Lemma}
\newtheorem{claim}[thm]{Claim}
\newtheorem{ourthm}{Theorem}
\newtheorem{theirthm}{Theorem} 

\theoremstyle{definition}
\newtheorem{defn}[thm]{Definition}

\newtheorem{remark}[thm]{Remark}

\renewcommand{\thefootnote}{\fnsymbol{footnote}}

\title[Endpoints of disjoint geodesics in the directed landscape]{Hausdorff dimensions for shared endpoints of disjoint geodesics in the directed landscape}


\author{Erik Bates}
\thanks{E.B. was partially supported by NSF grant DMS-1902734.} 
\address{\hspace{-0.18in}Department~of Mathematics, University of Wisconsin--Madison, Van Vleck Hall, 480 Lincoln Drive, Madison, WI 53706-1324}
\email{ewbates@wisc.edu}

\author{Shirshendu Ganguly}
\thanks{S.G. was partially supported by NSF grant DMS-1855688 and a Sloan Research Fellowship in Mathematics.}
\address{\hspace{-0.18in}Department~of Statistics, University of California, Berkeley, 401 Evans Hall, Berkeley, CA 94720-3840}
\email{sganguly@berkeley.edu}

\author{Alan Hammond}
\thanks{A.H. was partially supported by NSF grant DMS-1855550 and a Miller Professorship at U.C. Berkeley.}
\address{\hspace{-0.18in}Departments~of Mathematics and Statistics, University of California, Berkeley, 899 Evans Hall, Berkeley, CA 94720-3840}
\email{alanmh@berkeley.edu}

\begin{document}
\bibliographystyle{acm}

\renewcommand{\thefootnote}{\arabic{footnote}} \setcounter{footnote}{0}

\begin{abstract}
Within the Kardar--Parisi--Zhang universality class, the space-time Airy sheet is conjectured to be the canonical scaling limit for last passage percolation models.
In recent work \cite{dauvergne-ortmann-virag?} of Dauvergne, Ortmann, and Vir\'ag, this object was constructed and, upon a parabolic correction, shown to be the limit of one such model: Brownian last passage percolation.
The limit object without parabolic correction, called the directed landscape, admits geodesic paths between any two space-time points $(x,s)$ and $(y,t)$ with $s<t$.
In this article, we examine fractal properties of the set of these paths.
Our main results concern exceptional endpoints admitting disjoint geodesics.
First, we fix two distinct starting locations $x_1$ and $x_2$, and consider geodesics traveling $(x_1,0)\to (y,1)$ and $(x_2,0)\to (y,1)$.
We prove that the set of $y\in\R$ for which these geodesics coalesce only at time $1$ has Hausdorff dimension one-half.
Second, we consider endpoints $(x,0)$ and $(y,1)$ between which there exist two geodesics intersecting only at times $0$ and $1$.
We prove that the set of such $(x,y)\in\R^2$ also has Hausdorff dimension one-half.
The proofs require several inputs of independent interest, including (i) connections to the so-called {\it difference weight profile} studied in \cite{basu-ganguly-hammond21}; and
(ii) a tail estimate on the number of disjoint geodesics starting and ending in small intervals.
The latter result extends the analogous estimate proved for the prelimiting model in \cite{hammond20}.
\vspace{-2.4\baselineskip}
\end{abstract}

\maketitle
\tableofcontents


\section{Introduction}

\subsection{Random growth models and Kardar--Parisi--Zhang universality}
The Kardar--Parisi--Zhang (KPZ) universality class is a broad collection of 
random growth models sharing the asymptotic features exhibited by solutions to a stochastic PDE known as the KPZ equation \cite{corwin12,corwin16,hairer-quastel18}.
The models known or believed to belong to this collection, including asymmetric exclusion processes, first and last passage percolation, and directed polymers in random media, are characterized by the combination of local growth driven by white noise and a smoothing effect from some notion of surface tension.
For these $(1+1)$-dimensional models, the resulting growth interface $h(t,x)$ manifests a triple $(1,\frac{1}{3},\frac{2}{3})$ of exponents: at time $t$, the value of $h(t,x)$ is of order $t^1$, deviations of $h(t,x)$ from its mean are of order $t^{1/3}$, and fluctuations of this same order are observed when $x$ is varied on the scale of $t^{2/3}$.
Furthermore, once $h(t,x)$ is properly centered and rescaled according to these exponents, a universal limit emerges as $t\to\infty$ \cite{quastel-remenik14}.

For most models, the picture just described is conjectural even if representing the consensus view.
Nevertheless, recent developments have confirmed the convergence of several exactly solvable models  
to a well-defined scaling limit.
Depending on the level of information one seeks to retain as $t\to\infty$, various limiting objects can be discussed, in the same way that a standard normal random variable, a multivariate Gaussian, and Brownian motion can all arise from a common central limit theorem.
And just as Brownian motion---the most general scaling limit in this list---possesses interesting fractal properties, the analogous KPZ scaling limit invites inquiries into its own fractal geometry.
This object was introduced in \cite{dauvergne-ortmann-virag?} and named the \textit{directed landscape}. 
Before we introduce precise notation, let us describe which geometric features this paper will investigate.

The directed landscape is a random function which assigns a passage time $\LL(x,s;y,t)$ between any two space-time points $(x,s)$ and $(y,t)$ with $s<t$.
It respects the usual last passage composition rule,
\eq{
\LL(x,s;y,t) = \sup_{z\in\R} [\LL(x,s;z,r)+\LL(z,r;y,t)] \quad \text{for all $r\in(s,t)$,}
}
which allows one to define the passage time $\LL(\gamma)$ of any particular path $\gamma : [s,t]\to\R$; see definition \eqref{length_def}.
Then $\LL(x,s;y,t)$ is equal to the largest $\LL(\gamma)$ among paths satisfying $\gamma(s)=x$ and $\gamma(t)=y$, and a path achieving this maximum is called a \textit{geodesic}.
Typically geodesics are unique, and those with a shared endpoint typically coalesce before reaching that endpoint.
We are interested in the exceptional cases violating these properties.

Our first consideration is of the following scenario.
Fixing the starting locations $x_1,x_2$ and the time interval $[s,t]=[0,1]$, let $\DD_{x_1,x_2}\subset\R$ be the set of terminal locations $y\in\R$ for which there exist geodesics $(x_1,0)\to (y,1)$ and $(x_2,0) \to (y,1)$ whose only point of intersection is the endpoint itself; see Figure~\ref{yes_D1}.
These exceptional endpoints form a random perfect, nowhere dense set for which we have the following result.

\begin{ourthm} \label{thm_1}
The Hausdorff dimension of $\DD_{x_1,x_2}$ is almost surely $\frac{1}{2}$.
\end{ourthm}

There is an analogous bivariate scenario.
Fixing now only the interval $[s,t]=[0,1]$, let $\DD\subset\R^2$ be the set of $(x,y)$ for which there exist two geodesics $(x,0)\to (y,1)$ that intersect only at the endpoints; see Figure~\ref{yes_D2}.
Interestingly, this second exceptional set can be related to the first by associating to $\LL$ certain random measures with fractal supports.
In developing this connection, we also obtain the following.

\begin{ourthm} \label{thm_2}
The Hausdorff dimension of $\DD$ is almost surely $\frac{1}{2}$.
\end{ourthm}


We will restate these two results in Theorems~\ref{main_thm} and \ref{main_thm_2}, after having properly defined the relevant objects.
In fact, we expand on these dimension calculations by identifying $\DD_{x_1,x_2}$ and $\DD$ as the supports of random fractal measures which arise naturally out of the directed landscape (see also Section \ref{motivation_section} for further motivation).
But first we define \textit{Brownian last passage percolation (LPP)}, the semi-discrete model from which the directed landscape is realized as a scaling limit in \cite{dauvergne-ortmann-virag?}.

\begin{figure}[]
\centering
\subfloat[$y\notin \DD_{x_1,x_2}$]{
\includegraphics[trim=0.5in 0.5in 0.5in 0.5in, clip, width=0.48\textwidth]{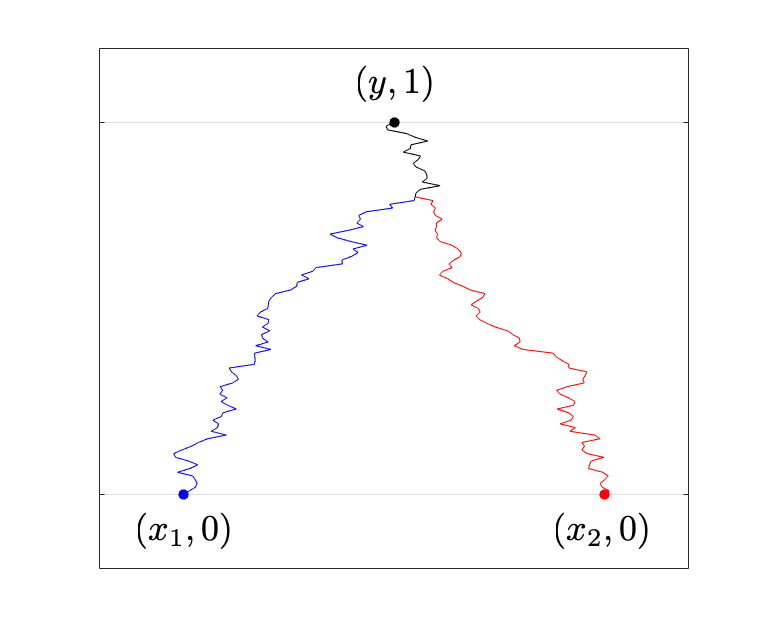}
\label{no_D1}
}
\hfill
\subfloat[$y\in \DD_{x_1,x_2}$]{
\includegraphics[trim=0.5in 0.5in 0.5in 0.5in, clip, width=0.48\textwidth]{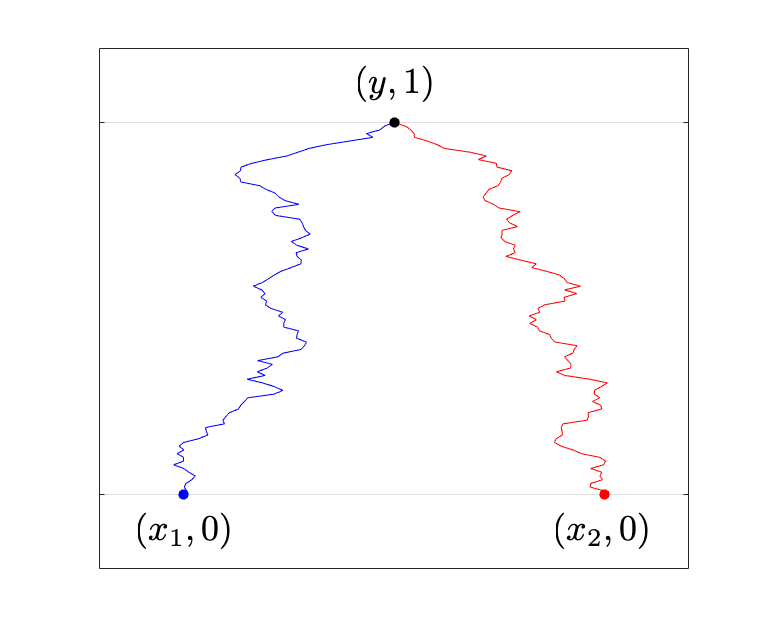}
\label{yes_D1}
}
\\
\subfloat[$(x,y)\notin \DD$]{
\includegraphics[trim=0.5in 0.5in 0.5in 0.5in, clip, width=0.48\textwidth]{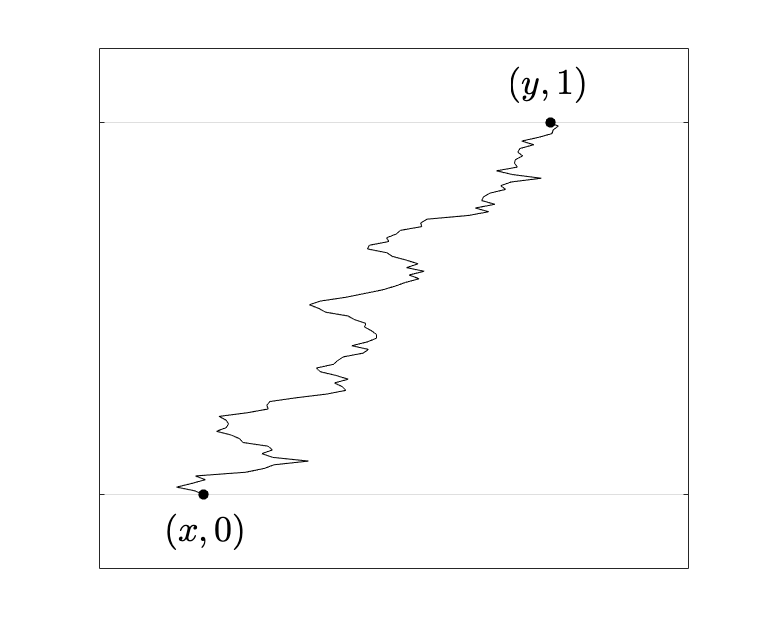}
\label{no_D2}
}
\hfill
\subfloat[$(x,y)\in \DD$]{
\includegraphics[trim=0.5in 0.5in 0.5in 0.5in, clip, width=0.48\textwidth]{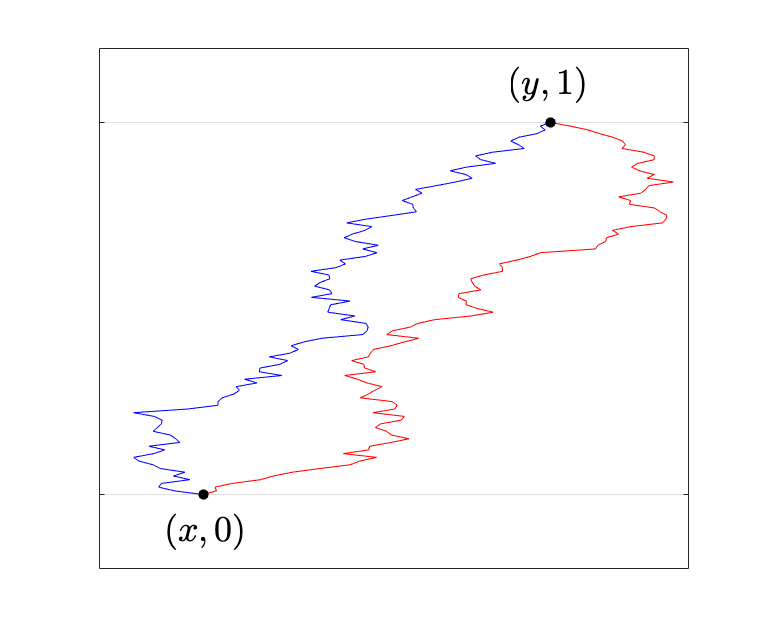}
\label{yes_D2}
}
\caption{Exceptionality of the sets $\DD_{x_1,x_2}$ and $\DD$.  
Time is visualized in the vertical direction, and space in the horizontal direction.  
The curves shown are graphs of geodesics as functions of the vertical coordinate.
For fixed $x_1,x_2$, and $y$, we will almost surely witness the scenario in (a), in which the coalescence of two geodesics sharing a terminal location $y$ happens before the terminal time $1$.
Similarly, for fixed $x$ and $y$, we will almost surely witness the scenario in (c), in which there is a unique geodesic associated to the pair of space-time points $(x,0)$ and $(y,1)$.}
\label{D_set}
\end{figure}




\subsection{Prelimiting model: Brownian last passage percolation} \label{prelimiting_model}
Let $B(\cdot,k) : \R\to\R$, $k\in\Z$, denote independent two-sided Brownian motions supported on a common probability space equipped with probability measure $\P$.
To each pair of real numbers $x \leq y$ together with any pair of integers $i\leq j$, we associate a passage time:
\eeq{ \label{blpp_def}
M(x,i;y,j) \coloneqq \sup\bigg\{\sum_{k=i}^j [B(z_{k+1},k)-B(z_k,k)]\ \bigg|\  x = z_i \leq z_{i+1} \leq \cdots \leq z_j \leq z_{j+1} = y\bigg\}.
}
That is, $M(x,i;y,j)$ is the largest number that can be obtained by partitioning the interval $[x,y]$ into $j-i+1$ ordered subintervals  $[z_i,z_{i+1}], [z_{i+1},z_{i+2}], \dots, [z_j, z_{j+1}]$ and then summing the Brownian increments incurred by traversing the subinterval $[z_k,z_{k+1}]$ using the $k^\text{th}$ Brownian motion.
This model, called \textit{Brownian LPP}, was first studied in \cite{glynn-whitt91} as the limit of a queueing problem. 

\subsubsection{Unscaled coordinates: staircases and a variational formula on functions}
By compactness and the continuity of the Brownian motions, there is at least one partition that achieves the supremum in \eqref{blpp_def}.
As another perspective, the candidate partitions are in bijection with right-continuous, non-decreasing functions $\vphi : [x,y) \to \llbrack i,j\rrbrack$, where  $\llbrack i,j\rrbrack$ denotes the integer interval $\{i,i+1,\dots,j\}$.
Namely, $\vphi(z) = k$ precisely when $z \in [z_k,z_{k+1})$.
Therefore, we can express $M$ formally as
\eeq{ \label{blpp_def_2}
M(x,i;y,j) \coloneqq \sup_{\vphi} \int_x^y \dd B(z,\vphi(z)).
}
Each $\vphi$ can be associated to a ``staircase" path in $\R^2$ starting at $(x,i)$, ending at $(y,j)$, and consisting of alternating horizontal and vertical line segments; see Figure~\ref{staircase_horizontal}.

\begin{figure}[]
\center
\subfloat[Unscaled staircase given by $\vphi$]{
\includegraphics[trim=0in 1.5in 1.1in 2in, clip, height=2.5in]{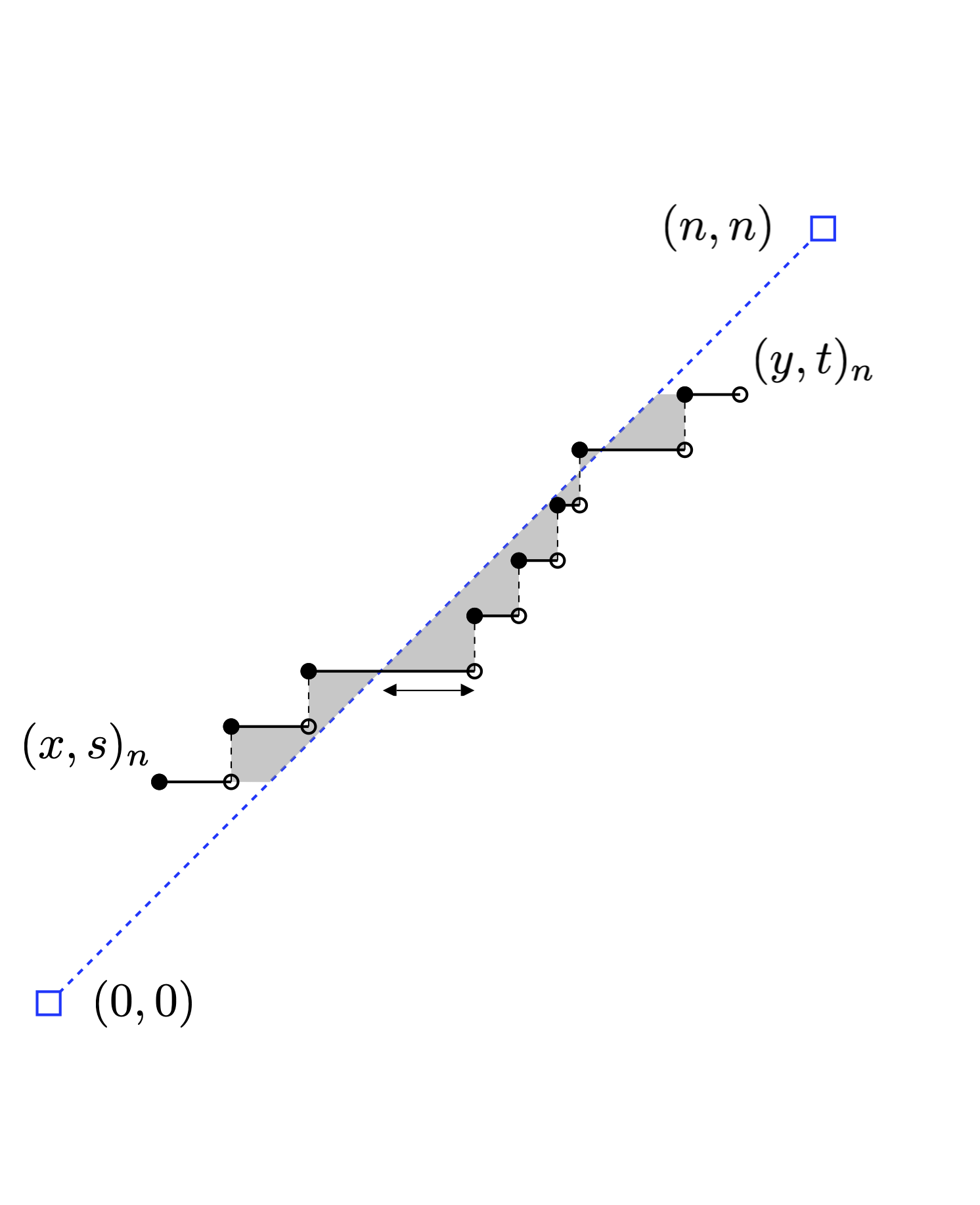}
\label{staircase_horizontal}	
}
\hfill
\subfloat[Scaled]{
\includegraphics[trim=2in 0in 18in 1.2in, clip, height=2.5in]{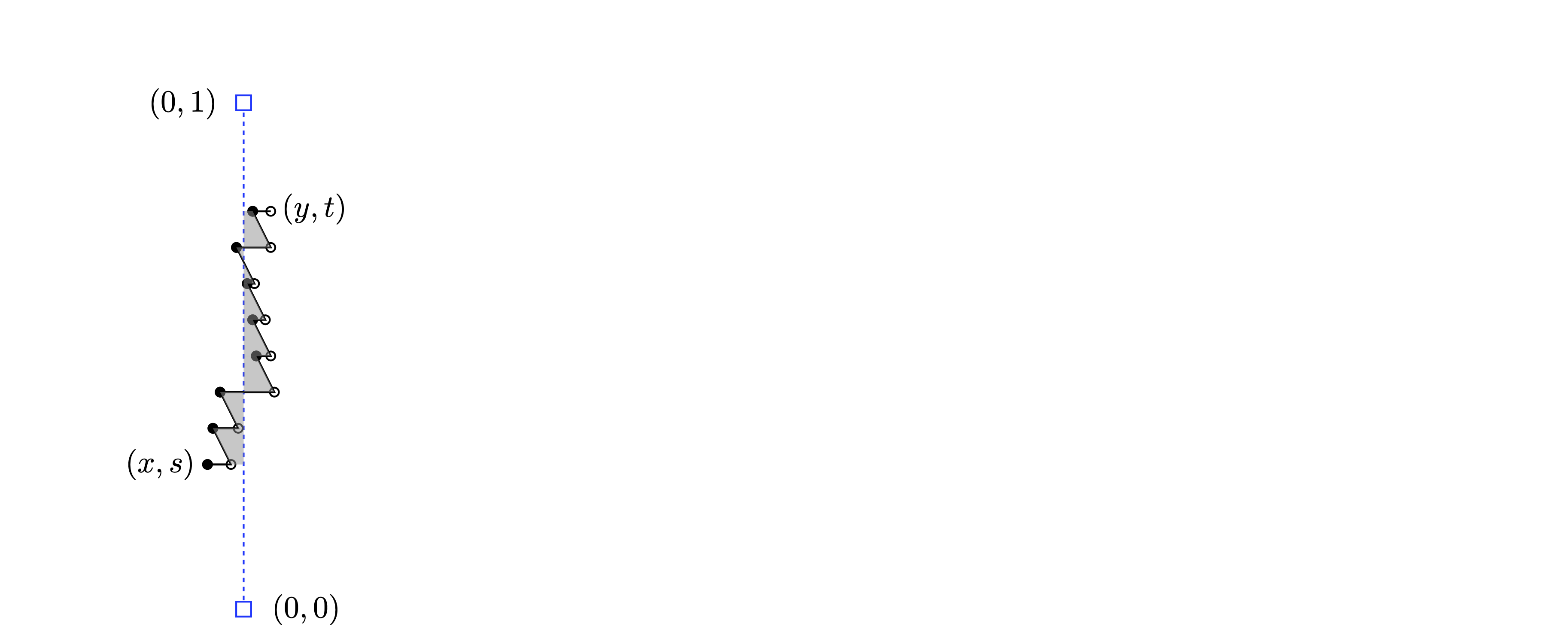}
\label{polymer_shade}
}
\hfill
\subfloat[$R_n(\vphi)$]{
\includegraphics[trim=2in 0in 18in 1.2in, clip, height=2.5in]{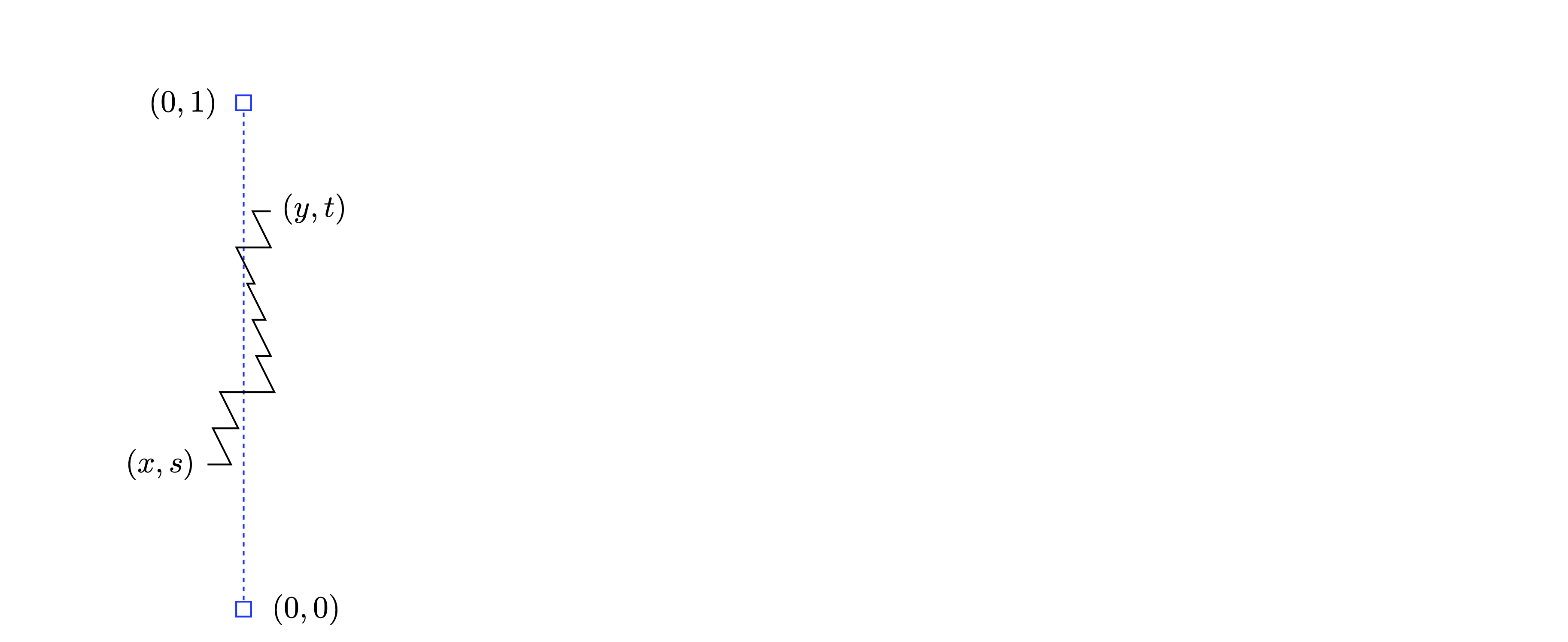}
\label{polymer_path}
}
\\
\subfloat[Unscaled staircase given by $\vphi$]{
\includegraphics[trim=3.3in 0in 12in 1.2in, clip, height=2.5in]{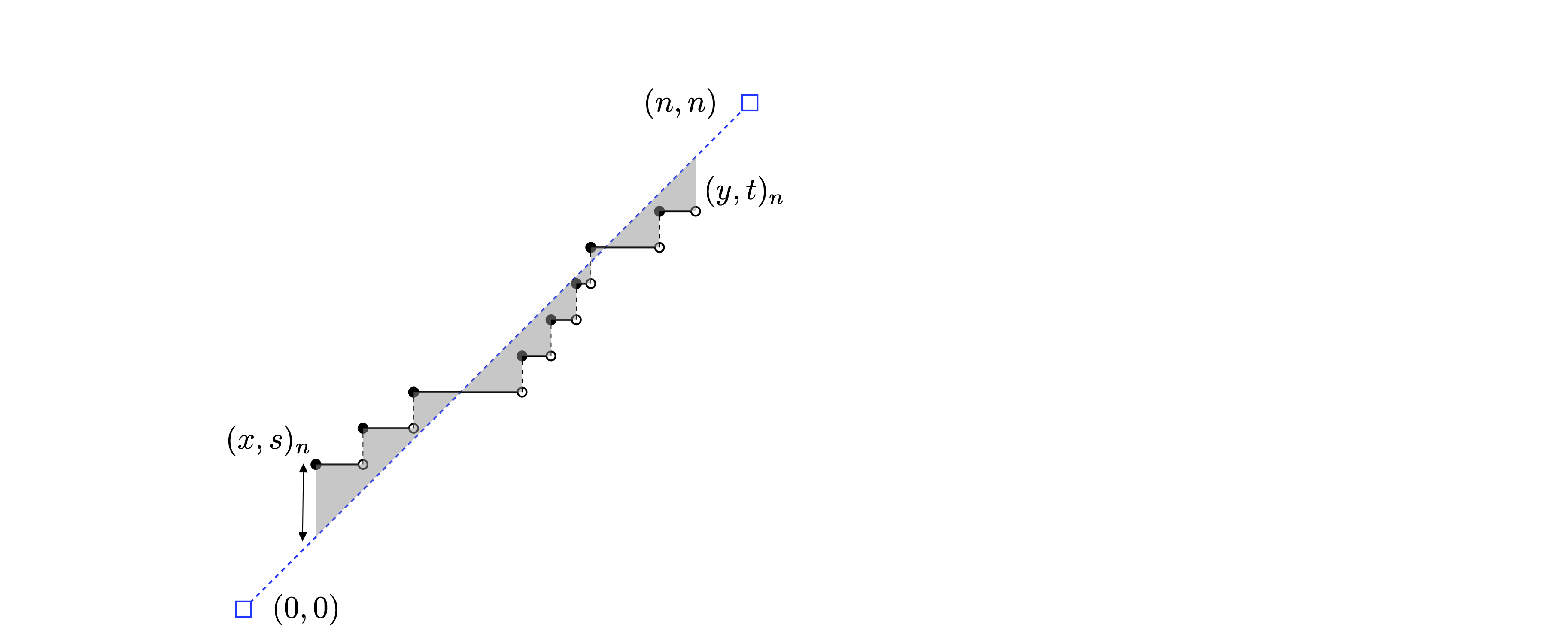}
\label{staircase_vertical}	
}
\hfill
\subfloat[$\Gamma_{n,u}^{(\vphi)}$]{
\includegraphics[trim=2in 0in 18in 1.2in, clip, height=2.5in]{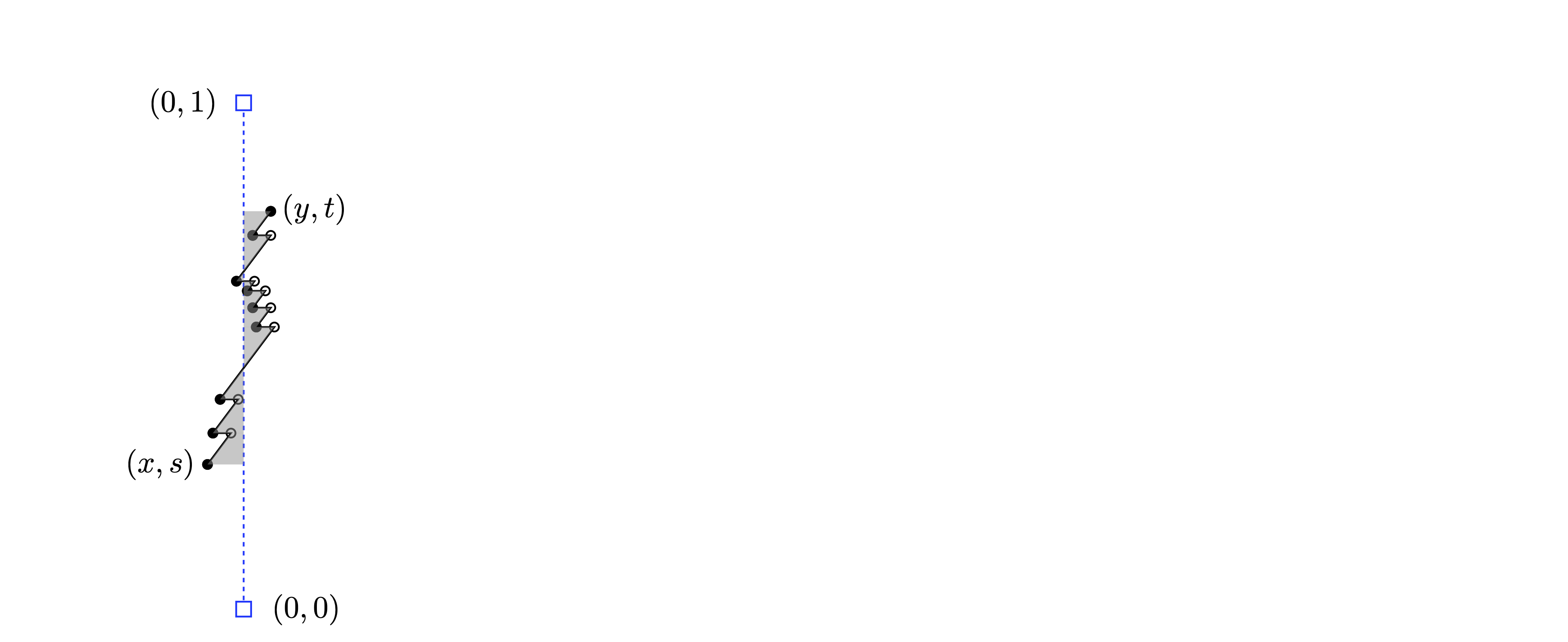}
\label{geodesic_shade}
}
\hfill
\subfloat[$\wt R_{n,u}(\vphi)$]{
\includegraphics[trim=2in 0in 18in 1.2in, clip, height=2.5in]{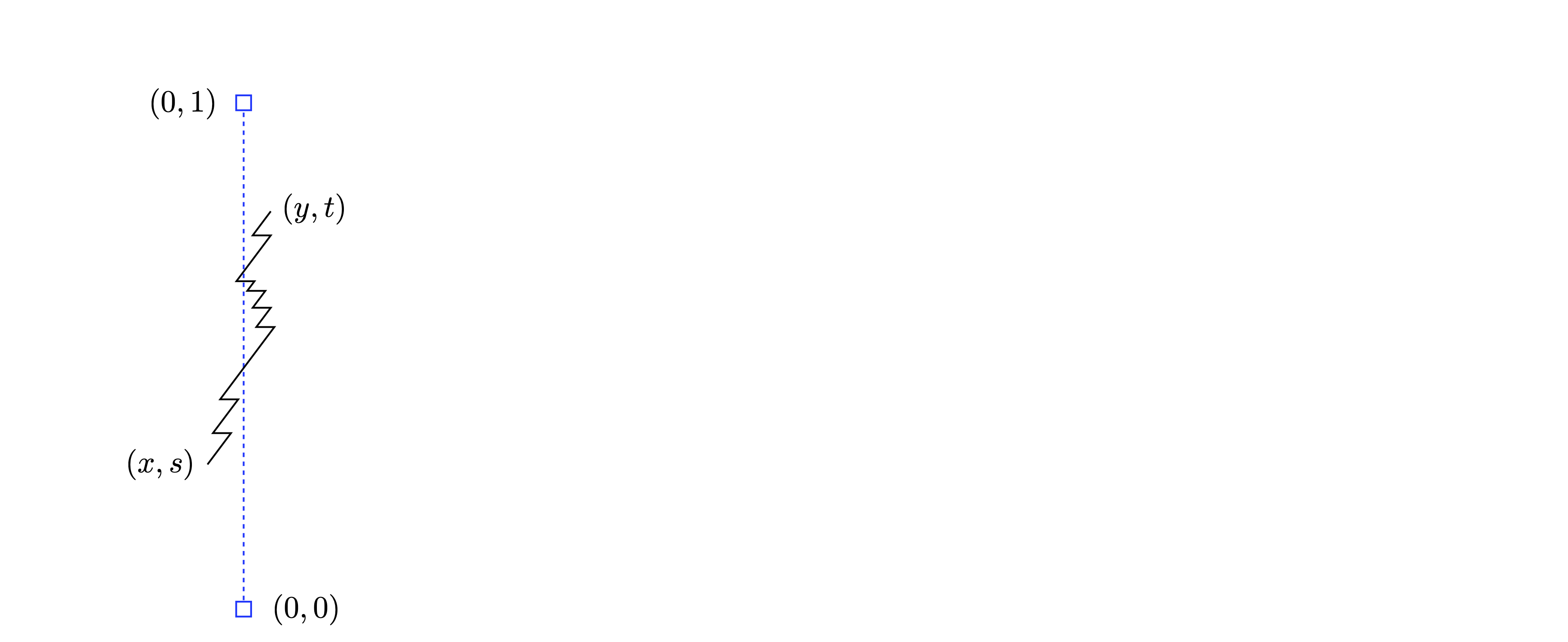}
\label{geodesic_path}
}
\caption{In this example, $x<0<y$, $0<s<t<1$, and we assume $sn,tn\in\Z$ for clarity.
The unscaled staircase in (a) and (d) between $(x,s)_n$ and $(y,t)_n$ is the graph of $\vphi$.
The horizontal segments are of the form $[z_k,z_{k+1}] \times \{k\}$ and connected via the vertical segments $\{z_{k+1}\}\times[k,k+1]$.
When deviations from the diagonal connecting $(0,0)$ and $(n,n)$ are measured as a function of the vertical coordinate and scaled by $2n^{2/3}$, the result is the zigzag in (b), equivalently realized by applying the scaling map $R_n$ to the picture in (a).
We will view the zigzag as a closed planar path $R_n(\vphi)\subset\R^2$, depicted in (c).
Its horizontal segments have length $\frac{1}{2}n^{-2/3}(z_{k+1}-z_k)$, while its oblique segments have slope $-2n^{-1/3}$ and extend over vertical distances that are multiples of $n^{-1}$.
Alternatively, when the deviations in (d) are regarded as a function of the horizontal coordinate and reparameterized via \eqref{n_geo_eq}, the resulting $\Gamma_{n,u}^{(\vphi)}$ is shown in (e) as a function of the time variable on the vertical axis.
The associated planar path $\wt R_{n,u}(\vphi)\subset\R^2$ is shown in (f). 
Its oblique segments have slope $(\frac{1}{2}n^{1/3}+\frac{y-x}{t-s})^{-1}$ and traverse a horizontal distance of $\frac{1}{2}n^{-2/3}(z_{k+1}-z_k)$, while the horizontal segments have lengths that are multiples of $\frac{1}{2}n^{-2/3}$.}
\label{blpp_fig}
\end{figure}

\subsubsection{Scaled coordinates: standardized passage times, zigzags, and polymers} \label{scaled_coordinates_sec}
For any $(x,i;y,j)$, the distribution of $M(x,i;y,j)$ can be inferred from that of $M(0,0;n,n)$ by Brownian rescaling, namely
\eeq{ \label{first_scaling_relation}
M(x,i;y,j) 
\stackrel{\text{dist}}{=} M(0,0;y-x,j-i)
\stackrel{\text{dist}}{=} \sqrt{\frac{y-x}{j-i}}M(0,0;j-i,j-i).
}
It is thus natural to study the quantity $M(0,0;n,n)$, which has the remarkable property of agreeing in distribution with the largest eigenvalue of an $(n+1)\times(n+1)$ matrix sampled from the Gaussian unitary ensemble (GUE) with entry variance $n$ \cite{baryshnikov01,gravner-tracy-widom01,oconnell-yor02}.
Therefore, the mean of $M(0,0;n,n)$ is asymptotic to $2n$ as $n\to\infty$, and its fluctuations about this mean are of order $n^{1/3}$, in agreement with KPZ scaling.
More precisely, if we define
\eeq{ \label{zero_zero}
W_n \coloneqq \frac{M(0,0;n,n) - 2n}{n^{1/3}},
}
then $W_n$ converges in law as $n\to\infty$ to the GUE Tracy--Widom distribution \cite[Thm.~3.1.4]{anderson-guionnet-zeitouni10}.

In order to recover a scaling limit for the joint process $(M(0,0;n+y,n))_{y\geq-n}$, one must also know on which scale to vary $y$ to induce fluctuations of order $n^{1/3}$ relative to $M(0,0;n,n)$.
A question which turns out to be equivalent is, by how much does a staircase achieving $M(0,0;n,n)$ in \eqref{blpp_def_2} deviate from the straight line connecting $(0,0)$ and $(n,n)$?
KPZ considerations put forward the answer of $n^{2/3}$, leading us to introduce the scaled coordinates
\eeq{ \label{scaled_coordinates}
(x,s)_n \coloneqq (sn+2n^{2/3}x,\lfloor sn\rfloor), \quad x,s\in\R,
}
so that $x$ and $s$ can be regarded as unit-order.
Generalizing \eqref{zero_zero}, we define the standardized passage time
\eeq{ \label{general_process}
W_n(x,s;y,t) \coloneqq \frac{M((x,s)_n;(y,t)_n) - 2(t-s)n - 2n^{2/3}(y-x)}{n^{1/3}}.
}

\begin{remark} \label{makes_sense_remark}
The definition \eqref{general_process} makes sense whenever $\lfloor sn\rfloor \leq \lfloor tn \rfloor $ and $sn+2n^{2/3}x\leq tn+2n^{2/3}y$.
When $s<t$, these requirements will be satisfied for all $n$ sufficiently large. 
Therefore, for any
\eq{
u \in \R^4_\uparrow \coloneqq \{(x,s;y,t) \in \R^4 : s<t\},
}
the quantity $W_n(u)$ is well-defined for all $n$ larger than some $n_0 = n_0(u)$.
Henceforth, whenever we consider $W_n(u)$ or any other object that depends on both $u$ and $n$, it will be implicitly assumed that $n\geq n_0(u)$.
For any compact $K \subset \R^4_\uparrow$, the choice of $n_0$ can made uniformly over $u\in K$.
\end{remark}

The scale prescribed by definitions \eqref{scaled_coordinates} and \eqref{general_process} has reshaped the original geometry of Brownian LPP via the linear transformation $R_n$ mapping $(2n^{2/3},0) = (1,0)_n \mapsto (1,0)$ and $(n,n) = (0,1)_n \mapsto (0,1)$.
The images under $R_n$ of the staircases from before are now ``zigzags" consisting of horizontal and oblique segments, 
as seen in Figures~\ref{polymer_shade} and \ref{polymer_path}.
Given a staircase function $\vphi$, let us write $R_n(\vphi)$ for the planar path determined by the associated zigzag.
To be completely precise, we make the following definitions.

\begin{defn}
For the right-continuous, non-decreasing $\vphi : [x,y)\to\llbrack i,j\rrbrack$ defined by $\vphi(z) = k$ for $z\in[z_k,z_{k+1})$, the \textit{zigzag} $R_n(\vphi)\subset\R^2$ is the image under $R_n$ of the following staircase:
\eq{
\underbrace{\bigcup_{k = i}^{j} ([z_k,z_{k+1}] \times \{k\})}_{\text{horizontal segments}} \cup \underbrace{\bigcup_{k=i}^{j-1} (\{z_{k+1}\}\times[k,k+1])}_{\text{vertical segments}}.
}
The reader might find this definition more transparent by simply examining Figure~\ref{polymer_path}.
\end{defn}

\begin{defn} \label{polymer_def}
For $u = (x,s;y,t)\in\R^4_\uparrow$, if $\vphi$ achieves the supremum in \eqref{blpp_def_2} for $M((x,s)_n;(y,t)_n)$, then we will say $R_n(\vphi)$ is an \textit{$n$-polymer} between the \textit{endpoints} $(x,s)$ and $(y,t)$.
Let $P_{n,u}$ denote the set of such polymers.
\end{defn}

For fixed $u$, there is almost surely a unique $n$-polymer in $P_{n,u}$ \cite[Lemma 4.6(1)]{hammond19I}, although there may be random exceptional $u\in\R^4_\uparrow$ admitting more than one.
While $n$-polymers are natural geometrically, they are formally subsets of $\R^2$ rather than functions.
When taking a limit $n\to\infty$, it is convenient to have a functional perspective of these objects.
This leads to the notion of \textit{$n$-geodesics}, which is discussed next.
Our choice of terminology is somewhat arbitrary, but we will need to distinguish between the two notions for technical reasons.


\subsubsection{Spatial deviations and geodesics}
After the application of $R_n$, the order $n^{2/3}$ spatial deviations mentioned before are observed as order $1$ deviations of polymers from the vertical axis connecting $(0,0)$ and $(0,1)$.
These fluctuations will be recorded as follows.
Given $u=(x,s;y,t)\in\R^4_\uparrow$ and a candidate $\vphi$ in \eqref{blpp_def_2} for $M((x,s)_n;(y,t)_n)$, i.e.,~a right-continuous, non-decreasing $\vphi: [sn+2n^{2/3}x,tn+2n^{2/3}y)\to\llbrack \lfloor sn\rfloor,\lfloor tn\rfloor\rrbrack$, we consider the function $\Gamma_{n,u}^{(\vphi)} : [s,t] \to \R$ given by
\eeq{ \label{n_geo_eq}
\Gamma_{n,u}^{(\vphi)}(r) \coloneqq \frac{L_{n,u}(r) - \vphi(L_{n,u}(r))}{2n^{2/3}}, \quad r\in[s,t), \qquad \Gamma_{n,u}^{(\vphi)}(t) \coloneqq \lim_{r\nearrow t} \Gamma_{n,u}^{(\vphi)}(r),
}
where 
\eeq{ \label{parameterized_time}
L_{n,u}(r) \coloneqq rn+\frac{t-r}{t-s}2n^{2/3}x+\frac{r-s}{t-s}2n^{2/3}y, \quad r\in[s,t].
}
That is, in unscaled coordinates, the vertical separation between the point $(L_{n,u}(r),L_{n,u}(r))$ and the staircase associated to $\vphi$ is exactly $\Gamma_{n,u}^{(\vphi)}(r)\cdot2n^{2/3}$; see Figures~\ref{staircase_vertical} and \ref{geodesic_shade}.

\begin{defn}
\label{n_geo_def}
For $u = (x,s;y,t)\in\R^4_\uparrow$, if $\vphi$ achieves the supremum in \eqref{blpp_def_2} for $M((x,s)_n;(y,t)_n)$, then $\Gamma_{n,u}^{(\vphi)}$ will be called an \textit{$n$-geodesic} between $(x,s)$ and $(y,t)$.
Let $G_{n,u}$ denote the set of such geodesics. 
\end{defn}


Henceforth, the variables $x$ and $y$ are to be thought of as spatial coordinates, despite their initial role as time coordinates in Brownian motions.
Instead, $s$ and $t$ are now the temporal coordinates, reflecting our desire to think of $\Gamma_{n,u}^{(\vphi)}$ as an upward moving path in $\R^2$ starting at $(x,s)$ and terminating at $(y,t)$, as illustrated in Figures~\ref{geodesic_shade} and \ref{geodesic_path}.
To be completely precise, though, we note that $\Gamma_{n,u}^{(\vphi)}(s)$ and $\Gamma_{n,u}^{(\vphi)}(t)$ are not necessarily exactly equal to $x$ and $y$, respectively.
For instance, the equality with $x$ will only be approximate if $sn\notin \Z$, or if $sn\in\Z$ but $\vphi(sn+2n^{2/3}x) > sn$.
When $\Gamma_{n,u}^{(\vphi)}$ is an $n$-geodesic, the latter scenario happens with probability zero.

Note that $\Gamma_{n,u}^{(\vphi)}\in G_{n,u}$ if and only if $R_n(\vphi) \in P_{n,u}$, and so $n$-geodesics and $n$-polymers are two (slightly) different ways of obtaining scaled versions of the maximizers in \eqref{blpp_def_2}.
The difference between the two objects is highlighted in Figure~\ref{blpp_fig}, and we will later discuss in Section \ref{inputs_2} how they nonetheless give rise to the same limiting object.

\subsection{Limiting model: the directed landscape} \label{limiting_model}

Section~\ref{scaled_coordinates_sec} saw the appearance of our first canonical object in the KPZ universality class, namely the Tracy--Widom law.
It has been proven to arise also in Poissonian LPP \cite{baik-deift-johansson99,baik-rains01}; the asymmetric simple exclusion process \cite{tracy-widom09}; the totally asymmetric simple exclusion process (TASEP) \cite{johansson00,ferrari-spohn06,borodin-ferrari-prahofer-sasamoto07}; the positive-temperature version of Brownian LPP \cite{borodin-corwin14,borodin-corwin-ferrari14} introduced by O'Connell and Yor~\cite{oconnell-yor01}; the continuum random polymer \cite{amir-corwin-quastel11,borodin-corwin-ferrari14} 
by Alberts, Khanin, and Quastel~\cite{alberts-khanin-quastel14I,alberts-khanin-quastel14II}; and the fully discrete log-gamma polymer \cite{borodin-corwin-remenik13} 
by Sepp\"al\"ainen~\cite{seppalainen12}.
Apart from its universality, the role of the Tracy--Widom law in our current setting is to describe the limiting behavior of any one-point distribution.
Indeed, upon knowing that $W_n = W_n(0,0;0,1)$ converges in distribution to a GUE Tracy--Widom random variable, one can infer---from \eqref{first_scaling_relation} and a Taylor expansion of the square root factor---that $W_n(x,s;y,t)$ obeys the same convergence but where the limit has been scaled by $(t-s)^{1/3}$ and then shifted by $-(y-x)^2/(t-s)$.

\subsubsection{The Airy$_2$ process and the space-time Airy sheet}
In the same way, any limiting \textit{multi}-point distributions (in a single spatial coordinate) can be deduced from the limiting behavior of the one-parameter process $y\mapsto W_n(0,0;y,1)$.
For this object and its counterparts in other one-dimensional growth models within the KPZ universality class, the distributional limit is a parabolically shifted \textit{Airy process}, our second canonical object.
In the setting of Brownian LPP, this means $y \mapsto W_n(0,0;y,1)+y^2$ converges in law as $n\to\infty$ to a stationary process $y\mapsto\AA_2(y)$ known as Airy$_2$, introduced in \cite{prahofer-spohn02}.
While this statement remains for most models a conjecture, it has been proved in the sense of finite-dimensional distributions for Poissonian LPP \cite{prahofer-spohn02,borodin-ferrari-sasamoto08I}
and TASEP \cite{borodin-ferrari-prahofer07,borodin-ferrari-sasamoto08II,baik-ferrari-peche10},
and as a functional limit theorem for geometric LPP \cite{johansson03}. 


Given these developments, it is only natural that more recent work has sought to understand the full four-parameter process $(x,s;y,t)\mapsto W_n(x,s;y,t)$.
In this case, the relevant limiting object (again after a parabolic correction) was conjectured in \cite{corwin-quastel-remenik15} to be the so-called \textit{space-time Airy sheet}, whose rigorous construction was left open. 
One view is that the function 
$y\mapsto W_n(0,0;y,t)$ is a Markov process in $t$ and evolves forward from a Dirac delta mass at $t=0$.
If this evolution has a large-$n$ limit, then the driving noise is the Airy sheet.
This perspective was advanced in \cite{matetski-quastel-remenik?}, where the authors allow very general initial conditions---the prelimiting model being TASEP---and derive determinantal formulas for the transition probabilities of the limiting Markov process.

Very recently in \cite{dauvergne-ortmann-virag?}, the Airy sheet was constructed directly from a last passage model on the Airy line ensemble from \cite{corwin-hammond14}, whose top curve is itself the distributional limit of $y\mapsto W_n(0,0;y,1)$.
Featured in \cite{dauvergne-ortmann-virag?} is an extension of the Robinson--Schensted--Knuth correspondence that allows the original Brownian LPP to be mapped to a last passage problem involving Brownian motions conditioned to not intersect.
It is this collection of non-intersecting Brownian motions that converges in a suitable scaling limit to the Airy line ensemble, hence the prospect---ultimately realized---that Brownian LPP has as its limit a last passage problem on this ensemble respecting said convergence.
(Separate works \cite{dauvergne-nica-virag?,dauvergne-virag?} make progress toward showing the same statement for other classical LPP models.)
We will now define the resulting object.


\begin{defn}
The \textit{directed landscape} is a random continuous function $\LL : \R^4_\uparrow \to \R$ almost surely satisfying
\eeq{ \label{max_prop}
\LL(x,s;y,t) = \sup_{z\in\R} [\LL(x,s;z,r)+\LL(z,r;y,t)] \quad \text{for all $(x,s;y,t)\in\R^4_\uparrow$, $r\in(s,t)$},
}
and having the property that for any disjoint intervals $(s_i,t_i)$, $i=1,\dots,n$, the following processes on $\R^2$ are independent and identically distributed:
\eeq{ \label{KPZ_scaling}
(x,y) \mapsto (t_i-s_i)^{-1/3}\LL(x(t_i-s_i)^{2/3},s_i;y(t_i-s_i)^{2/3},t_i), \quad i=1,\dots,n.
}
\end{defn}

The property \eqref{max_prop} means $\LL$ can be thought of as an ``anti-metric" in that it satisfies the reverse triangle inequality.
It was shown in \cite[Lemma 10.3]{dauvergne-ortmann-virag?} that $\LL$ exists and has a unique law determined entirely by the law of the two-parameter process $(x,y)\mapsto \LL(x,0;y,1)$.
The space-time Airy sheet is obtained by the parabolic correction $\AA(x,s;y,t) \coloneqq \LL(x,s;y,t)+\frac{(x-y)^2}{t-s}$. 
For any fixed $x\in\R$, the map $y\mapsto \AA(x,0;y,1)$ is an Airy$_2$ process, and $\AA$ also satisfies space-time stationarity,
\eq{ 
\AA(x,s;y,t) \stackrel{\text{dist}}{=} \AA(x+z,s+r;y+z,t+r) \quad \text{for any $z,r\in\R$.}
}
The following convergence result justifies the consideration of $\LL$ as a canonical object in the KPZ universality class.

\begin{theirthm}
\label{blpp_airy}
\textup{\cite[Thm.~1.5]{dauvergne-ortmann-virag?}} There exists a coupling of Brownian LPP and $\LL$ on some probability space $(\Omega,\FF,\P)$ so that
\eq{
W_n(x,s;y,t) = \LL(x,s;y,t) + o_n(x,s;y,t),
}
where $o_n$ is a random function admitting, for every compact $K\subset \R^4_\uparrow$, a deterministic constant $a>1$ such that $\E(a^{\sup_K|o_n|^{3/4}}) \to 1$.
\end{theirthm}

\begin{remark} \label{complete_remark}
We may assume without loss of generality that $\FF$ is complete.
That is, if $A_0 \in\FF$ satisfies $\P(A_0) = 0$, and $A \subset A_0$, then $A\in\FF$. 
Equivalently, if $A_1 \in\FF$ satisfies $\P(A_1) = 1$, and $A \supset A_1$, then $A\in\FF$. 
This assumption will be a technical convenience when we check the measurability of certain events.
\end{remark}

\subsubsection{Geodesics}

In the fully continuous setting of the directed landscape, the analogues of the $n$-polymers from Definition~\ref{polymer_def} are fractal, upward moving paths in $\R^2$ like the one seen in Figure~\ref{composition}.
Meanwhile, the analogue of the $n$-geodesics from Definition \ref{n_geo_def} are the functions parameterizing these paths.
For our purposes, it is no longer important in the limiting setting to differentiate the paths from their natural parameterizations, and so we will just use the latter, as made precise in Definition \ref{landscape_geodesic_def}.

Given a realization of $\LL$, each continuous $\gamma : [s,t]\to\R$ has a \textit{length} $\LL(\gamma)$ given by
\eeq{ \label{length_def}
\LL(\gamma) \coloneqq \inf_{k\in\N}\inf_{s=t_0<t_1<\cdots<t_k=t}\sum_{i=1}^k \LL(\gamma(t_{i-1}),t_{i-1};\gamma(t_i),t_i).
}
This definition is in analogy with that of Euclidean length, except that infima replace suprema because of the anti-metric nature of $\LL$.
If $\gamma(s) = x$ and $\gamma(t) = y$, then the coordinate pairs $(x,s)$ and $(y,t)$ are referred to as the \textit{endpoints} of $\gamma$.
By taking $k = 1$ in \eqref{length_def}, it is clear that $\LL(\gamma) \leq \LL(x,s;y,t)$.
The limiting version of \eqref{blpp_def_2} is now
\eq{
\LL(u) = \sup_{\gamma : (x,s) \to (y,t)} \LL(\gamma), \quad u = (x,s;y,t)\in\R^4_\uparrow,
}
where the supremum is taken over continuous $\gamma: [s,t]\to\R$ such that $\gamma(s) = x$ and $\gamma(t)=y$.

\begin{defn}
\label{landscape_geodesic_def}
Let $u = (x,s;y,t)\in\R^4_\uparrow$ and suppose $\gamma : [s,t]\to\R$ is a continuous function such that $\gamma(s)=x$ and $\gamma(t) = y$. 
If $\LL(\gamma) = \LL(u)$, then we say $\gamma$ is a \textit{geodesic} between $(x,s)$ and $(y,t)$.
Let us write $G_u$ to denote the set of all geodesics from $(x,s)$ to $(y,t)$.
\end{defn}


The collection of all geodesics was termed the \textit{polymer fixed point} in \cite{corwin-quastel-remenik15} and is thought to also be universal to the KPZ class.
For instance, the authors of \cite{corwin-quastel-remenik15} suggest the polymer fixed point might also be realized as the zero-temperature limit of the continuum directed random polymer \cite{alberts-khanin-quastel14II}, and so their use of ``polymer" serves as a nod to positive-temperature models.
This usage is consistent with the convention of reserving ``polymer" to refer to a positive-temperature object (a sample from a \textit{measure} on paths) while keeping to ``geodesic" for a zero-temperature object (a single path of maximal energy).
As this paper deals only with zero-temperature models, we have deviated from this convention in order to clearly distinguish between Definitions~\ref{polymer_def} and \ref{n_geo_def}.

The existence of geodesics is a consequence of \eqref{max_prop}, since one can consider $\gamma$ defined by $\gamma(r) = z_r$, where $z_r$ is a suitably chosen maximizer in \eqref{max_prop}. 
(For instance, see Lemma~\ref{geodesic_construction}.)
Typically the maximizer $z_r$ is unique for each $r\in(s,t)$, in which case
$G_u$ is a singleton; that is, for fixed $u\in\R^4_\uparrow$, we have $|G_u|=1$ with probability one
\cite[Thm.~12.1]{dauvergne-ortmann-virag?}.
It is also a fact that geodesics are typically H\"older-$2/3^-$ continuous in time \cite[Thm.~1.7]{dauvergne-ortmann-virag?} but \textit{not} H\"older-$2/3$ \cite[Thm.~10.2]{dauvergne-sarkar-virag?}.
Nevertheless, as in the prelimit, there may exist random exceptional points for which these statements are not true.

\begin{remark}
Admittedly, it is an abuse of notation to write $\LL(\gamma)$ given that $\LL$ was defined to be a continuous function on $\R^4_\uparrow$.
Nevertheless, this notational convenience should not lead to any confusion, as we will adhere to the following conventions:
\begin{itemize}
\item $u$ always denotes an element of $\R^4_\uparrow$.
\item $x,y,z,w,p,q$ are spatial coordinates (typically $x,z,p$ are associated to initial endpoints of geodesics, and $y,w,q$ to terminal endpoints).
\item $r,s,t$ are temporal coordinates (typically $s \leq r \leq t$).
\item $\vphi,\phi$ are maximizers achieving $M_n(x,i,y,j)$ in \eqref{blpp_def_2}, where $(x,i;y,j)$ will be apparent from context (in particular, $\vphi$ and $\phi$ are right-continuous, non-decreasing $\Z$-valued functions).
\item $\gamma$ denotes a continuous function of time, an object in the limiting model.
\item $\Gamma$ denotes the corresponding prelimiting object (generally a discontinuous function).
\item $a_j\nearrow a$ means $a_j\leq a_{j+1}$ for all $j$, and $a_j$ converges to $a$ as $j\to\infty$; similarly for $a_j\searrow a$.
\end{itemize}
\end{remark}

\begin{figure}[!t]
\centering
\includegraphics[trim=0.9in 0.7in 0.5in 0.5in, clip, width=0.55\textwidth]{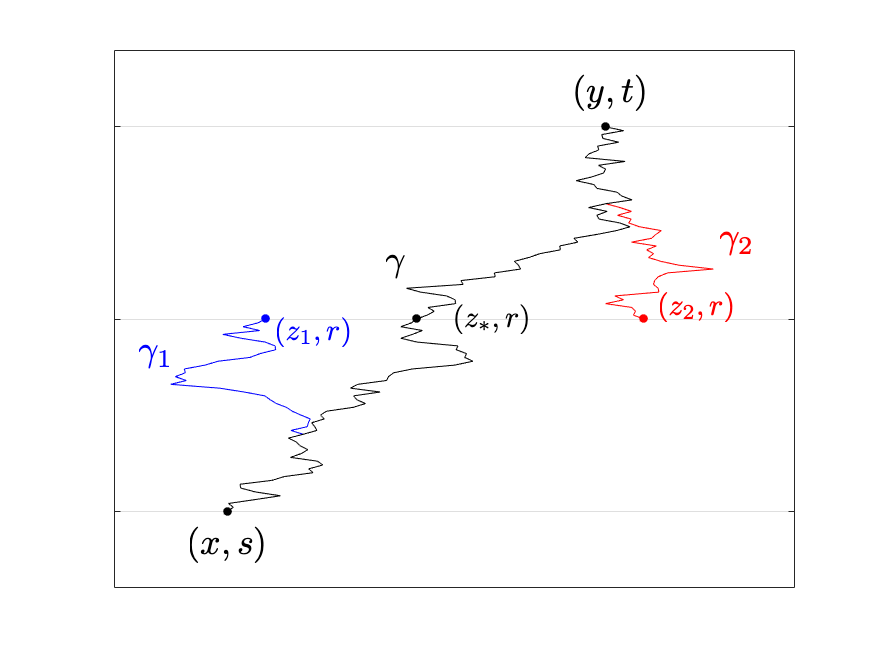}
\caption{Example geodesics in $\LL$.  
Time is visualized in the vertical direction.
Shown above is a geodesic $\gamma$ between $(x,s)$ and $(y,t)$ that passes through $(z_*,r)$, meaning $z_*$ achieves the supremum in \eqref{max_prop}.
For generic $z_1\neq z_*$, geodesics $\gamma_1$ and $\gamma$ will typically coincide for some random length of time, and then remain disjoint after said time.
Similarly, geodesics $\gamma_2$ and $\gamma$ will typically coincide for all times after their first intersection.
}
\label{composition}
\end{figure}

The following result from \cite{dauvergne-ortmann-virag?} confirms that the directed landscape retains the limiting information not only of the passage times in Brownian LPP, but also of the maximizing paths comprising the polymer fixed point.
Recall that for $u = (x,s;y,t)$ and sufficiently large $n$, $G_{n,u}$ denotes the set of $n$-geodesics $\Gamma_{n,u}^{(\vphi)} : [s,t]\to\R$ defined in \eqref{n_geo_eq}, where $\vphi$ is a maximizer in \eqref{blpp_def_2}.

\begin{theirthm}
\label{geodesics_converge}
\textup{\cite[Thm.~1.8 and 13.5]{dauvergne-ortmann-virag?}}
In the coupling of Theorem~\ref{blpp_airy}, there exists an event $\PP$ of probability $1$ such that the following holds for any $u\in\R^4_\uparrow$. 
On the almost sure event $\PP \cap \{|G_u|=1\}$, if $\gamma_u$ is the unique element of $G_u$, and $\Gamma_{n,u}^{(\vphi_n)} \in P_{n,u}$ for each $n$, then
\eq{
\lim_{n\to\infty} \|\Gamma_{n,u}^{(\vphi_n)} - \gamma_u\|_{\infty} = 0,
}
where $\|\cdot\|_\infty$ denotes the sup-norm on $[s,t]$.
\end{theirthm}

The event $\PP$ is that $\LL$ is \textit{proper}, a notion defined in \cite[Sec.~13]{dauvergne-ortmann-virag?} and recalled in Definition~\ref{proper_def}.
We will assume throughout the paper that this event occurs.


\subsection{Motivation for investigating fractal geometry} \label{motivation_section}
Before stating our main results in the upcoming Section~\ref{main_results}, we wish to recognize a broader context in which they can be understood.
The present study is motivated by the view that the fractal properties of the polymer fixed point have much to say about the probabilistic structure of the directed landscape.
Indeed, this philosophy has been borne out in many areas of probability, in particular other universality classes of statistical mechanics.
Let us highlight a few examples, beginning with the classical:
\begin{itemize}

\item The Hausdorff dimension of the zero set of Brownian motion is almost surely $\frac{1}{2}$.
This set can be regarded, of course, as the support of a random measure whose distribution function is the local time at zero.
In fact, after this paper was first posted, it was shown in \cite{ganguly-hegde?} that the function $\ZZ_{x_1,x_2}$ to be defined in Section \ref{main_results} does, in several senses, locally resemble Brownian local time.

\item For Bouchaud trap models with heavy-tailed trapping times, the relevant scaling limits are given by $\alpha$-stable processes which capture the effective time scale of the random walk.
Similar results hold for biased random walks on supercritical trees, and there has been some progress for biased random walks on supercritical percolation clusters; see \cite{benarous-fribergh16} for a survey.
Questions concerning Hausdorff dimension can also be asked of closely related scaling limits for continuous-time random walks on $\R^d$, e.g.~\cite{meerschaert-nane-xiao13}.

\item The now expansive literature on Schramm--Loewner evolutions (SLE) was initiated in \cite{schramm00} to understand fractal curves arising in two-dimensional critical phenomena.
For instance, the interfaces of critical percolation on the triangular lattice converge to SLE$_6$ \cite{smirnov01,smirnov06,camia-newman07}, those of the FK--Ising model converge to SLE$_{16/3}$, and those of the Ising model to SLE$_{3}$ \cite{chelkak-duminil-hongler-kemppainen-smirnov14,duminil13}.
The Hausdorff dimension of SLE curves is known \cite{rohde-schramm05,beffara08}; one application is computing the dimension of the frontier of planar Brownian motion \cite{lawler-schramm-werner01}.
In order to fully capitalize on this dimensional understanding, a key goal has been to verify convergence to SLE in the so-called natural time-parameterization: see \cite{holden-li-sun?} and references therein. 

\item In dynamical critical percolation on the triangular lattice, the set of times at which there exists an infinite cluster almost surely has Hausdorff dimension $\frac{31}{36}$~\cite{garban-pete-schramm10}.
As in the case of Brownian motion, this set has a local time interpretation \cite{hammond-pete-schramm15}.

\end{itemize}

A common purpose served by each of these examples is to give meaning to ``sampling" from fractal sets defined by exceptional behavior.
That is, not only are there salient fractal properties---which capture key statistics of the prelimiting object---but we are also equipped with a measure on the relevant fractal set that allows for a meaningful investigation of what a ``typical" exceptional instance looks like.
This theme plays out in the present work as well, as we will see in the next section.

\subsection{Main results: Hausdorff dimensions and measure description of exceptional sets} \label{main_results}
Our particular interest in endpoints of disjoint geodesics is very much aligned with the study of coalescence of geodesics
for both first and last passage percolation models (the literature is vast; see \cite[Chap.~4 and 5]{auffinger-damron-hanson17} for one starting point).
Somewhat similar to the Busemann functions employed in those settings, a novel object called the \textit{difference weight profile} was studied in \cite{basu-ganguly-hammond21}.
For simplicity, let us now fix the time horizon $[s,t] = [0,1]$ and write $G_{x,y} = G_{(x,0;y,1)}$ for sets of geodesics.
Given $x_1<x_2$, the difference weight profile is the random map $y \mapsto \ZZ_{x_1,x_2}(y) \coloneqq \LL(x_2,0;y,1) - \LL(x_1,0;y,1)$, an almost-surely continuous, non-decreasing function on the real line; see Figure~\ref{diff_wt}.
A striking feature, which is suggested by simulation of the prelimit (see Figure~\ref{simulation}), is that $\ZZ_{x_1,x_2}$ is locally constant almost everywhere in the sense of Lebesgue.
Indeed, the following theorem was shown in \cite{basu-ganguly-hammond21}.
Recall that $y\in \R$ is a point of local variation for a function $f : \R\to\R$ if there exists no open interval containing $y$ on which $f$ is constant.

\begin{figure}
\centering
\includegraphics[trim=0.9in 0.7in 0.5in 0.5in, clip, width=0.55\textwidth]{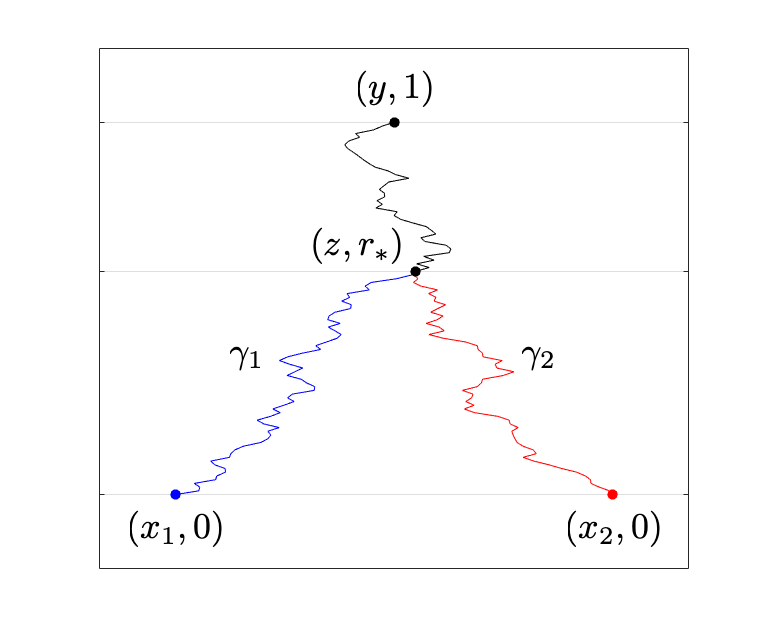}
\caption{Determining the value of $\ZZ_{x_1,x_2}(y)$.
Assuming geodesics exist, one can (using Lemma~\ref{new_geodesics}) always find $\gamma_1\in G_{x_1,y}$ and $\gamma_2\in G_{x_2,y}$ such that $\gamma_1(r)=\gamma_2(r)$ for all $r$ above some $r_*\in(0,1]$, and $\gamma_1(r)<\gamma_2(r)$ for all $r$ below $r_*$.
If $\gamma_1(r_*)=\gamma_2(r_*)=z$, then $\LL(\gamma_1) = \LL(x_1,0;z,r_*) + \LL(z,r_*;y,1)$ while
$\LL(\gamma_2) = \LL(x_2,0;z,r_*) + \LL(z,r_*;y,1)$. 
Therefore, $\ZZ_{x_1,x_2}(y) = \LL(\gamma_2)-\LL(\gamma_1)=\LL(x_2,0;z,r_*) - \LL(x_1,0;z,r_*)$.}
\label{diff_wt}
\end{figure}

\begin{figure}
\centering
\includegraphics[trim=0.25in 0.2in 0.25in 0in, clip, width=0.6\textwidth]{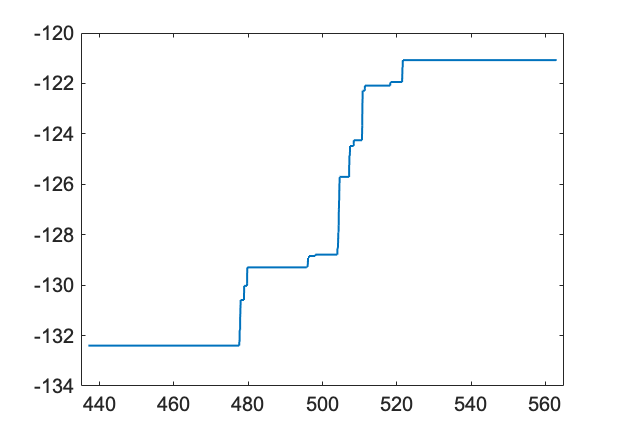}
\caption{A Brownian LPP simulation by Junou Cui, Zoe Edelson, Bijan Fard, and the first author.
With $n = 500$ and space discretized at intervals of size $0.01$, the unscaled difference $M(n^{2/3},0;y,n)-M(-n^{2/3},0;y,n)$ is plotted on the vertical axis against the unscaled location $y$ on the horizontal axis.
Once rescaled according to \eqref{general_process}, this corresponds to the picture in Figure~\ref{diff_wt} with $x_1=-\frac{1}{2}$, $x_2 = \frac{1}{2}$, and $y$ varying between $-\frac{1}{2}$ and $\frac{1}{2}$.
}
\label{simulation}
\end{figure}

\begin{theirthm}
\textup{\cite[Thm.~1.1]{basu-ganguly-hammond21}}
For any fixed $x_1<x_2$, the set of local variation for $\ZZ_{x_1,x_2}$ almost surely has Hausdorff dimension $\frac{1}{2}$.
\end{theirthm}

The proof of the upper bound for this result proceeded, at least heuristically, by examining points of coalescence between geodesics emanating from $(x_1,0)$ and $(x_2,0)$, and terminating at a common point $(y,1)$.
It was suggested that the set of local variation for $\ZZ_{x_1,x_2}$ is a subset of those $y$ for which the coalescence point is $(y,1)$ itself.
In other words, the supposed superset is the set $\DD_{x_1,x_2}\subset\R$ from Theorem~\ref{thm_1}, which consists of locations $y\in\R$ for which there exist two geodesics to $(y,1)$, one originating from $(x_1,0)$ and the other from $(x_2,0)$, that are disjoint except at the terminal point $(y,1)$.
In symbols, we have
\eeq{ \label{D_def}
\DD_{x_1,x_2} \coloneqq \Big\{y \in \R : \exists\, \gamma_1 \in G_{x_1,y}, \gamma_2 \in G_{x_2,y} \text{ such that } \gamma_1(r)<\gamma_2(r) \text{ for all $r\in(0,1)$}\Big\},
} 
where the inequality $\gamma_1(r)<\gamma_2(r)$ comes from the ordering of $x_1<x_2$.
(By planarity, $\gamma_1(r)>\gamma_2(r)$ only occurs if $\gamma_1(r')=\gamma_2(r')$ for some $r'<r$.)
See Figure~\ref{yes_D1} for an illustration. 
Our first main result shows that the heuristic from \cite{basu-ganguly-hammond21} turns out to be correct---moreover, the two sets are equal---and therefore the exceptional set $\DD_{x_1,x_2}$ has Hausdorff dimension one-half.

\begin{thm} \label{main_thm}
For any fixed $x_1<x_2$, the following statements hold almost surely.
\begin{enumerate}[label=\textup{(\alph*)}]

\item \label{main_thm_a} 
The set of local variation for $\ZZ_{x_1,x_2}$ is equal to $\DD_{x_1,x_2}$. 
\item \label{main_thm_b}
The Hausdorff dimension of $\DD_{x_1,x_2}$ is equal to $\frac{1}{2}$.

\end{enumerate}
\end{thm}

Since $\ZZ_{x_1,x_2}(y)$ is non-decreasing in $y$ by \cite[Thm.~1.1(1)]{basu-ganguly-hammond21}---a consequence of planarity---another perspective is that $\ZZ_{x_1,x_2}$ is the distribution function of a random measure $\mu_{x_1,x_2}$ on the real line, in the sense that
\begin{subequations}
\label{mu_var_def}
\eeq{ \label{mu_1var_def}
\mu_{x_1,x_2}([y_1,y_2])
&= \ZZ_{x_1,x_2}(y_2) - \ZZ_{x_1,x_2}(y_1) \\
&= \LL(x_2,0;y_2,1)+\LL(x_1,0;y_1,1)-\LL(x_1,0;y_2,1)-\LL(x_2,0;y_1,1).
}
In this language, the result of \cite{basu-ganguly-hammond21} is that the support of $\mu_{x_1,x_2}$ is a random perfect, nowhere dense set of Hausdorff dimension one-half, and Theorem~\hyperref[main_thm_a]{\ref*{main_thm}\ref*{main_thm_a}} says this set is equal to $\DD_{x_1,x_2}$.

A slightly more general perspective is that the final expression in \eqref{mu_1var_def} defines a measure $\mu$ on $\R^2$, namely
\eeq{ \label{mu_2var_def}
\mu([x_1,x_2]\times[y_1,y_2]) = \mu_{x_1,x_2}([y_1,y_2]).
}
\end{subequations}
In other words, by jointly varying the initial location $x$ and the terminal location $y$, one obtains a measure-theoretic encoding of $\LL$ on the time horizon $[0,1]$.
Then one can ask if the statements from before regarding $\mu_{x_1,x_2}$ have analogues for $\mu$.
Indeed, by modifications to the proof in \cite{basu-ganguly-hammond21}, 
one can show---as we do in Propositions~\ref{supp_thm} and \ref{upper_bd_thm}---that the support of $\mu$ also has Hausdorff dimension one-half.
Furthermore, in place of $\DD_{x_1,x_2}$, the related exceptional set is the one from Theorem~\ref{thm_2}:
\eeq{ \label{D2_def}
\DD \coloneqq \Big\{(x,y) \in \R^2 : \exists\, \gamma_1,\gamma_2 \in G_{x,y} \text{ such that } \gamma_1(r)<\gamma_2(r) \text{ for all $r\in(0,1)$}\Big\}.
}
In words, $\DD$ is the set of all pairs $(x,y)$ admitting two geodesics between $(x,0)$ and $(y,1)$ which coincide only at the endpoints; see Figure~\ref{yes_D2}.
We can now state our second main result.

\begin{thm} \label{main_thm_2}
The following statements hold almost surely.
\begin{enumerate}[label=\textup{(\alph*)}]

\item \label{main_thm_2_a}
The support of $\mu$ is equal to $\DD$.

\item \label{main_thm_2_b}
The Hausdorff dimension of $\DD$ is equal to $\frac{1}{2}$.

\end{enumerate}
\end{thm}

On a technical note, we mention that \eqref{mu_1var_def} and \eqref{mu_2var_def} prescribe well-defined measures given the almost sure event $\PP$; see Definition~\hyperref[proper_def_5]{\ref*{proper_def}\ref*{proper_def_5}}.
Off of $\PP$, one can simply take $\mu_{x_1,x_2}$ and $\mu$ to be zero measures.
Let us henceforth write $\Supp(\cdot)$ for the support of a measure.

\begin{remark}
\label{all_times_remark}
Upon rescaling and shifting according to \eqref{KPZ_scaling}, 
Theorems~\ref{main_thm} and \ref{main_thm_2} are seen to hold when the time horizon $[0,1]$ is replaced by any $[s,t]$, $s<t$.
Nevertheless, the almost sure events on which the statements hold could depend on $s$ and $t$, suggesting the possibility that the statements are false for some random times.
In Proposition \ref{containment_thm}, however, we rule out this possibility for Theorems~\hyperref[main_thm_a]{\ref*{main_thm}\ref*{main_thm_a}} and \hyperref[main_thm_2_a]{\ref*{main_thm_2}\ref*{main_thm_2_a}}.
We expect that Theorems~\hyperref[main_thm_b]{\ref*{main_thm}\ref*{main_thm_b}} and \hyperref[main_thm_2_b]{\ref*{main_thm_2}\ref*{main_thm_2_b}} will also be true simultaneously on all time horizons, although proving so will require certain estimates that go beyond the scope of this paper.
\end{remark}

\begin{remark} \label{nothing_else_needed}
It might seem natural to also consider the case when only the terminal location $y$ is fixed, while the starting location(s) is allowed to vary.
This perspective, however, is an ineffective way of studying the fractal geometry of the directed landscape. 
Indeed, it is possible to show that for any fixed $y\in\R$, almost surely $y\notin\DD_{x_1,x_2}$ for every $x_1\leq x_2$. 
Moreover, for the random exceptional $y$ belonging to some $\DD_{x_1,x_2}$, one trivially has $y\in\DD_{x_1',x_2'}$ for all $x_1'\leq x_1$ and $x_2'\geq x_2$, due to Lemma~\ref{reordering_lemma}.
From these observations, it is clear that a more refined description of the exceptional endpoints is obtained only by either fixing $x_1<x_2$ and varying $y$, or imposing $x_1=x_2=x$ and varying $(x,y)$.

Varying the time coordinates, on the other hand, raises a very interesting question unaddressed by this manuscript.
Namely, if we fix $x_1,x_2,y$ and consider the geodesics between $(x_1,s)$, $(x_2,s)$ and $(y,t)$, then can one describe the exceptional set of $s$ and $t$ for which the geodesics are disjoint?
Some of our methods can be adapted to this setting, but the situation is complicated in part by the absence of a monotonicity property (recall that $\ZZ_{x_1,x_2}(y)$ is non-decreasing in $y$).
\end{remark}

Interestingly, a measure similar to $\mu$ was studied in \cite{janjigian-rassoul20} in the context of planar LPP and positive-temperature directed polymers, which are fully discrete models.
That measure is also related via its support to exceptional disjointness of semi-infinite geodesics \cite{janjigian-rassoul-seppalainen?}, and Theorem~\hyperref[main_thm_2_a]{\ref*{main_thm_2}\ref*{main_thm_2_a}} bears a striking resemblance to \cite[Thm.~3.1]{janjigian-rassoul-seppalainen?}. 
Furthermore, one can consider similar objects in the prelimiting Brownian LPP.
Using classical facts of Brownian motion, it is possible to show that for any $i\in\Z$, the set of $x\in\R$ for which Brownian LPP admits disjoint semi-infinite geodesics starting at $(x,i)$ and having the same asymptotic direction, has Hausdorff dimension one-half \cite{seppalainen-sorensen?}.
This exceptional set of initial points, however, is dense and thus of a different nature than the sets we consider in Theorems~\ref{main_thm} and \ref{main_thm_2}.

\subsection{Comments on the proofs, and key inputs} \label{sketch_section}

The central issue and main novelty of this paper is proving the equalities claimed in Theorems~\hyperref[main_thm_a]{\ref*{main_thm}\ref*{main_thm_a}} and \hyperref[main_thm_2_a]{\ref*{main_thm_2}\ref*{main_thm_2_a}}.
We can summarize the argument for $\Supp(\mu)\subset\DD$ 
in the following broad strokes (the argument for the corresponding containment in \hyperref[main_thm_a]{\ref*{main_thm}\ref*{main_thm_a}} is similar):
\begin{itemize}
\item If $(x,y)\notin\DD$, then the leftmost geodesic $\gamma^\lt$ making the journey $(x,0)\to(y,1)$ must have non-trivial intersection with the rightmost geodesic.
An example is shown in Figure~\ref{pre_2var}. 
\item We prove in Lemma~\ref{limiting_geodesics} the general fact that if $y_j^\lt \nearrow y$ and $x_j^\lt\nearrow x$, then there is a corresponding sequence of geodesics $\gamma_j^\lt$ traveling $(x_j,0)\to(y_j,1)$ that converge uniformly to $\gamma^\lt$; similarly for $y_j^\rt\searrow y$, $x_j^\rt\searrow x$, and $\gamma^\rt$.
Therefore, by choosing $x^\lt < x < x^\rt$ and $y^\lt<y<y^\rt$ sufficiently close, we can find $\wt\gamma^\lt$ and $\wt\gamma^\rt$ approximating $\gamma^\lt$ and $\gamma^\rt$ to any desired precision.
\item  The critical observation, stated as Theorem~\ref{no_arbitrary_closeness}, is that if the endpoints of two geodesics belong to a common compact set, then the geodesics cannot approximate each other arbitrarily well without intersecting.
Therefore, $(x^\lt,x^\rt)$ and $(y^\lt,y^\rt)$ can be chosen so that  $\wt\gamma^\lt$ intersects $\gamma^\lt$, and $\wt\gamma^\rt$ intersects $\gamma^\rt$.
\item By forcing this intersection to occur at a suitably chosen location based on the non-trivial intersection of $\gamma^\lt$ and $\gamma^\rt$ (see Figure~\ref{post_2var}), we deduce that $\wt\gamma^\lt$ and $\wt\gamma^\rt$ do themselves intersect.
\item The proof is then completed by appealing to the observation---originally made in \cite{basu-ganguly-hammond21} and stated locally as Lemma~\ref{constant_lemma}---that $\ZZ_{x^\lt,x^\rt}$ is constant on $[y^\lt,y^\rt]$ whenever there are intersecting geodesics $\wt\gamma^\lt$ and $\wt\gamma^\rt$ making the respective journeys $(x^\lt,0)\to(y^\lt,1)$ and $(x^\rt,0)\to(y^\rt,1)$.
The reason for this constancy is outlined in Figure~\ref{diff_wt_constant}.
\end{itemize}
Meanwhile, the argument for $\Supp(\mu)\supset\DD$ is simpler (and the corresponding argument for~\hyperref[main_thm_a]{\ref*{main_thm}\ref*{main_thm_a}} is again similar):
\begin{itemize}
\item If $(x,y) \notin \Supp(\mu)$, then there are $x^\lt<x<x^\rt$ and $y^\lt<y<y^\rt$ such that $\ZZ_{x^\lt,x^\rt}(y^\lt) = \ZZ_{x^\lt,x^\rt}(y^\rt)$.
\item Any two geodesics making the journeys $(x^\lt,0)\to(y^\rt,1)$ and $(x^\rt,0)\to(y^\lt,1)$ necessarily intersect at some point $(z,r)$.
As a consequence of the first bullet point, we can produce by concatenation two geodesics traveling $(x^\lt,0)\to(y^\lt,1)$ and $(x^\rt,0)\to(y^\rt,1)$ and each passing through $(z,r)$; see \eqref{crossing_step} and the text that follows.
\item By the planar ordering of geodesics, we conclude that any geodesic traveling $(x,0)\to(y,1)$ must also pass through $(z,r)$.
In particular, $(x,y)$ does not belong to the exceptional set $\DD$.
\end{itemize}
It is worth emphasizing that the arguments summarized thus far work simultaneously on all time horizons.
That is, wherever an almost sure statement is needed, the probability-one set on which it holds does not depend on the choice of $[0,1]$ as the interval of interest (see Remark \ref{all_times_remark}).

Given that $\Supp(\mu_{x_1,x_2})$ and $\Supp(\mu)$ each have Hausdorff dimension one-half, Theorems~\hyperref[main_thm_b]{\ref*{main_thm}\ref*{main_thm_b}} and \hyperref[main_thm_2_b]{\ref*{main_thm_2}\ref*{main_thm_2_b}} are immediate from 
\hyperref[main_thm_a]{\ref*{main_thm}\ref*{main_thm_a}} and \hyperref[main_thm_2_a]{\ref*{main_thm_2}\ref*{main_thm_2_a}}.
%
Nevertheless, in recognition of the fact that the results appearing in \cite{basu-ganguly-hammond21} are only for $\mu_{x_1,x_2}$, we do separately check the required statements for $\mu$.
Namely, the dimension lower bound argument carried out in \cite{basu-ganguly-hammond21}---which used the local Gaussianity of weight profiles \cite{hammond19II}, which in turn is a consequence of the Brownian Gibbs property \cite{corwin-hammond14,hammond??} enjoyed by the parabolic Airy line ensemble and its prelimit---could be replicated here, but we instead present in Section~\ref{supp_thm_proof} a short proof that the Hausdorff dimension of $\Supp(\mu)$ is at least that of $\Supp(\mu_{x_1,x_2})$.
This statement is simply a consequence of the fact that $\mu_{x_1,x_2}$ is a projection of $\mu$.

This indirect lower bound provides sufficient information to conclude Theorem~\hyperref[main_thm_2_b]{\ref*{main_thm_2}\ref*{main_thm_2_b}} from~\hyperref[main_thm_2_a]{\ref*{main_thm_2}\ref*{main_thm_2_a}}, once  
we independently show the matching upper bound for the Hausdorff dimension of $\Supp(\mu)$; this is done in Section \ref{upper_bound_proof} in a manner similar to the approach of \cite{basu-ganguly-hammond21}.
The central idea is captured in Figure~\ref{upper_bound_fig} and is briefly described as follows:
\begin{itemize}
\item Let us restrict our attention to a bounded square $[-R,R]^2$.
Suppose $(x.y) \in \DD \cap [-R,R]^2$, and let $x-\eps < x_1 < x < x_2 < x+\eps$ and $y-\eps<y_1<y<y_2<y+\eps$.
Because $(x,y) \in \DD$, planarity guarantees that any geodesic $\gamma_1$ traveling $(x_1,0)\to (y_1,1)$ is disjoint from any geodesic $\gamma_2$ traveling $(x_2,0)\to(y_2,1)$, as shown in Figure~\ref{upper_bound_fig}.
Moreover, this holds for arbitrarily small $\eps>0$.
\item We have thus identified two small intervals $I = (x-\eps,x+\eps)$, $J=(y-\eps,y+\eps)$ that admit two disjoint geodesics starting in $I\times\{0\}$ and ending in $J\times\{1\}$.
A required input is that as $\eps\to0$, the likelihood of this event for given $x$ and $y$ is, to leading order, bounded from above by $\eps^{3/2}$.
This fact is stated as Corollary~\ref{disjoint_2geo_rarity}.
\item Now we cover $[-R,R]^2$ with order $\eps^{-2}$ many pairs of intervals $(I,J)$ of radius $\eps$.
The expected number of pairs of these intervals with the above property is now seen to be at most $\eps^{-2}\cdot\eps^{3/2} =\eps^{-1/2}$.
Therefore, the box-counting dimension of $\DD \cap [-R,R]^2$, which always serves as an upper bound for the Hausdorff dimension, is at most $\frac{1}{2}$.
\end{itemize}



As a slightly finer-scale comment,
this paper does marry arguments for the limiting model with inputs previously known or verified only in the prelimiting one, namely Brownian LPP.
As such, there are a number of important facts requiring extension to the limiting setting.
Therefore, in Sections~\ref{inputs_2} and \ref{inputs_1} below, we state several other new results concerning the directed landscape and the polymer fixed point.
These inputs, which are proved in Section~\ref{disjoint_geo_rarity_proofs}, may be of independent interest as tools in or inspiration for future works.
Section~\ref{preliminaries} contains several more input facts; these are straightforward statements about geodesics that will not be highlighted here.

\subsubsection{Convergence of polymers} \label{inputs_2}

Owing to the novelty of the construction in \cite{dauvergne-ortmann-virag?}, our first set of inputs work to build one bridge (of certainly many more) from the directed landscape to previously studied objects. 
In particular, we focus on Theorem~\ref{geodesics_converge}, which addresses the convergence of $n$-geodesics in Brownian LPP to their continuous counterparts in the directed landscape.
While these particular prelimiting objects, defined via \eqref{n_geo_eq}, enable the treatment of geodesics as functions, they are somewhat less natural from the view of geodesics as planar paths.
For this latter perspective, it is desirable to simply consider the polymers from Definition~\ref{polymer_def}, i.e., the piecewise linear paths obtained by applying the scaling map $R_n$ to the Brownian LPP staircase paths (recall Figure~\ref{blpp_fig}).
Indeed, the map $R_n$ is intrinsic to the coupling of Brownian LPP and the directed landscape in the first place.

On a practical level, many estimates for Brownian LPP are stated for $n$-polymers rather than $n$-geodesics; for instance, see \cite{hammond??,hammond19II,hammond19I,hammond20}.
In order to leverage these results to prove statements in the limiting setting, it will be useful to know that $n$-polymers share the same limit as $n$-geodesics.
For example, the upcoming Theorem~\ref{disjoint_geo_rarity} asks for probabilistic control on disjointness for geodesics.
Unfortunately, disjointness of $n$-polymers---for which we have the corresponding information---is not equivalent to disjointness of $n$-geodesics. 
This is because the time parameterization \eqref{parameterized_time} used in \eqref{n_geo_eq} depends on the spatial coordinates $x$ and $y$, meaning that distinct $u_1 = (x_1,s;y_1,t), u_2=(x_2,s;y,t)\in\R^4_\uparrow$ can admit $n$-geodesics $\Gamma_{(n,u_1)}^{(\vphi)}$, $\Gamma_{(n,u_2)}^{(\phi)}$ that intersect, 
\eq{
\vphi(L_{n,u_1}(r)) = \phi(L_{n,u_2}(r)) \quad \text{for some $r\in[s,t]$,}
} 
even if the corresponding $n$-polymers $R_n(\vphi)$, $R_n(\phi)$ do not:
\eq{
\vphi(z)\neq\phi(z) \quad \text{for all $z\in[sn+2n^{2/3}x_1,tn+2n^{2/3}y_1)\cap[sn+2n^{2/3}x_2,tn+2n^{2/3}y_2)$}.
}
In this and other circumstances, the next theorem and corollary can help translate between the two prelimiting objects.

Let $u = (x,s;y,t)\in\R^4_\uparrow$ and $\vphi : [sn+2n^{2/3}x,tn+2n^{2/3}y)\to\llbrack \lfloor sn\rfloor,\lfloor tn\rfloor\rrbrack$ be the right-continuous, non-decreasing function defined by $\vphi(z) = k$ if and only if $z\in[z_k,z_{k+1})$.
Recall the definitions of $\Gamma_{n,u}^{(\vphi)}$ and $L_{n,u}$ from \eqref{n_geo_eq} and \eqref{parameterized_time}.
For each $k\in\llbrack \lfloor sn\rfloor,\lfloor tn\rfloor\rrbrack$, let $r_k$ be the unique value in $[s,t]$ such that $L_{n,u}(r_k) = z_k$.

\begin{defn}
Denote by $\wt R_{n,u}(\vphi) \subset \R^2$ the planar path determined by the graph of $\Gamma_{n,u}^{(\vphi)}$; that is,
\eq{
\wt R_{n,u}(\vphi) \coloneqq \underbrace{\{(z,r) : r \in [s,t], z = \Gamma_{n,u}^{(\vphi)}(r)\}}_{\text{oblique segments}} \cup \underbrace{\bigcup_{k=\lfloor sn\rfloor+1}^{\lfloor tn\rfloor} \big\{(z,r_k) : z \in [\Gamma_{n,u}(r_k),\Gamma_{n,u}(r_k-)]\big\}}_{\text{horizontal segments}},
}
where $\Gamma_{n,u}(r_k-) \coloneqq \lim_{r\nearrow r_k} \Gamma_{n,u}(r)$ if $r_k > s$, and $\Gamma_{n,u}(s-) \coloneqq \Gamma_{n,u}(s)$.
See Figure~\ref{geodesic_path}.
\end{defn}

The theorem below says that the two planar paths $R_n(\vphi)$ and $\wt R_{n,u}(\vphi)$ are asymptotically equal if $\vphi$ is a maximizer in \eqref{blpp_def_2}.
In particular, $n$-polymers share the same limit as $n$-geodesics, but use the language of sets rather than of functions.
Recall that the \textit{Hausdorff distance} between two nonempty subsets $\XX,\YY$ of a metric space with metric $\tau$ is
\eq{
\dist_H(\XX,\YY) \coloneqq \max\Big\{\sup_{x\in\XX}\inf_{y\in\YY}\tau(x,y),\sup_{y\in\YY}\inf_{x\in\XX}\tau(x,y)\Big\}.
}
In stating the result, we return to the setting of Theorem~\ref{geodesics_converge}.
Recall that $\{\Gamma_{n,u}^{(\vphi_n)}\}_{n\geq1}$ was assumed to be a sequence of $n$-geodesics, known to converge uniformly to $\gamma_u$. 
Since $\wt R_{n,u}(\vphi_n)$ is simply the graph of $\Gamma_{n,u}^{(\vphi_n)}$, some definition chasing will show that $\wt R_{n,u}(\vphi_n)$ converges to the graph of $\gamma_u$.  
It is by this logic, carried out in Section~\ref{polymer_convergence}, that \eqref{polymer_geodesic} will follow from \eqref{coupled_polymers}.

\begin{thm} \label{polymers_converge}
In the coupling of Theorem~\ref{blpp_airy}, the following holds for any $u\in\R^4_\uparrow$. 
On the almost sure event $\PP \cap \{|G_u|=1\}$, if $\gamma_u$ is the unique element of $G_u$, and $R_n(\vphi_n) \in P_{n,u}$ for each $n$, then

\eeq{ \label{coupled_polymers}
\limsup_{n\to\infty} \frac{\dist_H(R_n(\vphi_n),\wt R_{n,u}(\vphi_n))}{n^{-1/3}} < \infty.
}
In particular, upon defining $\Gr(\gamma_u) \coloneqq \{(\gamma_u(r),r) : r\in[s,t]\}$, we have
\eeq{ \label{polymer_geodesic}
\lim_{n\to\infty} \dist_H(R_n(\vphi_n),\Gr(\gamma_u)) = 0.
}
\end{thm}

From this result, we will obtain the following essential ingredient to the proof of Theorem~\ref{disjoint_geo_rarity}.
While the statement does not immediately follow from Theorem~\ref{polymers_converge}, it will be an easy consequence of the argument we give for \eqref{polymer_geodesic}.
A short proof is included in Section~\ref{polymer_convergence}.

\begin{cor} \label{eventual_disjointness}
Let $u_1 = (x_1,s;y_1,t)$ and $u_2 = (x_2,s;y_2,t)$ with $x_1<x_2$ and $y_1<y_2$.
For each $n$, choose any $R_n(\vphi_n) \in P_{n,u_1}$ and $R_n(\phi_n) \in P_{n,u_2}$.
On the event $\PP\cap\{|G_{u_1}|=1\} \cap \{|G_{u_2}|=1\}$, if the unique geodesics $\gamma_1\in G_{u_1}$ and $\gamma_2\in G_{u_2}$ are disjoint, then $R_n(\vphi_n) \cap R_n(\phi_n) = \varnothing$ for all $n$ sufficiently large.
\end{cor}

\subsubsection{Estimates for disjoint collections of geodesics} \label{inputs_1}

We will say that two paths $\gamma_1,\gamma_2 : [s,t]\to\R$ are \textit{disjoint} if $\gamma_1(r) \neq \gamma_2(r)$ for all $r\in[s,t]$.
While the geodesics on which we will ultimately focus, namely those defining $\DD_{x_1,x_2}$ and $\DD$, are not strictly speaking disjoint---they coincide at the endpoint(s)---our arguments will exploit their influence on the disjointness of other nearby geodesics.
Therefore, our second series of inputs concerns the rarity of certain events involving disjointness.
Theorem~\ref{disjoint_geo_rarity} and Corollary~\ref{disjoint_2geo_rarity}, in particular, are very much in the aim of translating known results about the prelimiting model into ones about the limiting model.

For subsets $A,B,C\subset\R$ and times $s<t$, let $\MDG_{s,t}^C(A,B)$ denote the maximum size of a collection of disjoint geodesics whose endpoints lie in $(A\cap C)\times\{s\}$ and $(B\cap C)\times\{t\}$.
(Here we mean disjoint even at the endpoints.)
When $C$ is countable (by which we mean having a cardinality that is either finite or countably infinite), the measurability of this random variable is proved in Proposition~\ref{disjoint_measurable}.
The following tail bound is analogous to, and indeed proved from, \cite[Thm.~1.1]{hammond20}.
The proof appears in Section~\ref{rarity_section}.
 


\begin{thm}\label{disjoint_geo_rarity}
There exists a positive constant $G$ such that the following holds for all countable $C\subset\R$.
For any $\eps>0$, integer $k\geq2$, and $u = (x,s;y,t)\in\R^4_\uparrow$ satisfying
\eeq{ \label{disjoint_geo_rarity_conditions}
\frac{\eps}{(t-s)^{2/3}} \leq G^{-4k^2}, \quad  \frac{|x-y|}{(t-s)^{2/3}} \leq \Big(\frac{\eps}{(t-s)^{2/3}}\Big)^{-1/2}\Big(\log \frac{(t-s)^{2/3}}{\eps}\Big)^{-2/3}G^{-k},
}
we have
\eeq{ \label{disjoint_geo_rarity_ineq}
&\P\Big(\MDG_{s,t}^C([x-\eps,x+\eps],[y-\eps,y+\eps])\geq k\Big) \\
&\leq G^{k^3}\exp\Big\{G^k\Big(\log\frac{(t-s)^{2/3}}{\eps}\Big)^{5/6}\Big\}\Big(\frac{\eps}{(t-s)^{2/3}}\Big)^{(k^2-1)/2}.
}
\end{thm}

The restriction that $C$ be countable arises from the possibility that one of the mutually disjoint geodesics is associated to an exceptional $u\in\R^4_\uparrow$ for which $|G_u|\geq2$.
In this scenario, Corollary~\ref{eventual_disjointness} no longer guarantees that the concerned collection of geodesics can be realized from a disjoint collection of polymers in the prelimit, thereby rendering the estimate from \cite{hammond20} inapplicable.
When $C$ is countable, however, this hurdle can be avoided by simply assuming the almost sure event in which all geodesics whose spatial endpoints lie in $C$ are unique.

Notwithstanding these technical impediments, we anticipate that Theorem~\ref{disjoint_geo_rarity} is true with $C=\R$.
Indeed, the $k=2$ case admits a simple argument that will allow us to bootstrap to the following statement, proved at the end of Section \ref{rarity_section}.

\begin{cor} \label{disjoint_2geo_rarity}
Let $G$ be the constant from Theorem~\ref{disjoint_geo_rarity}.
For any $\eps>0$ and $u = (x,s;y,t)\in\R^4_\uparrow$ satisfying
\eeq{ \label{disjoint_2geo_rarity_conditions}
\frac{\eps}{(t-s)^{2/3}} \leq G^{-16}, \quad  \frac{|x-y|}{(t-s)^{2/3}} \leq \Big(\frac{\eps}{(t-s)^{2/3}}\Big)^{-1/2}\Big(\log \frac{(t-s)^{2/3}}{\eps}\Big)^{-2/3}G^{-2},
}
we have
\eeq{ \label{disjoint_2geo_rarity_ineq}
&\P\Big(\MDG_{s,t}^\R((x-\eps,x+\eps),(y-\eps,y+\eps))\geq 2\Big) \\
&\leq G^{8}\exp\Big\{G^2\Big(\log\frac{(t-s)^{2/3}}{\eps}\Big)^{5/6}\Big\}\Big(\frac{\eps}{(t-s)^{2/3}}\Big)^{3/2}.
}
\end{cor}

Our arguments for the upper bounds in Theorems~\hyperref[main_thm_b]{\ref*{main_thm}\ref*{main_thm_b}} and \hyperref[main_thm_2_b]{\ref*{main_thm_2}\ref*{main_thm_2_b}} will use Corollary~\ref{disjoint_2geo_rarity} directly.
Meanwhile, Theorems~\hyperref[main_thm_a]{\ref*{main_thm}\ref*{main_thm_a}} and \hyperref[main_thm_2_a]{\ref*{main_thm_2}\ref*{main_thm_2_a}} will require the following application of Corollary~\ref{disjoint_2geo_rarity} regarding geodesics that not only start and end nearby one another, but also remain close at all intermediate times.
This result is established in Section \ref{no_arbitrary_closeness_proof}.


\begin{thm} \label{no_arbitrary_closeness}
On the event $\PP$, for any compact $K\subset\R^4_\uparrow$, there is a random $\eps>0$ such that the following is true.
If $u_1 = (x,s;y,t),u_2 = (z,s;w,t)\in K$ admit geodesics $\gamma_1\in G_{u_1}$, $\gamma_2\in G_{u_2}$ satisfying $|\gamma_1(r)-\gamma_2(r)|<\eps$ for all $r\in[s,t]$, then $\gamma_1$ and $\gamma_2$ are not disjoint.
\end{thm}



\subsection{Organization}
For the reader's convenience, we list below the contents of the remaining sections.
All arguments are presented in logical order.
\begin{itemize}
\item Section~\ref{preliminaries} establishes basic properties of the directed landscape $\LL$ and its geodesics which will be needed in all later proofs.
These address existence, uniqueness, and ordering of geodesics, as well as restriction and concatenation operations.
Key regularities of $\LL$ are recalled in Definition~\ref{proper_def}.
\item Section~\ref{disjoint_geo_rarity_proofs} begins by verifying in Theorem~\ref{polymers_converge} the asymptotic equivalence of $n$-geodesics (the objects considered in the fundamental Theorem~\ref{geodesics_converge}) to what we call $n$-polymers (for which \cite{hammond20} contains a prelimiting version of Theorem~\ref{disjoint_geo_rarity}).
These results are in service of Theorem~\ref{no_arbitrary_closeness}, which says that two geodesics in a compact set must either intersect or be well separated in the uniform norm.
\item By capitalizing on this last observation, Section~\ref{lower_bound} proves Theorems~\hyperref[main_thm_a]{\ref*{main_thm}\ref*{main_thm_a}} and \hyperref[main_thm_2_a]{\ref*{main_thm_2}\ref*{main_thm_2_a}}, i.e.~the supports of the random measures $\mu_{x_1,x_2}$ and $\mu$ from \eqref{mu_var_def} are exactly the exceptional sets $\DD_{x_1,x_2}$ and $\DD$ from Theorems~\ref{thm_1} and \ref{thm_2}.
\item Using this description of the exceptional sets, we finally prove Theorems~\ref{thm_1} and \ref{thm_2} in Section~\ref{upper_bound}, as sketched in Section~\ref{sketch_section}.
\end{itemize}

\subsection{Acknowledgments}
B\'alint Vir\'ag gave a seminar at the R\'enyi Institute in January 2019 after which he showed simulations of the measure $\mu$ from  \eqref{mu_2var_def}  and
indicated that the Hausdorff dimension of its support equals one-half.
The third author attended this talk and would like to thank B\'alint Vir\'ag for beneficial discussions in person and by email regarding the fractal geometry of various exceptional sets embedded in the directed landscape
and relations between the measure $\mu$ and the Airy sheet. 
In these discussions, B\'alint indicated an argument, due to him and Duncan Dauvergne, which may be used to prove that uniform convergence of geodesics in the directed landscape entails that all but finitely many intersect the limiting path.
This fact is captured here by Theorem~\ref{no_arbitrary_closeness}.
The authors also thank Riddhipratim Basu, Timo Sepp\"al\"ainen, Evan Sorensen, and Benedek Valk\'o for helpful discussions, and the referees for several useful suggestions.

\section{Preliminary facts concerning geodesics} \label{preliminaries}

In this section, we establish some basic facts about paths and geodesics in the directed landscape.

\subsection{New geodesics from old}
We begin with a lemma concerning subpaths and subgeodesics.

\begin{lemma} \label{subgeodesic_lemma}
Let $\gamma : [s,t] \to \R$ be a continuous path, and suppose we are given a partition $s=t_0<t_1<\cdots<t_k=t$.
The following statements hold.
\begin{enumerate}[label=\textup{(\alph*)}]
\item \label{subgeodesic_lemma_a}
We have the concatenation identity
\eeq{ \label{concatenation}
\LL(\gamma) = \sum_{i=1}^k \LL(\gamma\big|_{[t_{i-1},t_i]}).
}
\item \label{subgeodesic_lemma_b}
If $\gamma$ is a geodesic, then $\gamma\big|_{[t_{i-1},t_i]}$ is a geodesic for each $i=1,\dots,k$, and
\eeq{ \label{any_subdivision}
\LL(\gamma) = \sum_{i=1}^k \LL(\gamma(t_{i-1}),t_{i-1};\gamma(t_i),t_i).
}
\end{enumerate}
\end{lemma}

\begin{proof}
First we prove (a).
By induction, it suffices to prove the claim in the case $k=2$ with $s < r < t$.
Since any pairing of a partition of $[s,r]$ 
with a partition of $[r,t]$ 
induces a partition of $[s,t]$, 
it is clear that $\LL(\gamma) \leq \LL(\gamma\big|_{[s,r]}) + \LL(\gamma\big|_{[r,t]})$.
On the other hand, for any partition of $[s,t]$ not arising in this way (i.e.,~a sequence $s=t_0 < t_1 < \cdots < t_k=t$ such that $t_{j-1} < r < t_{j}$ for some $j$), we have
\eq{
\sum_{i=1}^k \LL(\gamma(t_{i-1}),t_{i-1};\gamma(t_i),t_i)
&\stackrel{\hspace{0.5ex}\mbox{\footnotesize\eqref{max_prop}}\hspace{0.5ex}}{\geq} 
\sum_{i=1}^{j-1}\LL(\gamma(t_{i-1}),t_{i-1};\gamma(t_i),t_i)\, +\, \LL(\gamma(t_{j-1}),t_{j-1};\gamma(r),r) \\
&\phantom{\stackref{length_def}{\geq}}+\LL(\gamma(r),r;\gamma(t_{j}),t_{j})+\sum_{i=j+1}^{k}\LL(\gamma(t_{i-1}),t_{i-1};\gamma(t_i),t_i)  \\
&\stackref{length_def}{\geq} \LL(\gamma\big|_{[s,r]}) + \LL(\gamma\big|_{[r,t]}).
}
Hence $\LL(\gamma) \geq \LL(\gamma\big|_{[s,r]}) + \LL(\gamma\big|_{[r,t]})$, which completes the proof of (a).

For (b), we can again appeal to induction and reduce to the case $k=2$.
If $\gamma$ is a geodesic, then
\eq{
\LL(\gamma\big|_{[s,r]}) + \LL(\gamma\big|_{[r,t]}) 
\stackref{concatenation}{=} \LL(\gamma) 
= \LL(\gamma(s),s;\gamma(t),t)
&\stackref{max_prop}{\geq} \LL(\gamma(s),s;\gamma(r),r) + \LL(\gamma(r),r;\gamma(t),t).
}
Since we always have
\eq{
\LL(\gamma\big|_{[s,r]}) \leq \LL(\gamma(s),s;\gamma(r),r) \quad \text{and} \quad \LL(\gamma\big|_{[r,t]}) \leq\LL(\gamma(r),r;\gamma(t),t),
}
the only possibility is that each of the two inequalities in the above display is achieved with equality.
 That is, $\gamma\big|_{[s,r]}$ and $\gamma\big|_{[r,t]}$ are geodesics, in which case \eqref{any_subdivision} follows from \eqref{concatenation}.
\end{proof}

For the arguments to come, it will be useful to have the following notation for concatenating paths.
If $\gamma_1 : [s,r'] \to \R$ and $\gamma_2 : [r',t] \to \R$ satisfy $\gamma_1(r') = \gamma_2(r')$, then $\gamma_1 \oplus \gamma_2 : [s,t]\to\R$ will denote the function defined by
\eq{
(\gamma_1 \oplus \gamma_2)(r) \coloneqq \begin{cases}
\gamma_1(r) &\text{if }r \in [s,r'], \\
\gamma_2(r) &\text{if }r \in (r',t].
\end{cases}
}
The following lemma says that if two geodesics intersect twice, then exchanging their segments between these intersections results in another geodesic.

\begin{lemma} \label{new_geodesics}
Suppose $\gamma_1 : [s_1,t_1] \to \R$ and $\gamma_2 : [s_2,t_2] \to \R$ are geodesics.
If $\gamma_1(r') = \gamma_2(r')$ and $\gamma_1(r'') = \gamma_2(r'')$ for some $r'<r''$ belonging to $[s_1,t_1]\cap[s_2,t_2]$, then each of the following paths is a geodesic:
\begin{enumerate}[label=\textup{(\roman*)}]

\item \label{new_geodesics_1}
$\gamma_1\big|_{[s_1,r']} \oplus \gamma_2\big|_{[r',r'']} \oplus \gamma_1\big|_{[r'',t_1]}$,

\item \label{new_geodesics_2}
$\gamma_1\big|_{[s_1,r']} \oplus \gamma_2\big|_{[r',r'']}$,

\item \label{new_geodesics_3} 
$\phantom{\gamma_1\big|_{[s_1,r']} \oplus }\: \: \gamma_2\big|_{[r',r'']} \oplus \gamma_1\big|_{[r'',t_1]}$.

\end{enumerate}
\end{lemma}

\begin{proof}
First notice that (ii) and (iii) follow from (i) by Lemma~\hyperref[subgeodesic_lemma_b]{\ref*{subgeodesic_lemma}\ref*{subgeodesic_lemma_b}}, and so we just prove (i).
Let us write $\gamma = \gamma_1\big|_{[s_1,r']} \oplus \gamma_2\big|_{[r',r'']} \oplus \gamma_1\big|_{[r'',t_1]}$.
We have 
\eq{
\LL(\gamma_1) 
&\stackrel{\hspace{4ex}\mbox{\footnotesize\eqref{any_subdivision}}\hspace{4ex}}{=} \LL(x,s_1;\gamma_1(r'),r') + \LL(\gamma_1(r'),r';\gamma_1(r''),r'') + \LL(\gamma_1(r''),r'';\gamma_1(t),t) \\
&\stackrel{\hspace{4ex}\phantom{\mbox{\footnotesize\eqref{any_subdivision}}}\hspace{4ex}}{=}\LL(x,s_1;\gamma_1(r'),r') + \LL(\gamma_2(r'),r';\gamma_2(r''),r'') + \LL(\gamma_1(r''),r'';\gamma_1(t),t) \\
&\stackrel{\mbox{\footnotesize{Lemma~\hyperref[subgeodesic_lemma_b]{\ref*{subgeodesic_lemma}\ref*{subgeodesic_lemma_b}}}}}{=}
\LL(\gamma_1\big|_{[s_1,r']}) + \LL(\gamma_2\big|_{[r',r'']}) + \LL(\gamma_1\big|_{[r'',t_1]})\stackrel{\mbox{\footnotesize\eqref{concatenation}}}{=} \LL(\gamma).
}
Since $\gamma_1$ is a geodesic with the same endpoints as $\gamma$, it follows from $\LL(\gamma_1) = \LL(\gamma)$ that $\gamma$ is a geodesic.
\end{proof}

\subsection{Typical and atypical properties} \label{typical_atypical}

In subsequent proofs, it will be important to know what is entailed in the almost sure event $\PP$ from Theorem~\ref{geodesics_converge}.

\begin{defn} \label{proper_def}
\cite[Sec.~13]{dauvergne-ortmann-virag?}
The function $\LL:\R^4_\uparrow\to\R$ is said to be a \textit{proper landscape}, and we say $\PP$ occurs, if the following conditions hold:
\begin{enumerate}[label=\textup{(\roman*)}]

\item \label{proper_def_1}
$\LL$ is continuous;

\item \label{proper_def_2}
for every $R>0$, there is a constant $c$ such that
\eq{
\Big|\LL(x,s;y,t) + \frac{(x-y)^2}{t-s}\Big| \leq c \quad \text{for all $(x,s;y,t)\in\R^4_\uparrow\cap[-R,R]^4$};
}

\item \label{proper_def_3}
for every $(x,s;y,t)\in\R^4_\uparrow$ and $r\in(s,t)$, the supremum in \eqref{max_prop} is achieved by some $z\in\R$;

\item \label{proper_def_4}
for every compact set $K\subset\R^4_\uparrow$, the values of $z\in\R$ achieving the supremum in \eqref{max_prop} are uniformly bounded among $(x,s;y,t)\in K$ and $r\in(s,t)$; and

\item \label{proper_def_5}
for every $x_1\leq x_2$, $y_1\leq y_2$, and $s<t$, we have
\eq{
\LL(x_2,s;y_2,t)+\LL(x_1,s;y_1,t)-\LL(x_1,s;y_2,t)-\LL(x_2,s;y_1,t) \geq 0.
}

\end{enumerate}
\end{defn}
While conditions (i)--(iv) are various quantifications of tightness, property (v) can be regarded as a deterministic fact about planar geodesic spaces, discussed in \cite[Lemma 9.1]{dauvergne-ortmann-virag?} and \cite[Thm.~1.1(1)]{basu-ganguly-hammond21}.
In particular, the measures $\mu_{x_1,x_2}$ and $\mu$ given by \eqref{mu_var_def} are well-defined because of (v).

\begin{remark} \label{proper_remark}
For later use, we note some facts about Definition \ref{proper_def}.
When $\PP$ occurs:
\begin{enumerate}[label=\textup{(\alph*)}]

\item \label{proper_remark_a}
\ref{proper_def_1} $\implies$ $\LL$ is bounded on any compact subset of $\R^4_\uparrow$;

\item \label{proper_remark_b}
 \ref{proper_def_2} $\implies$ $\LL(x,s;y,t)\to-\infty$ as $t\searrow s$, and for $\eps>0$, this divergence is uniform over $x,y,s\in[-R,R]$ such that $|x-y|\geq\eps$; and

\item \label{proper_remark_c}
\ref{proper_def_4} $\implies$ for any compact $K\subset\R^4_\uparrow$, there is a random constant $R>0$ such that
\eq{
u=(x,s;y,t)\in K,\, \gamma\in G_u \quad \implies \quad |\gamma(r)|\leq R \quad \text{for all $r\in[s,t]$}.
}
This is because Lemma~\hyperref[subgeodesic_lemma_b]{\ref*{subgeodesic_lemma}\ref*{subgeodesic_lemma_b}} implies that for any $\gamma\in G_{(x,s;y,t)}$, the value $z=\gamma(r)$ is a maximizer in \eqref{max_prop} for every $r\in(s,t)$.
See also Lemma~\ref{geodesic_construction}.

\end{enumerate}

\end{remark}

For paths $\gamma_1:[s,t]\to\R$ and $\gamma_2:[s,t]\to\R$, let us write $\gamma_1\leq\gamma_2$ if $\gamma_1(r)\leq\gamma_2(r)$ for all $r\in[s,t]$.

\begin{defn} \label{left_right_def}
For $u=(x,s;y,t)\in\R^4_\uparrow$, we say that $\gamma^\lt$ is the \textit{leftmost} geodesic in $G_u$ if $\gamma^\lt\leq\gamma$ for all $\gamma\in G_u$.
Similarly, $\gamma^\rt$ is the \textit{rightmost} geodesic in $G_u$ if $\gamma\leq\gamma^\rt$ for all $\gamma\in G_u$.
\end{defn}

Typically geodesics are unique, in which case the leftmost and rightmost geodesics are the same. 
It will be useful to record this and two other types of almost sure events concerning geodesics:

\begin{enumerate}

\item (Existence)
By \cite[Lemma 13.2]{dauvergne-ortmann-virag?}, the following event is a superset of $\PP$ and thus occurs with probability one:
\eeq{ \label{existence_event}
\EE \coloneqq \{G_u \text{ contains a leftmost and a rightmost geodesic for every }u\in\R^4_\uparrow\}.
}

\item (Uniqueness) For any fixed $u\in\R^4_\uparrow$, the measurability of the event $\{|G_u| = 1\}$ is argued in \cite[Sec.~13]{dauvergne-ortmann-virag?}.
Moreover, \cite[Thm.~12.1]{dauvergne-ortmann-virag?} 
gives $\P(|G_u| = 1)=1$. 

\item (Ordering) Consider the event
\begin{subequations}
\label{ordering_events}
\begin{linenomath}\postdisplaypenalty=0
\begin{align}
\label{global_ordering_event}
\OO \coloneqq \bigcap_{s<t}\bigcap_{x_1<x_2} \bigcap_{y_1<y_2} \{\text{for every $\gamma_1\in G_{(x_1,s;y_1,t)}, \gamma_2\in G_{(x_2,s;y_2,t)}$, we have $\gamma_1 \leq \gamma_2$}\}.
\end{align}
\end{linenomath}
That is, whenever $x_1<x_2$ and $y_1 < y_2$, the geodesics from $(x_1,s)$ to $(y_1,t)$ do not ``cross" those from $(x_2,s)$ to $(y_2,t)$.
We will soon check in Lemma~\ref{no_crosses} that $\OO$ is an almost sure event.
We will also see that for fixed $x\in\R$, the following event occurs almost surely:
\begin{linenomath}
\begin{align}
\label{ordering_event}
\OO_{x} \coloneqq \bigcap_{t>0}\bigcap_{y_1<y_2} \{\text{for every $\gamma_1\in G_{(x,0;y_1,t)}, \gamma_2\in G_{(x,0;y_2.t)}$, we have
$\gamma_1\leq\gamma_2$}\}.
\end{align}
\end{linenomath}
\end{subequations}
\end{enumerate}

The reason for geodesic ordering is explained in Figure \ref{bad_crosses}, although our proofs below are not phrased in terms of contradiction.
The fully rigorous argument transpires through Lemmas~\ref{reordering_lemma} and \ref{no_crosses}, with the aid of Lemma \ref{new_geodesics}.
Indeed, our first step is show a deterministic statement: even if violations of geodesic ordering occur, we can still form new geodesics that observe the correct ordering.
For ease of notation, whenever the time horizon $[s,t]$ is fixed, let us assume without loss of generality that $[s,t] = [0,1]$. 
(By \eqref{KPZ_scaling}, we can always rescale and shift coordinates to reduce to this case.)
In such scenarios, we simply write $G_{x,y} \coloneqq G_{(x,0;y,1)}$.

\begin{figure}
\centering
\subfloat[Violation of $\UU^\Q\subset\OO$]{
\includegraphics[trim=0.6in 0.6in 0.6in 0.6in, clip, width=0.48\textwidth]{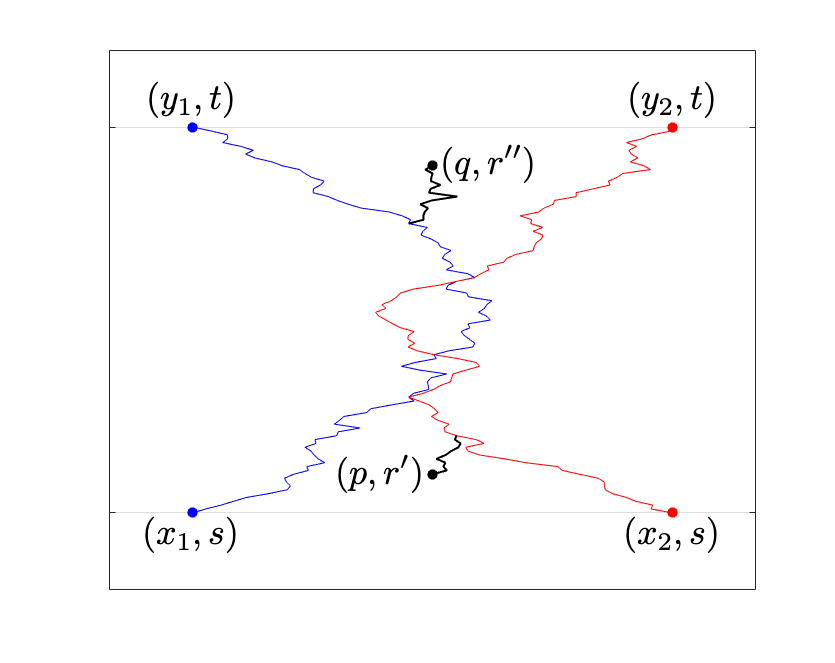}
\label{bad_cross_1}
}
\hfill
\subfloat[Violation of $\UU_x^\Q\subset\OO_x$]{
\includegraphics[trim=0.6in 0.6in 0.6in 0.64in, clip, width=0.48\textwidth]{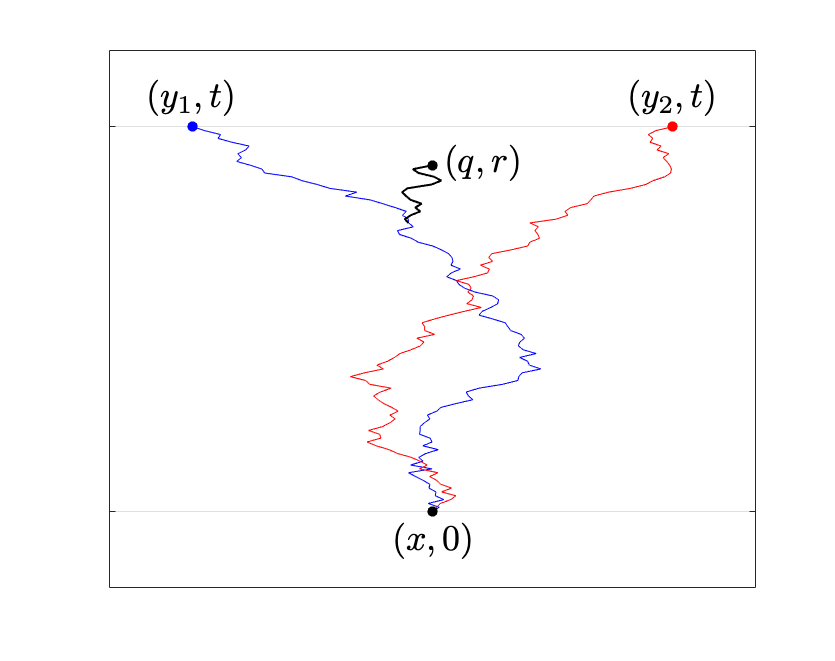}
\label{bad_cross_2}
}
\caption{Proof sketch of geodesic ordering.
Diagram (a) gives a scenario in which $\OO$ is not satisfied; that is, some geodesic $(x_1,s)\to(y_1,t)$ passes to the right of a geodesic $(x_2,s)\to(y_2,t)$.  As a consequence of planarity and concatenation, there is then more than one geodesic between suitably chosen points $(p,r')$ and $(q,r'')$ with rational coordinates.
As this last statement happens with probability zero, it follows that $\P(\OO)=1$.
A similar argument is pursued in diagram (b), which illustrates a violation of geodesic ordering when the initial endpoint $(x,0)$ is fixed.
If any two geodesics starting at $(x,0)$ separate only to later intersect before reaching their terminal locations, then we can find $(q,r)$ having rational coordinates and admitting more than one geodesic from $(x,0)$.
}
\label{bad_crosses}
\end{figure}

\begin{lemma} \label{reordering_lemma}
The following statements hold for any $x_1\leq x\leq x_2$ and $y_1\leq y\leq y_2$ such that $G_{x,y}$ is nonempty.
\begin{enumerate}[label=\textup{(\alph*)}]
\item \label{reordering_lemma_a}
For any $\gamma_1 \in G_{x_1,y_1}$, there is $\gamma \in G_{x,y}$ such that $\gamma_1 \leq \gamma$.
\item \label{reordering_lemma_b}
For any $\gamma_2 \in G_{x_2,y_2}$, there is $\gamma \in G_{x,y}$ such that $\gamma \leq \gamma_2$.
\item \label{reordering_lemma_c}
For any $\gamma_1 \in G_{x_1,y_1}$, $\gamma_2 \in G_{x_2,y_2}$ satisfying $\gamma_1\leq\gamma_2$, there is $\gamma \in G_{x,y}$ such that $\gamma_1\leq\gamma \leq \gamma_2$.
\end{enumerate}
\end{lemma}

\begin{proof}
The statements (a) and (b) are symmetric, and so we just prove (a).
Take any $\gamma_1 \in G_{x_1,y_1}$ and $\wt\gamma \in G_{x,y}$.
If
\eeq{ \label{reordering_setup}
\gamma_1(r) \neq \wt\gamma(r) \quad \text{for all $r\in[0,1]$},
} 
then we must have $x_1 < x$, and so together \eqref{reordering_setup} and the continuity of geodesics force $\gamma_1(r) < \wt\gamma(r)$ for all $r\in[0,1]$, as desired. 
If, on the other hand, $\gamma_1(r) = \wt\gamma(r)$ for some $r \in [0,1]$, then upon defining 
\eq{
s \coloneqq \inf\{r \geq 0 : \gamma_1(r) = \wt\gamma(r)\}, \qquad t \coloneqq \sup\{r \leq 1 : \gamma_1(r) = \wt\gamma(r)\},
}
we have 
\eeq{ \label{ordered_times}
0 \leq s \leq t \leq 1.
}
Furthermore, continuity implies
\eq{
\gamma_1(s) = \wt\gamma(s), \qquad \gamma_1(t) = \wt\gamma(t),
}
as well as
\eq{
\gamma_1(r)<\wt\gamma(r) \quad \text{for all $r\in[0,s)\cup(t,1]$}.
}
The first of the previous two displays allows us to define the path $\gamma \coloneqq \wt\gamma\big|_{[0,s]} \oplus \gamma_1\big|_{[s,t]} \oplus \wt\gamma\big|_{[t,1]}$, where if any of the inequalities in \eqref{ordered_times} is an equality, we simply omit the corresponding segment.
The second display implies that $\gamma_1 \leq \gamma$.
Finally, Lemma~\ref{new_geodesics} ensures $\gamma \in G_{x,y}$, thus completing the proof of (a).


For (c), we can apply (a) and (b) in succession upon noting that in the above proof, we had $\gamma(r)\in\{\wt\gamma(r),\gamma_1(r)\}$ for all $r\in[0,1]$.
Therefore, if $\wt\gamma$ is taken to be the element of $G_{x,y}$ resulting from (b), then $\gamma$ as defined above for (a) will necessarily satisfy
\eq{
\gamma(r) \leq \wt\gamma(r) \vee \gamma_1(r) \leq \gamma_2(r) \quad \text{for all $r\in[0,1]$},
}
in addition to $\gamma \geq \gamma_1$.
\end{proof}

We are now ready to prove geodesic ordering as described by the events $\OO$ and $\OO_x$ from \eqref{ordering_events}.
The notation of the following proof is mimicked in Figure \ref{bad_crosses}, to which the reader might refer for a visual explanation.

\begin{lemma} \label{no_crosses}
We have $\P(\OO)=1$ and $\P(\OO_{x}) = 1$  for any $x\in\R$.
\end{lemma}

\begin{proof}
Let us define the events
\eq{ 
\UU^\Q \coloneqq \bigcap_{\substack{r',r''\in\Q \\ r'<r''}}\bigcap_{p,q\in\Q} \{|G_{(p,r';q,r'')}|=1\}, \qquad
\UU^\Q_{x} \coloneqq \bigcap_{r \in \Q\cap(0,\infty)}\bigcap_{q\in\Q}\{|G_{(x,0;q,r)}|=1\}.
}
Since $\Q$ is countable and $\P(|G_u|=1)=1$ for any $u\in\R^4_\uparrow$, we have $\P(\UU^\Q)=\P(\UU^\Q_{x}) = 1$.
Therefore, if we can show that $\OO\supset\UU^\Q$ and $\OO_{x}\supset\UU^\Q_{x}$,
then $\OO$ and $\OO_{x}$ are necessarily measurable by Remark~\ref{complete_remark}, and also $\P(\OO)=\P(\OO_{x})=1$.
Let us first prove $\UU^\Q_{x} \subset \OO_{x}$, as the argument for $\UU^\Q\subset\OO$ will require only slight modifications.

Assume $\UU^\Q_{x}$ occurs.
Consider any time $t>0$, any positions $y_1 < y_2$, and any geodesics $\gamma_1\in G_{(x,0;y_1,t)}$ and $\gamma_2\in G_{(x,0;y_2,t)}$.
Since $y_1 < y_2$ and $\gamma_1$ and $\gamma_2$ are continuous functions, there is some $\eps\in(0,t)$ such that $\gamma_1(r) < \gamma_2(r)$ for all $r\in[t-\eps,t]$.
So take any $r \in [t-\eps,t]\cap\Q$ and pick any $q\in(\gamma_1(r),\gamma_2(r))$; consider the unique $\gamma \in G_{(x,0;q,r)}$.
Given this uniqueness, we can apply Lemma~\hyperref[reordering_lemma_a]{\ref*{reordering_lemma}\ref*{reordering_lemma_a}} to conclude $\gamma_1\big|_{[0,r]} \leq \gamma$.
Analogously, by Lemma~\hyperref[reordering_lemma_b]{\ref*{reordering_lemma}\ref*{reordering_lemma_b}}, we must also have $\gamma\leq\gamma_2\big|_{[0,r]}$.
Together, these two facts yield $\gamma_1\big|_{[0,r]}\leq\gamma_2\big|_{[0,r]}$.
Since our choice of $r$ ensures $\gamma_1\big|_{[r,t]}\leq\gamma_2\big|_{[r,t]}$, we thus have $\gamma_1\leq\gamma_2$.
Indeed, $\OO_x$ has occurred.

Now assume the occurrence of $\UU^\Q$.
Consider any $x_1<x_2$, $y_1<y_2$, $s<t$, and any geodesics $\gamma_1\in G_{(x_1,s;y_1,t)}$ and $\gamma_2\in G_{(x_2,s;y_2,t)}$.
Since $\gamma_1(s) = x_1<x_2 = \gamma_2(s)$, continuity guarantees that $\gamma_1(r)<\gamma_2(r)$ for all $r\in[s,s+\eps]$, for some $\eps>0$.
By symmetric reasoning, we may assume $\gamma_1(r)<\gamma_2(r)$ for all $r\in[t-\eps,t]$.
Now pick any rational times $r'\in\Q\cap[s,s+\eps]$, $r''\in\Q\cap[t-\eps,t]$, as well as rational spatial coordinates $p\in(\gamma_1(r'),\gamma_2(r'))$, $q\in(\gamma_1(r''),\gamma_2(r''))$.
By assumption, there is a unique $\gamma\in G_{(p,r';q,r'')}$.
Moreover, Lemma~\hyperref[subgeodesic_lemma_b]{\ref*{subgeodesic_lemma}\ref*{subgeodesic_lemma_b}} ensures that $\gamma_1\big|_{[r',r'']}$ and $\gamma_2\big|_{[r',r'']}$ are themselves geodesics.
Therefore, the same argument as above (using Lemma~\ref{reordering_lemma} with a rescaled time horizon) yields
\eq{
\gamma_1(r)\leq\gamma(r)\leq\gamma_2(r) \quad \text{for all $r\in[r',r'']$}.
}
But of course, since $r'\leq s+\eps$ and $r''\geq t-\eps$, we also know 
\eq{
\gamma_1(r) < \gamma_2(r) \quad \text{for all $r\in[s,r']\cup[r'',t]$}.
}
Hence $\gamma_1(r)\leq\gamma_2(r)$ at every $r\in[s,t]$, as desired so that $\OO$ is seen to occur.
\end{proof}


While Lemmas~\ref{reordering_lemma} and \ref{no_crosses} give us control over violations of geodesic ordering, the final result of this section considers violations of geodesic uniqueness.
It is not required elsewhere in the paper, but rather included as an incidental result.
Given $x\in\R$, let us consider the set 
\eq{ 
\MM_{x} \coloneqq \{y \in \R : |G_{x,y}| \geq 2\}.
}

\begin{lemma} \label{non_unique_lemma}
On the almost sure event $\OO_x$ from \eqref{ordering_event}, the set $\MM_{x}$ is countable.
\end{lemma}

\begin{proof} 
Suppose $y\in\R$ is such that $G_{x,y}$ contains two distinct elements $\gamma_{y}$ and $\wt\gamma_{y}$.
Without loss of generality, $\gamma_{y}(r_y) < \wt\gamma_{y}(r_y)$ for some $r_y\in(0,1)$, where we may assume by continuity that $r_y\in\Q$.
Moreover, we can choose $q_y\in\Q$ such that
\eq{
\gamma_{y}(r_y) < q_y < \wt\gamma_{y}(r_y).
}
Now, if $y_1 < y_2$ and both $|G_{x,y_1}|$ and $|G_{x,y_2}|$ are at least 2, then it must be that $(q_{y_1},r_{y_1}) \neq (q_{y_2},r_{y_2})$.
Indeed, we would otherwise have
\eq{
\gamma_{y_2}(r_{y_2}) < q_{y_2}=q_{y_1} < \wt\gamma_{y_1}(r_{y_1}) = \wt\gamma_{y_1}(r_{y_2}),
}
which is exactly the scenario ruled out by $\OO_{x}$.
In summary, each $y$ for which $|G_{x,y}| \geq 2$ can be associated uniquely to some element of the countable set $\Q \times \Q$.
The claim of the lemma is thus evident.
\end{proof}


\section{Proofs of input results} \label{disjoint_geo_rarity_proofs}

In Section~\ref{polymer_convergence}, we prove Theorem~\ref{polymers_converge} and Corollary~\ref{eventual_disjointness}.
We then use Corollary~\ref{eventual_disjointness} in Section~\ref{rarity_section} to deduce Theorem~\ref{disjoint_geo_rarity}  from the corresponding result in \cite{hammond20}.
Corollary~\ref{disjoint_2geo_rarity} will follow from a brief topological argument.
Finally, Section~\ref{no_arbitrary_closeness_proof} gives the proof of Theorem~\ref{no_arbitrary_closeness}.

\subsection{Convergence of polymers} \label{polymer_convergence}

Throughout this section, we fix $u = (x,s;y,t)\in\R^4_\uparrow$ and assume the setting of Theorem~\ref{geodesics_converge}.
That is, $G_u$ consists of a single element $\gamma_u$, the event $\PP$ occurs, and for each $n$, we have chosen 
\eq{
\vphi_n : [sn+2n^{2/3}x,tn+2n^{2/3}y)\to\llbrack \lfloor sn\rfloor,\lfloor tn\rfloor\rrbrack
}
such that $\Gamma^{(\vphi_n)}_{n,u}\in G_{n,u}$ (equivalently, $R_n(\vphi_n) \in P_{n,u}$).
By Theorem~\ref{geodesics_converge}, we have the following uniform convergence of functions on $[s,t]$:
\eeq{ \label{uniform_convergence}
\lim_{n\to\infty} \|\Gamma^{(\vphi_n)}_{n,u} - \gamma_u\|_{\infty} = 0.
}
We preface the proof of Theorem~\ref{polymers_converge} with the following simple observations about the geometry of $n$-polymers and $n$-geodesics.
The reader may find Figure~\ref{blpp_fig} to be a useful reference.

\begin{lemma} \label{to_good_pts}
Assume the setting of Theorem~\ref{geodesics_converge}.
Then for any $\eps>0$, there is $N$ such that for all $n\geq N$, we have the following:
\begin{enumerate}[label=\textup{(\alph*)}]

\item \label{to_good_pts_a}
Every horizontal segment in $\wt R_{n,u}(\vphi_n)$ has length at most $\eps$.

\item \label{to_good_pts_b}
 Every oblique segment in $R_n(\vphi_n)$ has horizontal width at most $\eps$.
 
\item \label{to_good_pts_c}
Every oblique segment in $R_n(\vphi_n)$ has vertical height at most $2\eps n^{-1/3}$.
\end{enumerate}
\end{lemma}

\begin{proof}
Fix $\eps>0$.
Notice that if $\wt R_{n,u}(\vphi_n)$ has a horizontal segment at height $r$, then $\Gamma^{(\vphi_n)}_{n,u}(\cdot)$ has a jump discontinuity at time $r$, where the size of the jump is exactly equal to the length of the horizontal segment.
Therefore, to satisfy (a), it suffices to choose $N$ large enough that for all $n\geq N$, every discontinuity of $\Gamma^{(\vphi_n)}_{n,u}$ is no larger than $\eps$.
Such an $N$ exists by \eqref{uniform_convergence} and the uniform continuity of $\gamma_u$.

Now (b) follows from (a) because every oblique segment in $R_n(\vphi_n)$ corresponds to a horizontal segment in $\wt R_{n,u}(\vphi_n)$ of the same width.
Finally, (c) follows from (b) because the slope of any oblique segment in $R_n(\vphi_n)$ is $-2n^{-1/3}$.
\end{proof}

\begin{proof}[Proof of Theorem~\ref{polymers_converge}]
First we prove \eqref{coupled_polymers}.
Fix any $\eps>0$.
Observe that the horizontal segments in $R_n(\vphi_n)$, minus their rightmost points, consist entirely of points $(z,r)$ of the form
\eeq{ \label{normal_form}
(z,r) = R_n(z',\vphi_n(z')), \quad z' \in [sn+2n^{2/3}x,tn+2n^{2/3}y).
}
Similarly, the oblique segments in $\wt R_{n,u}(\vphi_n)$, minus their uppermost points, consist entirely of points $(\tilde z,\tilde r)$ of the form
\eeq{ \label{tilde_form}
 \tilde z = \Gamma_{n,u}^{(\vphi_n)}(\tilde r), \quad L_{n,u}(\tilde r) = z' \in [sn+2n^{2/3}x,tn+2n^{2/3}y).
}
Therefore, these two categories of points are in bijection $(z,r) \leftrightarrow (\tilde z,\tilde r)$ through the unscaled coordinate $z'$.
If we can show that
\eeq{ \label{pts_in_bijection}
\limsup_{n\to\infty} \sup_{z'} \|(z,r)-(\tilde z,\tilde r)\| = 0,
}
then we claim \eqref{coupled_polymers} holds.
Indeed, Lemma~\hyperref[to_good_pts_a]{\ref*{to_good_pts}\ref*{to_good_pts_a}} shows that for $n \geq N$, every point in $\wt R_{n,u}(\vphi_n)$ is within distance $\eps$ of some $(\tilde z,\tilde r)$ of the form \eqref{tilde_form}.
Meanwhile, Lemma~\hyperref[to_good_pts_b]{\ref*{to_good_pts}(b,c)} shows that every point in $R_{n}(\vphi_n)$ is within distance $\eps + 2\eps n^{-1/3}$ of some $(z,r)$ of the form \eqref{normal_form}.
Consequently, \eqref{pts_in_bijection} leads to
\eq{
\limsup_{n\to\infty} \dist_H(R_n(\vphi_n),\wt R_{n,u}(\vphi_n)) \leq \limsup_{n\to\infty} \Big[2\eps+2\eps n^{-1/3} + \sup_{z'}\|(z,r)-(\tilde z,\tilde r)\| \Big] = 2\eps.
}
As $\eps>0$ is arbitrary, \eqref{coupled_polymers} follows.
We now proceed to establish \eqref{pts_in_bijection}.

Because of \eqref{uniform_convergence}, there exist random $L$ (depending only on $\gamma_u$) and random $N$ large enough that
\eeq{ \label{horizontal_bd}
|\Gamma_{n,u}^{(\vphi_n)}(r)| \leq L \quad \text{for all $r\in[s,t]$, $n\geq N$}.
}
Fix any $z'\in[sn+2n^{2/3}x,tn+2n^{2/3}y)$, and consider $(z,r)$ and $(\tilde z,\tilde r)$ as defined through \eqref{normal_form} and \eqref{tilde_form}.
In particular,
\eq{
z = \frac{z' - \vphi(z')}{2n^{2/3}} = \frac{L_{n,u}(\tilde r) - \vphi_n(L_{n,u}(\tilde r))}{2n^{2/3}} = \Gamma^{(\vphi_n)}_{n,u}(\tilde r) = \tilde z,
}
and so
\eeq{ \label{no_z_diff}
\|(z,r) - (\tilde z,\tilde r)\| = |r - \tilde r|.
}
Now observe that
\eq{
L_{n,u}(\tilde r) = z' = 2n^{2/3}\Gamma_{n,u}^{(\vphi_n)}(\tilde r) + rn \quad \implies \quad  r =\tilde r + 2n^{-1/3}\Big( \frac{t-\tilde r}{t-s}x + \frac{\tilde r-s}{t-s}y-\Gamma^{(\vphi_n)}_{n,u}(\tilde r)\Big),
}
from which we can deduce, by \eqref{horizontal_bd}, the uniform bound
\eeq{ \label{uniform_r_diff}
| r - \tilde r| \leq 2n^{-1/3}( |x| + |y|+L).
}
Together, \eqref{no_z_diff} and \eqref{uniform_r_diff} imply \eqref{pts_in_bijection}, and so \eqref{coupled_polymers} has been proved.

Now we turn our attention to showing \eqref{polymer_geodesic}, which is clearly implied by the following statement:
%
%
\eeq{ \label{eventual_disjointness_prep}
\lim_{n\to\infty} \sup_{(z,r)\in R_n(\vphi_n)} |z - \gamma_u(r)| = 0.
}
So let us just establish \eqref{eventual_disjointness_prep}.
%
If we denote, for each $r\in[s,t]$, the leftmost and rightmost points of $(\R\times\{r\})\cap R_{n}(\vphi_n)$ by
\eq{
a_n(r) \coloneqq \inf\{z \in \R : (z,r) \in R_{n}(\vphi_n)\}, \qquad b_n(r) \coloneqq \sup\{z \in \R : (z,r) \in  R_{n}(\vphi_n)\},
}
then \eqref{eventual_disjointness_prep} is equivalent to
\eeq{ \label{left_right}
a_n(r) \to \gamma_u(r) \quad \text{and} \quad b_n(r) \to \gamma_u(r) \quad \text{uniformly in $r\in[s,t]$}.
}
To argue \eqref{left_right}, let us consider the analogous quantities for $\wt R_{n,u}(\vphi_n)$, namely
\eq{
\tilde a_n(r) \coloneqq \inf\{z \in \R : (z,r) \in \wt R_{n,u}(\vphi_n)\}, \qquad
\tilde b_n(r) \coloneqq \sup\{z \in \R : (z,r) \in \wt R_{n,u}(\vphi_n)\},
}
and observe (perhaps with the aid of Figure~\ref{geodesic_shade}) that
\eq{
\tilde a_n(r) = \Gamma^{(\vphi_n)}_{n,u}(r), \qquad  
\tilde b_n(r) = \begin{cases}
\lim_{r'\nearrow r} \Gamma^{(\vphi_n)}_{n,u}(r') &\text{if } r\in(s,t], \\
\tilde a_n(s) &\text{if }r=s.
\end{cases}
}
By \eqref{uniform_convergence}, we then have 
\eeq{ \label{tilde_left_right}
\tilde a_n(r) \to \gamma_u(r) \quad \text{and} \quad \tilde b_n(r) \to \gamma_u(r) \quad \text{uniformly in $r\in[s,t]$}.
}
Now let $\eps > 0$ and choose $\delta\in(0,\eps]$ sufficiently small that 
\eq{
|\gamma_u(r')-\gamma_u(r)| \leq \eps \quad \text{whenever $r,r'\in[s,t]$, $|r-r'|\leq\delta$}.
}
By \eqref{tilde_left_right} and \eqref{coupled_polymers}, we can select $N$ such that for all $n\geq N$, we have
\eq{
|\tilde a_n(r)-\gamma_u(r)| \leq \eps \quad \text{and} \quad |\tilde b_n(r)-\gamma_u(r)| \leq \eps \quad \text{for all $r\in[s,t]$},
}
as well as
\eq{
\dist_H(R_n(\vphi_n),\wt R_{n,u}(\vphi_n)) \leq \delta.
}
Since $(a_n(r),r)\in R_n(\vphi)$, it follows from the above display that
\eq{
\inf_{(\tilde z,\tilde r)\in \wt R_{n,u}(\vphi_n)} \|(a_n(r),r) - (\tilde z,\tilde r)\| \leq \delta \quad \text{for all $r\in[s,t]$, $n\geq N$},
}
which can be trivially rewritten as
\eq{
\inf_{\substack{(\tilde z,\tilde r)\in \wt R_{n,u}(\vphi_n) \\ \tilde r \in [r-\delta,r+\delta]}} \|(a_n(r),r) - (\tilde z,\tilde r)\| \leq \delta \quad \text{for all $r\in[s,t]$, $n\geq N$}.
}
On the other hand, for any $(\tilde z,\tilde r)\in\wt R_{n,u}(\vphi_n)$ with $|\tilde r-r|\leq \delta$, we have
\eq{
|\tilde z - \gamma_u(r)| &\leq |\tilde z - \gamma_u(\tilde r)| + |\gamma_u(\tilde r) - \gamma_u(r)| \\
&\leq |\tilde a_n(\tilde r) - \gamma_u(\tilde r)|+|\tilde b_n(\tilde r) - \gamma_u(\tilde r)| + |\gamma_u(\tilde r) - \gamma_u(r)| \leq 3\eps.
}
Together, the two previous displays imply
\eq{
|a_n(r) - \gamma_u(r)| \leq \delta+3\eps \leq 4\eps \quad \text{for all $r\in[s,t]$, $n\geq N$},
}
and an analogous argument shows
\eq{
|b_n(r) - \gamma_u(r)| \leq \delta+3\eps \leq 4\eps \quad \text{for all $r\in[s,t]$, $n\geq N$}.
}
As $\eps>0$ is arbitrary, we conclude that \eqref{left_right} holds.
\end{proof}

Given the convergence \eqref{polymer_geodesic} from Theorem~\ref{polymers_converge} (or equivalently, \eqref{eventual_disjointness_prep}), it is a simple matter to verify Corollary~\ref{eventual_disjointness}.

\begin{proof}[Proof of Corollary~\ref{eventual_disjointness}]
Recall the notation from the statement of the corollary.
It is trivial that
\eq{
&\hspace{1.1ex}\inf\{|z_1-z_2| : (z_1,r)\in R_n(\vphi_n),(z_2,r)\in R_n(\phi_n),r\in[s,t]\}\\
\geq &\inf_{r\in[s,t]}|\gamma_{u_1}(r) - \gamma_{u_2}(r)|
- \sup_{(z_1,r_1)\in R_n(\vphi_n)} |z_1 - \gamma_{u_1}(r)|
- \sup_{(z_2,r_2)\in R_n(\phi_n)} |z_2 - \gamma_{u_2}(r_2)|.
}
Under the hypotheses of the corollary, we have
\eq{
\inf_{r\in[s,t]}|\gamma_{u_1}(r) - \gamma_{u_2}(r)| > 0,
}
while \eqref{eventual_disjointness_prep} gives
\eq{
\lim_{n\to\infty} \sup_{(z_1,r_1)\in R_n(\vphi_n)} |z_1 - \gamma_{u_1}(r_1)| 
= \lim_{n\to\infty} \sup_{(z_2,r_2)\in R_n(\phi_n)} |z_2 - \gamma_{u_2}(r_2)| = 0.
}
Therefore, for all $n$ sufficiently large, we have
\eq{
\inf\{|z_1-z_2| : (z_1,r)\in R_n(\vphi_n),(z_2,r)\in R_n(\phi_n),r\in[s,t]\} > 0,
}
meaning that $R_n(\vphi_n)$ and $R_n(\phi_n)$ are disjoint.
\end{proof}

\subsection{Tail estimates for the size of a disjoint collection of geodesics} \label{rarity_section}
Before proving Theorem~\ref{disjoint_geo_rarity}, we need to know that the relevant random variable is measurable.

\begin{prop} \label{disjoint_measurable}
For any times $s<t$ and subsets $A,B,C\subset\R$ with $C$ countable, the quantity \linebreak $\MDG_{s,t}^C(A,B)$ is a measurable random variable almost surely taking values in $\{1,2,\dots\}\cup\{\infty\}$.
\end{prop}

Since the proof of Proposition~\ref{disjoint_measurable} will need to consider certain events involving geodesics, it will be useful to have the following description of a geodesic.

\begin{lemma}
\label{geodesic_construction}
\textup{\cite[proof of Lemma 13.2]{dauvergne-ortmann-virag?}}
On the almost sure event $\PP$ of Theorem~\ref{geodesics_converge}, the following is true for every $u\in\R^4_\uparrow$.
If $G_u$ consists of a single element $\gamma_u$, then for every $r\in(s,t)$, there is a unique $z_r\in\R$ satisfying
\eq{ 
\LL(x,s;z_r,r)+\LL(z_r,r;y,t) = \sup_{z\in\R}\, [\LL(x,s;z,r)+\LL(z,r;y,t)] = \LL(x,s;y,t),
}
and $\gamma_u(r) = z_r$.
\end{lemma}

We use the above characterization of $\gamma_u$ to prove the next lemma, which constitutes the bulk of the work toward Proposition~\ref{disjoint_measurable}.
Let $\NI_{s,t}(x_1,x_2;y_1,y_2)$ denote the event that every $\gamma_1\in G_{(x_1,s;y_1;t)}$ is disjoint from every $\gamma_2\in G_{(x_2,s;y_2,t)}$.

\begin{lemma} \label{NI_measurable}
For any $x_1,x_2,y_1,y_2\in\R$ and $s<t$, the event $\NI_{s,t}(x_1,x_2;y_1,y_2)$ is measurable.
\end{lemma}

\begin{proof}
Recall that almost surely, both $G_{(x_1,s;y_1;t)}$ and $G_{(x_2,s;y_2,t)}$ are singletons, and $\PP$ occurs.
Because we have assumed in Remark~\ref{complete_remark} that $\FF$ is complete, it suffices to show that the intersection of $\NI_{s,t}(x_1,x_2;y_1,y_2)$ with these three almost sure events is measurable.
So let us assume henceforth that $\gamma_1\in G_{(x_1,s;y_1;t)}$ and $\gamma_2 \in G_{(x_2,s;y_2;t)}$ are unique, and that $\PP$ occurs; thus $\gamma_1$ and $\gamma_2$ are as described in Lemma~\ref{geodesic_construction}.
In this case,
\eq{
\NI_{s,t}(x_1,x_2;y_1,y_2) = \{\gamma_1,\gamma_2\text{ disjoint}\},
}
and so it suffices to show the measurability of $\Omega\setminus\{\gamma_{1},\gamma_{2}\text{ disjoint}\}$.
Indeed, by Lemma~\ref{geodesic_construction}, $\gamma_1$ and $\gamma_2$ fail to be disjoint precisely when there exists $(z,r)\in\R\times(s,t)$ such that simultaneously
\eq{
\LL(x_1,s;z,r) + \LL(z,r;y_1,t) = \LL(x_1,s;y_1,t) \quad \text{and} \quad \LL(x_2,s;z,r) + \LL(z,r;y_2,t) = \LL(x_2,s;y_2,t).
}
Since $\LL$ is continuous, the existence of such $(z,r)$ can be determined by the values of $\LL$ on a countable set.
To be completely precise, $\Omega\setminus\{\gamma_{1},\gamma_{2}\text{ disjoint}\}$ is equivalent to
\eq{
\bigcup_{R=1}^\infty \bigcap_{m=1}^\infty \bigcup_{(z,r)\in\Q^2\cap([-R,R]\times[s+R^{-1},t-R^{-1}])}\bigcup_{i\in\{1,2\}}\{\LL(x_i,s;z,r) + \LL(z,r;y_i,t) + 1/m > \LL(x_i,s;y_i,t)\}.
}
Therefore, this event is measurable.
\end{proof}

\begin{proof}[Proof of Proposition~\ref{disjoint_measurable}]
We wish to show that, for any $k\in\{1,2,\dots\}$, the disjointness event
$\{\MDG_{s,t}^C(A,B)\geq k\}$
belongs to the sigma-algebra $\FF$.
First note that the countability of $C$ implies that the event $\bigcap_{p,q\in C} \{|G_{(p,s;q,t)}|=1\}$ occurs with probability one.
Therefore, by Remark~\ref{complete_remark}, it suffices to show that
\eq{
\{\MDG_{s,t}^C(A,B)\geq k\} \cap \bigcap_{p,q\in C} \{|G_{(p,s;q,t)}|=1\} \cap \PP \in \FF.
}
On the almost sure event $\bigcap_{p,q\in C} \{|G_{(p,s;q,t)}|=1\} \cap \PP$, 
the set under consideration can be expressed as
\eq{
\{\MDG_{s,t}^C(A,B)\geq k\} = \bigcup_{\substack{p_1,\dots,p_k\in A\cap C \\ q_1,\dots,q_k\in B\cap C}}
\bigcap_{1 \leq i < j \leq k} \NI_{s,t}(p_i,p_j;q_i,q_j). 
}
As the union and the intersection in this display take place over countable index sets, Lemma~\ref{NI_measurable} completes the proof.
\end{proof}


Recall the definition of an $n$-polymer from Section~\ref{prelimiting_model}.
For intervals $I,J\subset\R$, denote by $\MDP_n(I,J)$ the maximum size of a collection of disjoint $n$-polymers having endpoints of the form $(x,0)$ and $(y,1)$ with $x\in I$ and $y\in J$.
The following result of \cite{hammond20} will naturally translate into Theorem~\ref{disjoint_geo_rarity}, since we know from Section~\ref{polymer_convergence} that $n$-polymers converge to (the graph of) geodesics.

\begin{thm}
\label{disjoint_poly_rarity}
\textup{\cite[Thm.~1.1]{hammond20}}
There exists a positive constant $G$ such that the following holds.
For any $\delta>0$, integers $k$ and $n$, and $z,w\in\R$ satisfying
\eeq{ \label{disjoint_poly_rarity_conditions}
k\geq2, \quad \delta \leq G^{-4k^2}, \quad n\geq G^{k^2}(1+|z-w|^{36})\delta^{-G}, \quad |z-w| \leq \delta^{-1/2}(\log \delta^{-1})^{-2/3}G^{-k},
}
we have
\eq{
\P\Big(\MDP_n([z-\delta,z+\delta],[w-\delta,w+\delta])\geq k\Big)
\leq G^{k^3}\exp\big\{G^k(\log\delta^{-1})^{5/6}\big\}\delta^{(k^2-1)/2}.
}
\end{thm}

\begin{proof}[Proof of Theorem~\ref{disjoint_geo_rarity}]
Let $C\subset\R$ be countable.
Fix $k$, $\eps>0$, and $u=(x,s;y,t)\in\R^4_\uparrow$ satisfying \eqref{disjoint_geo_rarity_conditions}.
Then \eqref{disjoint_poly_rarity_conditions} is satisfied with $\delta = \eps/(t-s)^{2/3}$, $z=x/(t-s)^{2/3}$, $w=y/(t-s)^{2/3}$, and $n$ sufficiently large, so that we will ultimately be able to invoke Theorem~\ref{disjoint_poly_rarity}.
By the scaling in \eqref{KPZ_scaling}, 
we have
\eeq{ \label{rescale}
&\P\Big(\MDG_{s,t}^C([x-\eps,x+\eps],[y-\eps,y+\eps])\geq k\Big) \\
&= \P\Big(\MDG_{0,1}^{(t-s)^{-2/3}C}([z-\delta,z+\delta],[w-\delta,w+\delta])\geq k\Big).
}
Now assume of the coupling of Theorem~\ref{blpp_airy}, and suppose as in Theorem~\ref{geodesics_converge}, the almost sure occurrence of $\PP$ and of the event $\{|G_u|=1\}$ for every $u = (z,0;w,1)$ with $z,w \in (t-s)^{-2/3}C$.
If
\eq{
\MDG_{0,1}^{(t-s)^{-2/3}C}([z-\delta,z+\delta],[w-\delta,w+\delta])\geq k,
}
then there are $u_i = (z_i,0;w_i,1)$, $i=1,\dots,k$, such that 
\eq{
z_i \in [z-\delta,z+\delta]\cap (t-s)^{-2/3}C, \qquad w_i \in [w-\delta,w+\delta]\cap (t-s)^{-2/3}C,
}
admitting disjoint and unique geodesics $\gamma_{u_1},\dots,\gamma_{u_k} : [0,1]\to\R$.
By Corollary~\ref{eventual_disjointness}, if we select some $R_n(\vphi_n^{(i)}) \in P_{n,u_i}$ for each $n$, then the polymers $R_n(\vphi_n^{(1)}),\dots,R_n(\vphi_n^{(k)})$ must be disjoint for all $n$ sufficiently large.
We have thus argued that, on the almost sure event 
\eq{
\PP \cap \bigcap_{z,w\in(t-s)^{-2/3}C} \{|G_{(z,0;w,1)}| = 1\},
} 
we have
\eq{
&\Big\{\MDG_{0,1}^{(t-s)^{-2/3}C}([z-\delta,z+\delta],[w-\delta,w+\delta])\geq k\Big\}  \\
&\subset \bigcup_{N=1}^\infty \bigcap_{n=N}^\infty \{\MDP_n([z-\delta,z+\delta],[w-\delta,w+\delta])\geq k\}.
}
This containment gives the first inequality in the following chain:
\eq{
&\P\Big(\MDG_{s,t}^C([x-\eps,x+\eps],[y-\eps,y+\eps])\geq k\Big) \\
&\stackref{rescale}{=}\P\Big(\MDG_{0,1}^{(t-s)^{-2/3}C}([z-\delta,z+\delta],[w-\delta,w+\delta])\geq k\Big) \\
&\stackrefp{rescale}{\leq} \P\bigg(\bigcup_{N=1}^\infty \bigcap_{n=N}^\infty \{\MDP_n([z-\delta,z+\delta],[w-\delta,w+\delta])\geq k\}\bigg) \\
&\stackrefp{rescale}{=} \lim_{N\to\infty} \P\bigg(\bigcap_{n=N}^\infty \{\MDP_n([z-\delta,z+\delta],[w-\delta,w+\delta])\geq k\}\bigg) \\
&\stackrefp{rescale}{\leq}\liminf_{n\to\infty} \P\Big(\MDP_n([z-\delta,z+\delta],[w-\delta,w+\delta])\geq k\Big) \\
&\stackrefp{rescale}{\leq} G^{k^3}\exp\big\{G^k(\log\delta^{-1})^{5/6}\big\}\delta^{(k^2-1)/2},
}
where we have used Theorem~\ref{disjoint_poly_rarity} to obtain the final inequality.
\end{proof}

We now prove that in the case $k=2$, one can take $C=\R$.

\begin{proof}[Proof of Corollary~\ref{disjoint_2geo_rarity}]
Assume the occurrence of the geodesic existence and ordering events $\EE$ and $\OO$ from \eqref{existence_event} and \eqref{global_ordering_event}.
We will argue that on $\EE\cap\OO$, we have the following equality of events for any $u=(x,s;y,t)\in\R^4_\uparrow$ and $\eps>0$:
\eq{
&\{\MDG^\R_{s,t}((x-\eps,x+\eps),(y-\eps,y+\eps))\geq2\} \\
&= \{\MDG^\Q_{s,t}((x-\eps,x+\eps),(y-\eps,y+\eps)) \geq 2\}.
}
Since the first event is clearly implied by the second, we need only prove the reverse containment.
Once this is done, we will have shown (i) that the first event is measurable, as the second is measurable by Proposition~\ref{disjoint_measurable}; and (ii) that whenever $\eps>0$ satisfies \eqref{disjoint_2geo_rarity_conditions}, the first event adheres to the estimate \eqref{disjoint_2geo_rarity_ineq}, since the second event is contained in $\{\MDG^\Q_{s,t}([x-\eps,x+\eps],[y-\eps,y+\eps]) \geq 2\}$, which in turn adheres to \eqref{disjoint_geo_rarity_ineq}.

So suppose
$\MDG_{s,t}^\R((x-\eps,x+\eps),(y-\eps,y+\eps)) \geq 2$.
That is, there are $x_1<x_2$ in $(x-\eps,x+\eps)$ and $y_1 < y_2$ in $(y-\eps,y+\eps)$ admitting geodesics $\gamma_1\in G_{(x_1,s;y_1,t)}$ and $\gamma_2\in G_{(x_2,s;y_2,t)}$ that satisfy
\eeq{ \label{1_less_2}
\gamma_1(r) < \gamma_2(r) \quad \text{for all $r\in[s,t]$}.
}
Then select any rationals
\eq{
&p_1\in\Q\cap(x-\eps,x_1),\quad p_2\in\Q\cap(x_2,x+\eps), \\
&q_1\in\Q\cap(y-\eps,y_1), \quad q_2\in\Q\cap(y_2,y+\eps),
}
and any $\wt\gamma_1\in G_{(p_1,s;q_1,t)}$ and $\wt\gamma_2\in G_{(p_2,s;q_2,t)}$.
By geodesic ordering, we have
\eq{
\wt\gamma_1(r) \leq \gamma_1(r) \quad \text{and} \quad \gamma_2(r)\leq\wt\gamma_2(r) \quad \text{for all $r\in[s,t]$}.
}
In light of \eqref{1_less_2}, this implies that $\wt\gamma_1$ and $\wt\gamma_2$ are disjoint, meaning
\eq{
\MDG_{s,t}^\Q((x-\eps,x+\eps),(y-\eps,y+\eps)) \geq 2.
}
\end{proof}

\subsection{Geodesics in a common compact set cannot be arbitrarily close} \label{no_arbitrary_closeness_proof}

The proof of a final input remains.

\begin{proof}[Proof of Theorem~\ref{no_arbitrary_closeness}]
Let $K$ be a given compact subset of $\R^4_\uparrow$.
For each $\eps>0$, define the event
\eeq{ \label{B_eps_def}
\BB_\eps \coloneqq \bigg\{\parbox[c][][c]{2.2in}{\centering$\exists\, u_1 = (x_1,s;w_1,t)\in K, \gamma_1\in G_{u_1}$ \\ $\exists\, u_2 = (x_2,s;w_2,t) \in K, \gamma_2\in G_{u_2}$} : 0<\gamma_2(r)-\gamma_1(r)<\eps\ \ \forall\, r\in[s,t]\bigg\}.
}
We wish to show $\P\big(\PP\cap\bigcap_{\eps>0} \BB_\eps\big) = 0$.
Recall the random number $R>0$ from Remark~\hyperref[proper_remark_c]{\ref*{proper_remark}\ref*{proper_remark_c}}.
By possibly replacing $R$ with a larger deterministic number, we may assume that $K\subset[-R,R]^4$.
If we can show that for any integer $m$, the event $\{R\leq m\}\cap\bigcap_{\eps>0}\BB_\eps$ is contained in a probability zero event, then measurability will be implied by Remark~\ref{complete_remark}, and
\eq{
\P\Big(\PP\cap\bigcap_{\eps>0} \BB_\eps\Big) 
= \P\bigg(\PP\cap\bigcup_{m=1}^\infty \Big(\{R\leq m\}\cap\bigcap_{\eps>0}\BB_\eps\Big)\bigg) 
&\leq \sum_{m=1}^\infty \P\Big(\{R\leq m\} \cap \bigcap_{\eps>0} \BB_\eps\Big) = 0.
}

So let us assume $R\leq m$. 
By compactness of $K$, there is a deterministic number $\delta>0$ such that whenever $u = (x,s;y,t)\in K$, we have $t-s\geq6\delta$.
Therefore, if $\gamma_1$ and $\gamma_2$ are as in \eqref{B_eps_def}, then there is some
$r \in [-m,m]\cap 3\delta\Z$ satisfying $[r,r+3\delta]\subset[s,t]$.
Furthermore, there are $x,y,z,w\in [-m,m]\cap \eps\Z$ such that
\eq{
\gamma_1(r),\gamma_2(r) &\in(x-\eps,x+\eps), & \gamma_1(r+\delta),\gamma_2(r+\delta) &\in(y-\eps,y+\eps), \\
 \gamma_1(r+2\delta),\gamma_2(r+2\delta) &\in(z-\eps,z+\eps), &
  \gamma_1(r+3\delta),\gamma_2(r+3\delta) &\in(w-\eps,w+\eps).
}
Since subpaths of geodesics are again geodesics by Lemma~\hyperref[subgeodesic_lemma_b]{\ref*{subgeodesic_lemma}\ref*{subgeodesic_lemma_b}}, the disjointness between $\gamma_1$ and $\gamma_2$ now implies
\eq{
\MDG^\R_{r,r+\delta}((x-\eps,x+\eps),(y-\eps,y+\eps)) \geq 2, \\
\MDG^\R_{r+\delta,r+2\delta}((y-\eps,y+\eps),(z-\eps,z+\eps)) \geq 2, \\
\MDG^\R_{r+2\delta,r+3\delta}((z-\eps,z+\eps),(w-\eps,w+\eps)) \geq 2.
}
See Figure~\ref{closeness} for an illustration.
\begin{figure}
\centering
\includegraphics[trim=0.5in 0in 0.5in 0in, clip, width=0.5\textwidth]{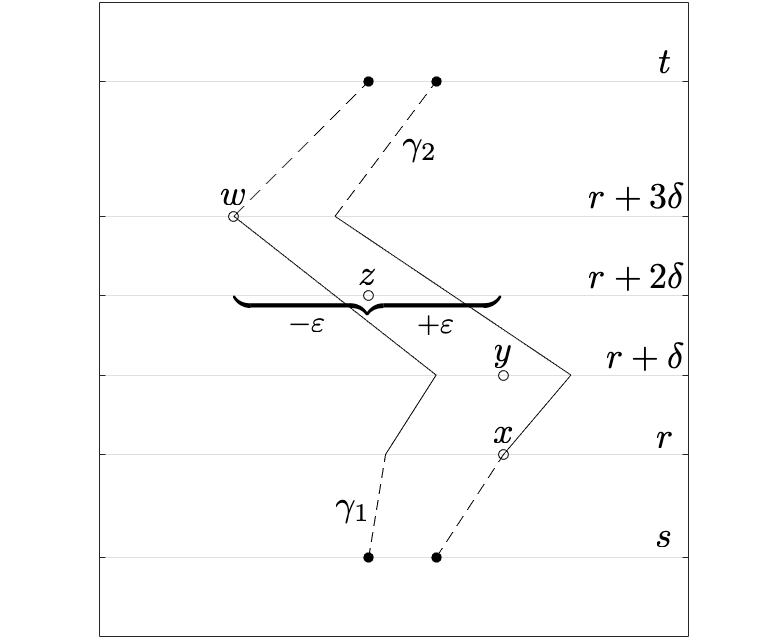}
\caption{Scenario implied by $\BB_\eps$.
Because $s,t\in[-m,m]$ satisfy $t-s \geq 6\delta$, there is some $r\in[-m,m]\cap3\delta\Z$ for which $[r,r+3\delta]$ is a subinterval of $[s,t]$.
This subinterval is divided into three further subintervals, each of which admits the two disjoint subgeodesics arising from $\gamma_1$ and $\gamma_2$. 
(Here, three is the minimum number needed for our analysis to ensure probability $0$, but it could be replaced by any larger integer.)
The points $x,y,z,w\in[-m,m]\cap\eps\Z$ are chosen so that intervals of radius $\eps$ about these points contain both $\gamma_1$ and $\gamma_2$ at the associated times.}
\label{closeness}
\end{figure}
Notice that the three random variables appearing above are independent and identically distributed, since the time intervals $(r,r+\delta)$, $(r+\delta,r+2\delta)$, and $(r+2\delta,r+3\delta)$ are disjoint and have the same length.
Consequently, if we define $\BB_{\eps,x,y,z,w,r}$ to be the intersection of the three events in the previous display, then
\eeq{ \label{local_bad_event}
\P(\BB_{\eps,x,y,z,w,r}) = \P\Big(\MDG^\R_{r,r+\delta}((x-\eps,x+\eps),(y-\eps,y+\eps)) \geq 2\Big)^3,
}
and our discussion has shown
\eeq{ \label{union_bad_event}
\{R\leq m\} \cap \bigcap_{\eps>0}\BB_\eps \subset \bigcap_{k=1}^\infty\, \bigcup_{r\in[-m,m]\cap3\delta\Z}\,\bigcup_{x,y,z,w\in[-m,m]\cap k^{-1} \Z}\BB_{k^{-1},x,y,z,w,r}.
}
Now recall the constant $G$ from Theorem~\ref{disjoint_geo_rarity}.
Since we will soon take $\eps\searrow0$, we may assume $\eps$ to be sufficiently small that the following technical conditions are satisfied:
\eq{
\frac{\eps}{\delta^{2/3}} \leq G^{-16}, \quad 
 \frac{2m}{\delta^{2/3}} < \frac{\eps^{-1/2}}{\delta^{-1/3}}\Big(\log \frac{\delta^{2/3}}{\eps}\Big)^{-2/3}G^{-2}, 
 \quad G^8\exp\Big\{G^2\Big(\log\frac{\delta^{2/3}}{2\eps}\Big)^{5/6}\Big\}\frac{(2\eps)^{3/2}}{\delta} \leq \eps^{11/8}.
}
We can then apply Corollary~\ref{disjoint_2geo_rarity} to obtain
\eeq{ \label{prob_bad_event}
\P\Big(\MDG^\R_{r,r+\delta}((x-\eps,x+\eps),(y-\eps,y+\eps)) \geq 2\Big) \leq 
\eps^{11/8}.
}
Putting together \eqref{local_bad_event}--\eqref{prob_bad_event} now yields the desired result:
\eq{
\P\Big(\{R\leq m\} \cap \bigcap_{\eps>0}\BB_\eps\Big) 
&\leq\P\Big(\bigcap_{k=1}^\infty\,\bigcup_{r\in[-m,m]\cap3\delta\Z}\,\bigcup_{x,y,z,w\in[-m,m]\cap k^{-1} \Z}\BB_{k^{-1},x,y,z,w,r}\Big)  \\
&\leq \limsup_{k\to\infty}\, \P\Big(\bigcup_{r\in[-m,m]\cap3\delta\Z}\,\bigcup_{x,y,z,w\in[-m,m]\cap k^{-1} \Z}\BB_{k^{-1},x,y,z,w,r}\Big) \\
&\leq \limsup_{k\to\infty} \sum_{r\in[-m,m]\cap3\delta\Z}\, \sum_{x,y,z,w\in[-m,m]\cap k^{-1}\Z} k^{-33/8} \\
&\leq \limsup_{k\to\infty}\ (2m(3\delta)^{-1}+2)(2mk+1)^4k^{-33/8} = 0. \qedhere
}
\end{proof}

\section{Proofs of Theorems~\hyperref[main_thm_a]{\ref*{main_thm}\ref*{main_thm_a}} and \hyperref[main_thm_2_a]{\ref*{main_thm_2}\ref*{main_thm_2_a}}} \label{lower_bound}
In this section, we prove Theorems~\hyperref[main_thm_a]{\ref*{main_thm}\ref*{main_thm_a}} and \hyperref[main_thm_2_a]{\ref*{main_thm_2}\ref*{main_thm_2_a}}, which are restated in Proposition~\ref{containment_thm}.
In fact, Proposition~\ref{containment_thm} is 
a stronger result since we consider the relevant sets for all time horizons simultaneously.
In the ``univariate" case of Theorem~\ref{main_thm}, this means we fix the initial time at $0$ but vary the terminal time $t>0$.
More precisely, we consider the set
\eq{ 
\DD_{x_1,x_2;t} \coloneqq \Big\{y \in \R : \exists\, \gamma_1 \in G_{(x_1,0;y,t)}, \gamma_2 \in G_{(x_2,0;y,t)} \text{ such that } \gamma_1(r)<\gamma_2(r) \text{ for all $r\in(0,t)$}\Big\},
}
as well as the measure $\mu_{x_1,x_2;t}$ on $\R$, defined by 
\begin{subequations}
\label{mu_t_def}
\begin{linenomath}\postdisplaypenalty=0
\begin{align}
\mu_{x_1,x_2;t}([y_1,y_2]) = 
\ZZ_{x_1,x_2;t}(y_2) - \ZZ_{x_1,x_2;t}(y_1),
\end{align}
\end{linenomath}
where
\begin{linenomath}\postdisplaypenalty=0
\begin{align}
\ZZ_{x_1,x_2;t}(y) \coloneqq \LL(x_2,0;y,t)-\LL(x_1,0;y,t).
\end{align}
\end{linenomath}
\end{subequations}
In the ``bivariate" case of Theorem~\ref{main_thm_2}, both the initial time $s$ and terminal time $t$ are allowed to vary.
The exceptional set under consideration is
\eq{
\DD_{s;t} \coloneqq \Big\{(x,y) \in \R^2 : \exists\, \gamma_1,\gamma_2 \in G_{(x,s;y,t)} \text{ such that } \gamma_1(r)<\gamma_2(r) \text{ for all $r\in(s,t)$}\Big\},
}
and the relevant measure on $\R^2$ is $\mu_{s;t}$, defined by
\eeq{ \label{mu_st_def}
\mu_{s;t}([x_1,x_2]\times[y_1,y_2]) = \LL(x_2,s;y_2,t)+\LL(x_1,s;y_1,t)-\LL(x_1,s;y_2,t)-\LL(x_2,s;y_1,t).
}
Recall the event $\PP$ from Definition \ref{proper_def}, and the geodesic ordering events $\OO$ and $\OO_x$ from \eqref{ordering_events}.

\begin{prop} \label{containment_thm}
The following statements hold. 
\begin{enumerate}[label=\textup{(\alph*)}]

\item \label{containment_thm_a}
For any $x_1<x_2$, on the almost sure event $\PP\cap\OO_{x_1}\cap\OO_{x_2}$, we have $\Supp(\mu_{x_1,x_2;t}) = \DD_{x_1,x_2;t}$ for all $t>0$.

\item \label{containment_thm_b}
On the almost sure event $\PP\cap\OO$, we have $\Supp(\mu_{s;t}) = \DD_{s;t}$ for all $s<t$.

\end{enumerate}
\end{prop}

We will prove this result in Section \ref{containment_proof} after stating two key lemmas.

\subsection{Results about geodesics with ordered endpoints}
We begin by considering sequences of geodesics whose endpoints converge monotonically. 
For the particular lemma we state below, the time horizon is irrelevant; so let us fix it to be $[0,1]$ and write $G_{x,y} = G_{(x,0;y,1)}$ for sets of geodesics.
Recall Definition~\ref{left_right_def} of leftmost and rightmost geodesics, whose existence is given by the almost sure event $\EE$ from \eqref{existence_event}.
Recall also that $\EE$ is implied by $\PP$, the event from Definition~\ref{proper_def}.

\begin{lemma} \label{limiting_geodesics}
On the event $\PP$, the following statements hold for any $x,y\in\R$.
Let $\gamma^L$ and $\gamma^\rt$ be the leftmost and rightmost geodesics in $G_{x,y}$.
\begin{enumerate}[label=\textup{(\alph*)}]
\item If $x_j\nearrow x$ and $y_j\nearrow y$, then there is a sequence of $\gamma_j\in G_{x_j,y_j}$ so that $\gamma_j\nearrow\gamma^\lt$ uniformly.
\item If $x_j\searrow x$ and $y_j\searrow y$, then there is a sequence of $\gamma_j\in G_{x_j,y_j}$ so that $\gamma_j\nearrow\gamma^\rt$ uniformly.
\end{enumerate}
\end{lemma}

\begin{proof}
The statements (a) and (b) are symmetric to one another, and so we will prove only (a).

Take any sequences $x_j\nearrow x$ and $y_j\nearrow y$.
By the assumed occurrence of $\EE\supset\PP$, there exists a leftmost geodesic $\gamma^\lt\in G_{x,y}$, and $G_{x_j,y_j}$ is nonempty for every $j$.
By successive applications of Lemma~\ref{reordering_lemma}, we may assume that $\gamma_j(t)\leq\gamma_{j+1}(t) \leq \gamma^\lt(t)$ for all $r\in[0,1]$.
Therefore, $\gamma_j(t)$ must converge as $j\to\infty$ to some value we call $\gamma(t)$, which is at most $\gamma^\lt(t)$.
We claim that $\gamma\in G_{x,y}$, and so $\gamma$ is necessarily equal to $\gamma^\lt$.

First we show that $\gamma : [0,1]\to\R$ is continuous.
In particular, the convergence $\gamma_j\nearrow\gamma$ is uniform by Dini's theorem.
Suppose toward a contradiction that $\gamma$ is discontinuous at some $r\in[0,1]$.
We will assume $\gamma$ is right-discontinuous at $r$ (in particular, $r<1$); the case of left-discontinuity is handled in a symmetric fashion.
That is, there exists some $\eps>0$ such that for every $\delta>0$, we have
\eq{
\sup_{t'\in(r,r+\delta]\cap[0,1]} |\gamma(t)-\gamma(r)| \geq 4\eps.
}
In particular, there is a sequence $t_\ell\searrow r$ such that
\eq{
|\gamma(t_\ell) - \gamma(r)| \geq 3\eps \quad \text{for every $\ell$}.
}
From the pointwise convergence $\gamma_j\to\gamma$, we can select indices $j_\ell\nearrow\infty$ such that
\eq{
|\gamma_{j_\ell}(t_\ell) - \gamma(t_\ell)| \leq \eps \quad \text{and} \quad |\gamma_{j_\ell}(r)-\gamma(r)|\leq\eps \quad \text{for every $\ell$}.
}
These choices yield 
\eq{
|\gamma_{j_\ell}(t_\ell) - \gamma_{j_\ell}(r)| \geq \eps \quad \text{for every $\ell$}.
}
In light of Remark~\hyperref[proper_remark_b]{\ref*{proper_remark}\ref*{proper_remark_b}} and the fact that $t_\ell\searrow r$, this last display implies
\eeq{ \label{toward_neg_inf}
\lim_{\ell\to\infty} \LL(\gamma_{j_\ell}(r),r;\gamma_{j_\ell}(t_\ell),t_\ell) = -\infty.
}
On the other hand, Remark~\hyperref[proper_remark_a]{\ref*{proper_remark}\ref*{proper_remark_a}} guarantees
\begin{subequations}
\label{not_inf}
\begin{linenomath}\postdisplaypenalty=0
\begin{align}
\limsup_{\ell\to\infty} |\LL(\gamma_{j_\ell}(t_\ell),t_\ell;y_{j_\ell},t)| < \infty,
\intertext{as well as}
\limsup_{\ell\to\infty} |\LL(\gamma_{j_\ell}(r),r;y_{j_\ell},t)| < \infty.
\end{align}
\end{linenomath}
\end{subequations}
Given that each $\gamma_{j_\ell}$ is a geodesic, we have
\eq{
\LL(\gamma_{j_\ell}(r),r;y_{j_\ell},t) = \LL(\gamma_{j_\ell}(r),r;\gamma_{j_\ell}(t_\ell),t_\ell)+\LL(\gamma_{j_\ell}(t_\ell),t_\ell;y_{j_\ell},t),
}
and so \eqref{toward_neg_inf} is in contradiction with \eqref{not_inf}.
Consequently, $\gamma$ must be continuous on all of $[0,1]$.

To complete the proof, we need to show $\LL(\gamma) = \LL(x,0;y,1)$.
For any partition $0=t_0<t_1<\cdots<t_k=1$, we have 
\eq{
\sum_{i=1}^k \LL(\gamma(t_{i-1}),t_{i-1};\gamma(t_i),t_i) 
&=  \sum_{i=1}^k \lim_{j\to\infty} \LL(\gamma_j(t_{i-1}),t_{i-1};\gamma_j(t_i),t_i) \\
&= \lim_{j\to\infty} \LL(x_j,0;y_j,1) = \LL(x,0;y,1),
}
where the first and last equalities hold by continuity of $\LL$, and the middle equality is valid because each $\gamma_j$ is a geodesic.
Taking an infimum over all partitions, we conclude that $\LL(\gamma)=\LL(x,0;y,1)$.
\end{proof}

We will need one more fact from \cite{basu-ganguly-hammond21} which is stated as the next lemma. 
It links the disjointness of geodesics to the measures $\mu_{x_1,x_2;t}$ and $\mu_{s;t}$.
Although originally proved for Brownian LPP, it needs no revision in its extension to the directed landscape. 
Nevertheless, given the conceptual importance, we recall the ideas of the proof in Figure~\ref{diff_wt_constant}.
Since only part \ref{constant_lemma_a} was explicitly stated in \cite{basu-ganguly-hammond21}, we also point out in Figure~\ref{diff_wt_constant} the equivalence of part \ref{constant_lemma_b}.
Let $\NI_{s,t}(x_1,x_2;y_1,y_2)$ be the event that every $\gamma_1\in G_{(x_1,s;y_1,t)}$ is disjoint from every $\gamma_2\in G_{(x_2,s;y_2,t)}$.
\begin{lemma}
\label{constant_lemma}
\textup{\cite[Lemma 3.6]{basu-ganguly-hammond21}}
The following statements hold on the event $\PP$ from Definition \ref{proper_def}.
\begin{enumerate}[label=\textup{(\alph*)}]

\item \label{constant_lemma_a}
If $\NI_{0,t}(x_1,x_2;y^\lt,y^\rt)$ does not occur for some $y^\lt<y^\rt$, then $(y^\lt,y^\rt)\cap\Supp(\mu_{x_1,x_2;t})$ is empty.

\item \label{constant_lemma_b}
 If $\NI_{s,t}(x^\lt,x^\rt;y^\lt,y^\rt)$ does not occur for some $x^\lt<x^\rt$ and $y^\lt<y^\rt$, then $\big((x^\lt,x^\rt)\times(y^\lt,y^\rt)\big)\cap\Supp(\mu_{s;t})$ is empty.

\end{enumerate}
\end{lemma}

\begin{figure}
\centering
\includegraphics[trim=0.9in 0.7in 0.5in 0.5in, clip, width=0.5\textwidth]{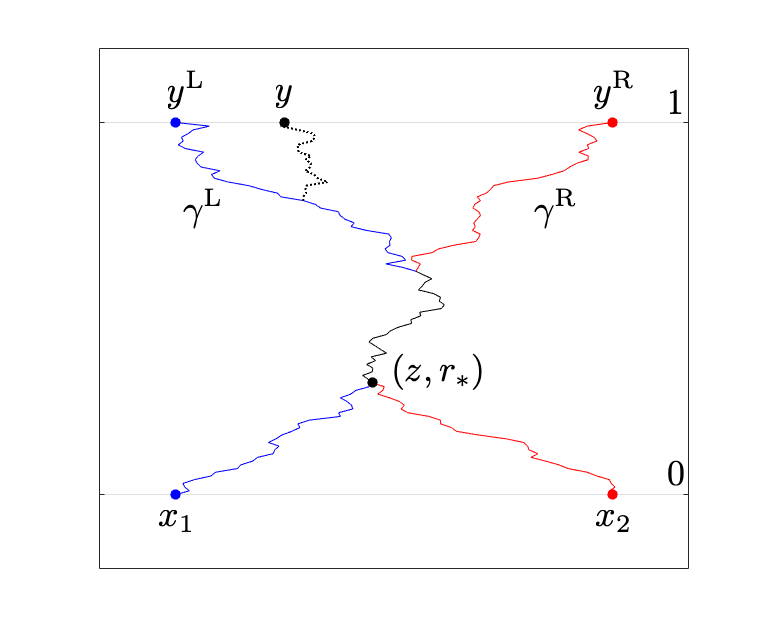}
\caption{Proof sketch of Lemma~\ref{constant_lemma} with $s=0$ and $t=1$.
In the above diagram, $\gamma^\lt\in G_{x_1,y^\lt}$ and $\gamma^\rt\in G_{x_2,y^\rt}$.
It is assumed that these two geodesics intersect so that $\NI_{0,1}(x_1,x_2,y^\lt,y^\rt)$ does not occur, and $(z,r_*)$ is their lowest point of intersection.
If $y\in(y^\lt,y^\rt)$, then a combination of Lemmas~\ref{new_geodesics} and \hyperref[reordering_lemma_c]{\ref*{reordering_lemma}\ref*{reordering_lemma_c}} yields geodesics $\gamma_1\in G_{x_1,y}$ and $\gamma_2\in G_{x_2,y}$ which agree at all points above and including $(z,r_*)$, and then follow $\gamma^\lt$ and $\gamma^\rt$ respectively below time $r_*$.
As this can be done for every $y\in(y^\lt,y^\rt)$, we conclude that $\ZZ_{x_1,x_2}$ is constant on this interval (recall from Figure~\ref{diff_wt} how the value of $\ZZ_{x_1,x_2}(y)$ can be determined), the precise constant being $\LL(x_2,0;z,r_*)-\LL(x_1,0;z,r_*)$.
From definitions \eqref{mu_1var_def} and \eqref{mu_2var_def}, this implies $\mu_{x_1,x_2}([y_1,y_2]) = \mu([x_1,x_2]\times[y_1,y_2]) = 0$, meaning $y\notin\Supp(\mu_{x_1,x_2})$ and $(x,y)\notin\Supp(\mu)$ for any $x\in(x_1,x_2)$.}
\label{diff_wt_constant}
\end{figure}

\subsection{Proving equality of measure supports and exceptional sets} \label{containment_proof}
We are now ready to prove Proposition~\ref{containment_thm}, the central result of Section \ref{lower_bound}.

\begin{proof}[Proof of Proposition~\ref{containment_thm}]
We first prove that $\R\setminus\DD_{x_1,x_2;t} \subset \R\setminus\Supp(\mu_{x_1,x_2;t})$.
Here we need only to assume the occurrence of $\PP$.
Consider any $y\in\R\setminus\DD_{x_1,x_2;t}$.
Let $\gamma^\lt$ and $\gamma^\rt$ be the leftmost and rightmost geodesics in $G_{(x_1,0;y,t)}$ and $G_{(x_2,0;y,t)}$, respectively.
As $y\notin\DD_{x_1,x_2;t}$, there must be some $r_*\in(0,t)$ such that $\gamma^\lt(r_*)=\gamma^\rt(r_*)$.

Next take any strictly monotonic sequences $y_j^\lt\nearrow y$ and $y_j^\rt\searrow y$, along with geodesics $\gamma_j^\lt\in G_{(x_1,0;y_j^\lt,t)}$ and $\gamma_j^\rt\in G_{(x_2,0;y_j^\rt,t)}$ guaranteed by Lemma~\ref{limiting_geodesics}; see Figure~\ref{pre_1var}.
\begin{figure}[]
\centering
\subfloat[{$\gamma_j^\lt$ and $\gamma_j^\rt$ may be initially disjoint}]{
\includegraphics[trim=0.5in 0.5in 0.5in 0.5in, clip, width=0.48\textwidth]{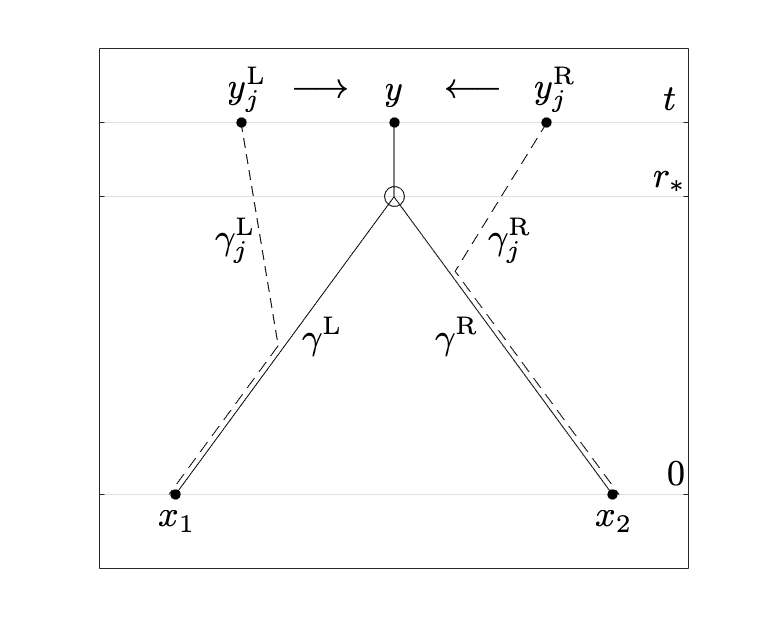}
\label{pre_1var}
}
\hfill
\subfloat[\, {\parbox[c][][c]{1.6in}{$j$ sufficiently large \newline $\implies$ intersections above $r_*$ \newline $\implies$ intersection at time $r_*$}}]{
\includegraphics[trim=0.5in 0.5in 0.5in 0.5in, clip, width=0.48\textwidth]{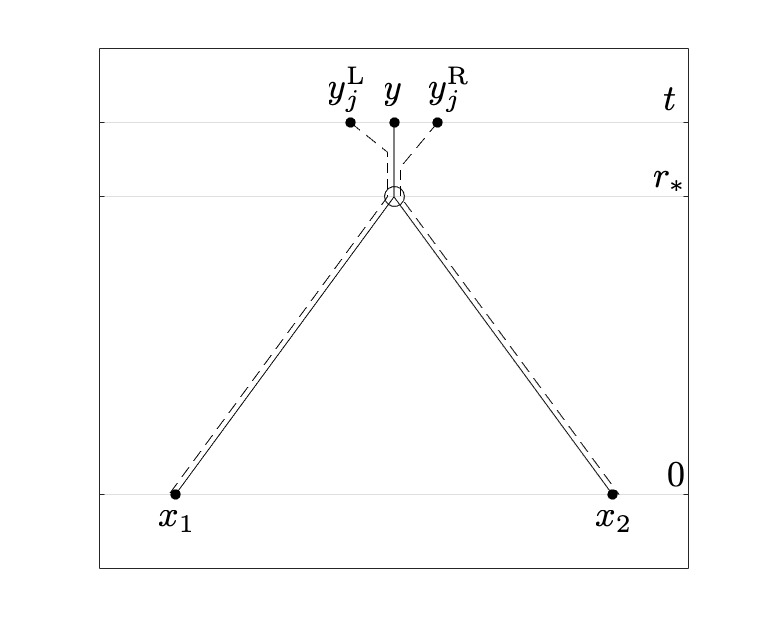}
\label{post_1var}	
}
\\
\subfloat[{$\gamma_j^\lt$ and $\gamma_j^\rt$ may be initially disjoint}]{
\includegraphics[trim=0.5in 0.5in 0.5in 0.5in, clip, width=0.48\textwidth]{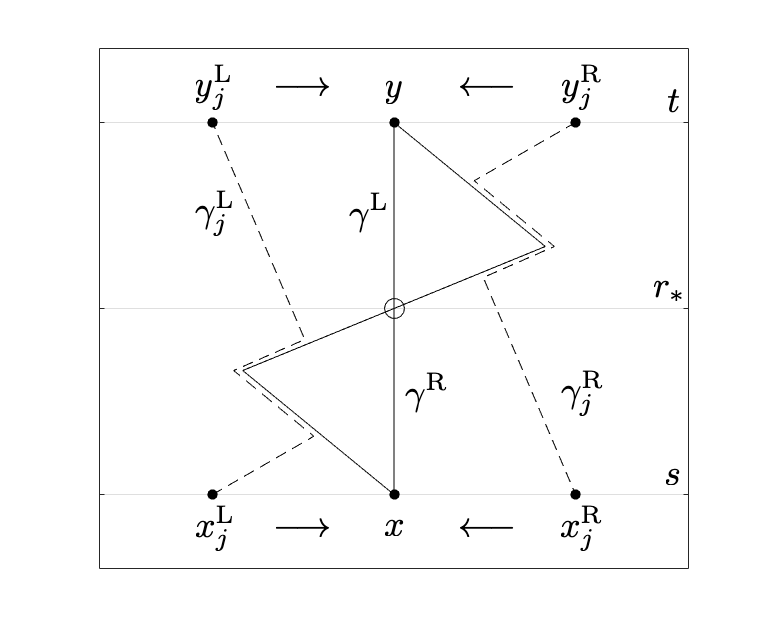}
\label{pre_2var}
}
\hfill
\subfloat[\, {\parbox[c][][c]{2in}{$j$ sufficiently large \newline $\implies$ intersections above/below $r_*$ \newline $\implies$ intersection at time $r_*$}}]{
\includegraphics[trim=0.5in 0.5in 0.5in 0.5in, clip, width=0.48\textwidth]{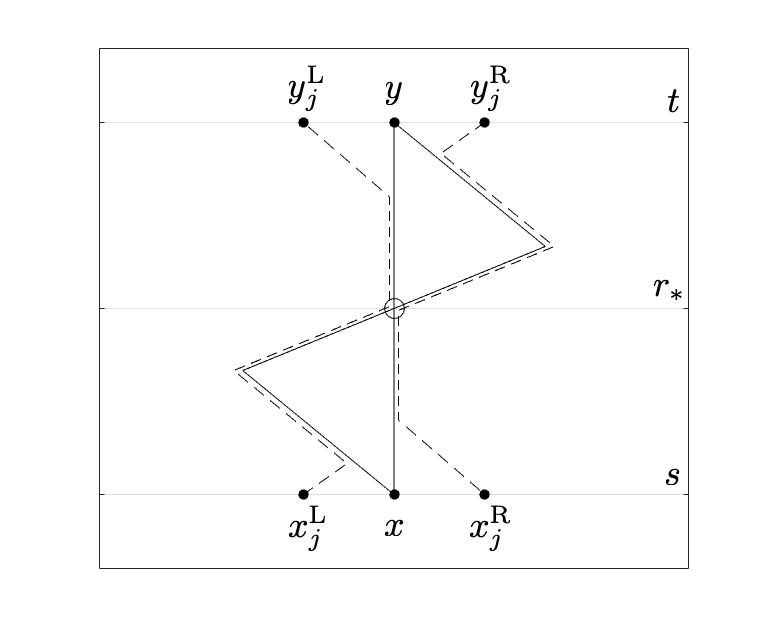}
\label{post_2var}	
}
\caption{Geodesics considered in the proof of Proposition~\ref{containment_thm}.
Diagrams (a) and (b) illustrate the argument for $\Supp(\mu_{x_1,x_2;t})\subset\DD_{x_1,x_2;t}$, while (c) and (d) illustrate the argument for $\Supp(\mu_{s;t})\subset\DD_{s;t}$.
In each case, there is assumed to be some intermediate time $r_*$ at which $\gamma^\lt$ and $\gamma^\rt$ intersect;
the point of intersection is marked with an open circle.
If $\gamma_j^\lt$ intersects $\gamma^\lt$ in (a), then the two geodesics coincide at all lower times; similarly for $\gamma_j^\rt$ with $\gamma^\rt$.
When $j\to\infty$, both pairs must experience intersections above time $r_*$, thereby forcing an intersection between $\gamma_j^\lt$ and $\gamma_j^\rt$ at time $r_*$, where $\gamma^\lt$ and $\gamma^\rt$ agree.
A slightly different argument is needed for the scenario in (c). 
If $\gamma_j^\lt$ intersects $\gamma^\lt$ at \textit{two} distinct times, then the two geodesics coincide at all times in between; similarly for $\gamma_j^\rt$ with $\gamma^\rt$.
For both pairs, sending $j\to\infty$ forces at least one intersection before time $r_*$ and one after time $r_*$.
Hence a common intersection eventually appears at $r_*$.
}
\label{non_trivial_intersections}
\end{figure}
That is, $\gamma_j^\lt\nearrow\gamma^\lt$ and $\gamma_j^\rt\searrow\gamma^\rt$, uniformly as $j\to\infty$.
So for any $\eps>0$, we have the following for all $j$ sufficiently large:
\eq{
\|\gamma_j^\lt-\gamma^\lt\|_\infty < \eps. 
}
In particular, we can choose some $j$ for which
\eq{
|\gamma_j^\lt(r)-\gamma^\lt(r)| < \eps 
\quad \text{for all $r\in[r_*,t]$}.
}
Given Theorem~\ref{no_arbitrary_closeness}---applied to a \textit{random} compact set $K\subset\R^4_\uparrow$ containing both $(\gamma^\lt(r_*),r_*;y,t)$ and $(\gamma_j^\lt(r_*),r_*;y_j^\lt,t)$ for every $j$---we can take $\eps$ to be so small that the above display implies
$\gamma_j^\lt(\tau) = \gamma^\lt(\tau)$ for some $\tau\in[r_*,t]$. 
We now claim that 
\eeq{ \label{remain_touching_left}
\gamma_j^\lt(r) = \gamma^\lt(r) \quad \text{for all $r\in[0,\tau]$}.
}
Indeed, Lemma~\hyperref[new_geodesics_3]{\ref*{new_geodesics}\ref*{new_geodesics_3}} shows that $\gamma_j^\lt\big|_{[0,\tau]}\oplus\gamma^\lt\big|_{[\tau,t]}$ belongs to $G_{(x_1,0;y,t)}$.
If \eqref{remain_touching_left} were false, then because we already know $\gamma_j \leq \gamma^\lt$ by construction, we would have $\gamma_j^\lt(r) < \gamma^\lt(r)$ for some $r\in(0,\tau)$, 
in which case the previous sentence would contradict the choice of  $\gamma^\lt$ as the leftmost element of $G_{(x_1,0;y,t)}$.
Therefore, our claim \eqref{remain_touching_left} is true; in particular, we have $\gamma_j^\lt(r_*)=\gamma^\lt(r_*)$ for all large $j$.

By the same logic, we also have $\gamma_j^\rt(r_*) = \gamma^\rt(r_*)$ for all large enough $j$.
Since $\gamma^\lt(r_*)=\gamma^\rt(r_*)$ by assumption, we now see that $\gamma_j^\lt(r_*) = \gamma_j^\rt(r_*)$ for all large $j$, as shown in Figure~\ref{post_1var}.
In particular, there is some $j$ for which $\NI_{0,t}(x_1,x_2,y^\lt_j,y^\rt_j)$ does not occur, and so $y$ does not belong to $\Supp(\mu_{x_1,x_2;t})$ by Lemma~\hyperref[constant_lemma_a]{\ref*{constant_lemma}\ref*{constant_lemma_a}}.

The argument to show $\R^2\setminus\DD_{s;t}\subset\R^2\setminus\Supp(\mu_{s;t}) $ will be similar.
Consider any $(x,y)\in\R^2\setminus\DD_{s;t}$.
Let $\gamma^\lt$ and $\gamma^\rt$ be the leftmost and rightmost geodesics between $(x,s)$ and $(y,t)$.
As before, there must be some $r_*\in(s,t)$ such that $\gamma^\lt(r_*)=\gamma^\rt(r_*)$.
Using Lemma~\ref{limiting_geodesics} once more, we take strictly monotonic sequences $x_j^\lt\nearrow x$, $x_j^\rt\searrow x$, $y^\lt_j\nearrow y$, $y^\rt_j\searrow y$, and consider geodesics $\gamma_j^\lt\in G_{(x_j^\lt,s;y_j^\lt,t)}$ and
$\gamma_j^\rt\in G_{(x_j^\rt,s;y_j^\rt,t)}$ such that $\gamma_j^\lt \nearrow \gamma^\lt$ uniformly and $\gamma_j^\rt\searrow\gamma^\rt$ uniformly.
An illustration is provided in Figure~\ref{pre_2var}.

The same argument as the one leading to \eqref{remain_touching_left}---but now using Lemma~\hyperref[new_geodesics_1]{\ref*{new_geodesics}\ref*{new_geodesics_1}}---tells us that if $\gamma_j^\lt$ intersects $\gamma^\lt$ at \textit{two} distinct times, then the geodesics must agree at all intermediate times.
That is, for any $s < \sigma < \tau <t$, we have the following implication:
\eeq{ \label{remain_touching_2}
\gamma_j^\lt(\sigma) = \gamma^\lt(\sigma),\, \gamma_j^\lt(\tau) = \gamma^\lt(\tau) \quad \implies \quad 
\gamma_j^\lt(r) = \gamma^\lt(r) \quad \text{for all $r\in[\sigma,\tau]$,}
}
Meanwhile, by invoking Theorem~\ref{no_arbitrary_closeness} twice (once for each of the two intervals $[s,r_*]$ and $[r_*,t]$), we can conclude that for all $j$ sufficiently large,
the hypothesis of \eqref{remain_touching_2} is satisfied by some $\sigma\in[s,r_*]$ and some $\tau\in[r_*,t]$.
By the conclusion of \eqref{remain_touching_2}, it follows that $\gamma_j^\lt(r_*) = \gamma^\lt(r_*)$; see Figure~\ref{post_2var}.
By analogous reasoning, we also have $\gamma^\rt_j(r_*)=\gamma^\rt(r_*)$ for all large $j$.
Since $\gamma^\lt(r_*)=\gamma^\rt(r_*)$, it follows that $\gamma^\rt_j(r_*)=\gamma^\lt_j(r_*)$ for some large $j$.
Now  Lemma~\hyperref[constant_lemma_b]{\ref*{constant_lemma}\ref*{constant_lemma_b}} gives the desired conclusion: $(x,y)$ is not an element of $\Supp(\mu_{s;t} )$.

\begin{figure}[]
\centering
\subfloat[Crossing geodesics $\gamma_a$ and $\gamma_b$]{
\includegraphics[trim=0.5in 0.5in 0.5in 0.5in, clip, width=0.48\textwidth]{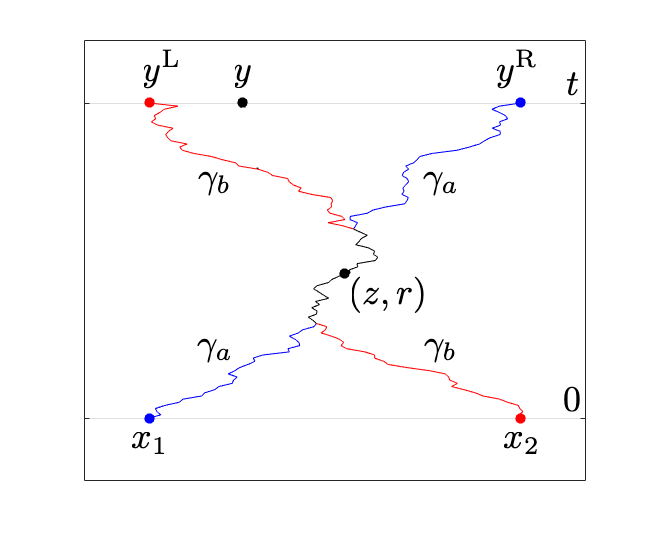}
\label{cross_1}
}
\hfill
\subfloat[Identifying $\gamma^\lt$ and $\gamma^\rt$]{
\includegraphics[trim=0.5in 0.5in 0.5in 0.5in, clip, width=0.48\textwidth]{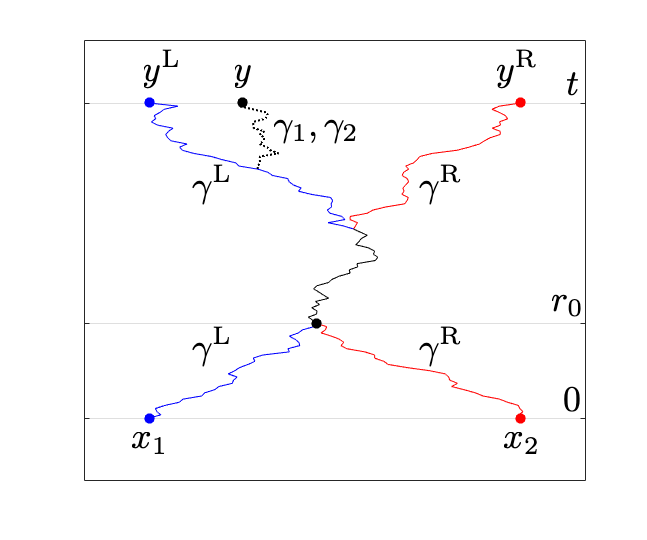}
\label{cross_2}
}
\caption{Geodesics considered in the proof of Proposition~\ref{containment_thm} for showing $\DD_{x_1,x_2;t}\subset\Supp(\mu_{x_1,x_2;t})$.
In (a), the geodesics $\gamma_a: (x_1,0)\to(y^\rt,t)$ and $\gamma_b:(x_2,0)\to(y^\lt,t)$ must intersect by planarity; the point $(z,r)$ can be any intersection point.
In (b), it is assumed that $y\notin\Supp(\mu_{x_1,x_2;t})$, which implies that the paths $\gamma^\lt : (x_1,0)\to(y^\lt,t)$ and $\gamma^\rt: (x_2,0)\to(y^\rt,t)$, formed by concatenating the relevant portions of $\gamma_a$ and $\gamma_b$, are geodesics.
Any geodesic $\gamma_1: (x_1,0)\to(y,t)$ cannot pass to the left of $\gamma^\lt$, while any geodesic $\gamma_2:(x_2,0)\to(y,t)$ cannot pass to the right of $\gamma^\rt$.
Shown here is the typical situation in which $\gamma_1$ and $\gamma_2$ initially coincide with $\gamma^\lt$ and $\gamma^\rt$, respectively.
Hence $\gamma_1$ and $\gamma_2$ will coalesce at the same time $r_0$ at which $\gamma^\lt$ and $\gamma^\rt$ first intersect, although in principle $\gamma_1$ and $\gamma_2$ may intersect earlier.
}
\label{cross_fig}
\end{figure}

Now we prove the reverse containments.
First we show $\R\setminus\Supp(\mu_{x_1,x_2;t})\subset\R\setminus\DD_{x_1,x_2;t}$, for which we assume that $\PP\cap\OO_{x_1}\cap\OO_{x_2}$ occurs.
So take any $y\in\R\setminus\Supp(\mu_{x_1,x_2;t})$; by definition \eqref{mu_t_def}, this means $\ZZ_{x_1,x_2;t}$ is constant on an open interval $(y^\lt,y^\rt)$ containing $y$.
Now take any geodesics $\gamma_a \in G_{(x_1,0;y^\rt,t)}$ and $\gamma_b \in G_{(x_2,0;y^\lt,t)}$.
By planarity, $\gamma_a$ and $\gamma_b$ must intersect at some time $r\in(0,t)$.
That is, $\gamma_a(r)=\gamma_b(r) = z$ for some $z\in\R$, as in illustrated in Figure~\ref{cross_1}.
By applying Lemma \ref{reordering_lemma} separately on the intervals $[0,r]$ and $[r,t]$, we may assume that $\gamma_a\big|_{[0,r]}\leq\gamma_b\big|_{[0,r]}$ and $\gamma_b\big|_{[r,t]}\leq\gamma_a\big|_{[r,t]}$; here we are using Lemma~\hyperref[subgeodesic_lemma_b]{\ref*{subgeodesic_lemma}\ref*{subgeodesic_lemma_b}} to ensure these subpaths are geodesics.
Meanwhile, Lemma~\hyperref[subgeodesic_lemma_a]{\ref*{subgeodesic_lemma}\ref*{subgeodesic_lemma_a}} allows us to write
\eeq{ \label{crossing_step}
0 &= \ZZ_{x_1,x_2;t}(y^\rt)-\ZZ_{x_1,x_2;t}(y^\lt) \\
&= \LL(x_2,0;y^\rt,t) - \LL(\gamma_a) - \LL(\gamma_b) + \LL(x_1,0;y^\lt,t) \\
&= [\LL(x_2,0;y^\rt,t) - \LL(x_2,0;z,r) - \LL(z,r;y^\rt,t)] + [\LL(x_1,0;y^\lt,t) - \LL(x_1,0;z,r) - \LL(z,r;y^\lt,t)]. \raisetag{2\baselineskip}
}
By \eqref{max_prop}, each bracketed sum in the final line of \eqref{crossing_step} is nonnegative and so must actually be equal to zero.
Hence the composition $\gamma^\lt \coloneqq \gamma_a\big|_{[0,r]}\oplus\gamma_b\big|_{[r,t]}$ is an element of $G_{(x_1,0;y^\lt,t)}$, while $\gamma^\rt\coloneqq\gamma_b\big|_{[0,r]}\oplus\gamma_a\big|_{[r,t]}$ is an element of $G_{(x_2,0;y^\rt,t)}$.
To complete the argument, we now consider any $\gamma_1\in G_{(x_1,0;y,t)}$ and $\gamma_2\in G_{(x_2,0;y,t)}$.
The occurrence of $\OO_{x_1}$ forces $\gamma^\lt\leq\gamma_1$, while the occurrence of $\OO_{x_2}$ implies $\gamma_2\leq \gamma^\rt$.
Since $\gamma^\lt\leq\gamma^\rt$ by our assumptions on $\gamma_a$ and $\gamma_b$, yet $\gamma^\lt(r)=\gamma^\rt(r)$, we are left to conclude that $\gamma_1(r) = \gamma_2(r)$, as shown in Figure \ref{cross_2}.
As this conclusion is valid for every choice of $\gamma_1$ and $\gamma_2$, we have shown that $y$ does not belong to $\DD_{x_1,x_2;t}$.

The argument for $\R\setminus\Supp(\mu_{s;t})\subset\R\setminus\DD_{s;t}$ is similar.
Assume the occurrence $\PP\cap\OO$.
Consider any $y\in\R\setminus\Supp(\mu_{s;t})$; by definition \eqref{mu_st_def}, this means there are points $x^\lt,x^\rt,y^\lt,y^\rt$ such that $x^\lt < x < x^\rt$, $y^\lt < y < y^\rt$, and 
\eq{
\LL(x^\rt,s;y^\lt,t) - \LL(x^\lt,s;y^\lt,t)
= \LL(x^\rt,s;y^\rt,t) - \LL(x^\lt,s;y^\rt,t).
}
Proceeding as above (but now invoking the event $\OO$ instead of $\OO_{x_1}\cap\OO_{x_2}$), we can find $\gamma^\lt\in G_{(x^\lt,s;y^\lt,t)}$ and $\gamma^\rt \in G_{(x^\rt,s;y^\rt,t)}$ such that $\gamma^\lt(r)=\gamma^\rt(r)$ for some $r\in(s,t)$, and $\gamma^\lt\leq\gamma\leq\gamma^\rt$ for any $\gamma\in G_{(x,s;y,t)}$.
It follows that any two elements of $G_{(x,s;y,t)}$ must intersect at time $r$; in particular, $(x,y)$ is not an element of $\DD_{s;t}$.
\end{proof}

\section{Proofs of Theorems~\hyperref[main_thm_b]{\ref*{main_thm}\ref*{main_thm_b}} and \hyperref[main_thm_2_b]{\ref*{main_thm_2}\ref*{main_thm_2_b}}} \label{upper_bound}

Here we return to the setting of a single time horizon $[s,t] = [0,1]$.
Recall from Proposition~\hyperref[containment_thm_a]{\ref*{containment_thm}\ref*{containment_thm_a}} that the exceptional set $\DD_{x_1,x_2}$ from \eqref{D_def} is almost surely equal to the support of the measure $\mu_{x_1,x_2}$ from \eqref{mu_1var_def}.
Therefore, the following input from \cite{basu-ganguly-hammond21} is equivalent to Theorem~\hyperref[main_thm_b]{\ref*{main_thm}\ref*{main_thm_b}}.

\begin{thm}
\label{LV_thm}
\textup{\cite[Thm.~1.1]{basu-ganguly-hammond21}}
For any $x_1<x_2$, the Hausdorff dimension of $\Supp(\mu_{x_1,x_2})$ is equal to $\frac{1}{2}$ almost surely.
\end{thm}

For completeness, we note that Theorem~\ref{LV_thm} was proved in \cite{basu-ganguly-hammond21} for the case $x_1=-1$, $x_2=1$.
But the proof there is not specific to these coordinates or the time horizon $[0,1]$.
Moreover, if one knows the result to be generalizable in one respect (either allowing $x_1$ and $x_2$ to be arbitrary, or allowing $s<t$ to be arbitrary), then it must be generalizable in the other respect, by rescaling via \eqref{KPZ_scaling}.

In order to establish Theorem~\hyperref[main_thm_2_b]{\ref*{main_thm_2}\ref*{main_thm_2_b}}, we 
handle the lower and upper bounds separately.
Namely, the lower bound is stated in terms of the measure $\mu$ from \eqref{mu_2var_def}, while the upper bound is given directly for the set $\DD$ from \eqref{D2_def}.
In light of Proposition~\hyperref[containment_thm_b]{\ref*{containment_thm}\ref*{containment_thm_b}}, the following two statements are together equivalent to Theorem~\hyperref[main_thm_2_b]{\ref*{main_thm_2}\ref*{main_thm_2_b}}.

\begin{prop}
\label{supp_thm}
The Hausdorff dimension of $\Supp(\mu)$ is at least $\frac{1}{2}$ almost surely.
\end{prop}

\begin{prop} \label{upper_bd_thm}
The Hausdorff dimension of $\DD$ is at most $\frac{1}{2}$ almost surely.
\end{prop}

We establish Propositions~\ref{supp_thm} and \ref{upper_bd_thm} in Sections~\ref{supp_thm_proof} and \ref{upper_bound_proof}, respectively.
Before proceeding, let us formally define Hausdorff dimension.
Recall that
for $d\in[0,\infty)$, the $d$-dimensional \textit{Hausdorff content} of a metric space $\XX$ is
\eeq{ \label{H_content_def}
H^d(\XX) \coloneqq \inf\bigg\{\sum_i\diam(U_i)^d : \{U_i\} \text{ is a countable cover of $\XX$}\bigg\}.
}
The \textit{Hausdorff dimension} of $\XX$ is
\eeq{ \label{H_dim_def}
d_H(\XX) \coloneqq \inf
\{d \geq 0 : H^d(\XX) = 0\}.
}

\subsection{Proof of dimension lower bound in bivariate case} \label{supp_thm_proof}
Here we prove that $d_H(\Supp(\mu))\geq\frac{1}{2}$ almost surely.

\begin{proof}[Proof of Proposition~\ref{supp_thm}]
By Theorem~\ref{LV_thm}, it suffices to show the following inequality:
\eeq{ \label{dim_supp_2show}
d_H(\Supp(\mu_{x_1,x_2})) \leq d_H(\Supp(\mu)) \quad \text{for any $x_1<x_2$}.
}
To this end, we first prove that
\eeq{ \label{containment_implication}
y \in \Supp(\mu_{x_1,x_2}) \quad 
&\implies \quad ([x_1,x_2]\times\{y\})\cap\Supp(\mu) \neq \varnothing.
}
Indeed, let us check the contrapositive.

If $(x,y) \notin \Supp(\mu)$, then there exists an open ball $B_\eps(x,y)\subset\R^2$ of radius $\eps>0$ and centered at $(x,y)$, such that $\mu(B_\eps(x,y)) = 0$.
For each $x$, define
\eq{
\eps_x \coloneqq \sup\{\eps>0 : \mu(B_\eps(x,y))=0\} \vee 0.
}
The map $x\mapsto \eps_x$ is continuous, in fact with Lipschitz constant $1$, as seen from the following chain of implications (if $\eps\leq 0$, then $B_\eps(\cdot,\cdot)$ is taken to be the empty set):
\eq{
&\mu(B_{\eps_x-\eps}(x,y)) = 0 < \mu(B_{\eps_x+\eps}(x,y)) \quad \forall\ \eps>0 \\
\implies\quad &\mu(B_{\eps_x-\delta-\eps}(x',y)) = 0 < \mu(B_{\eps_x+\delta+\eps}(x',y)) \quad \forall\ x'\in[x-\delta,x+\delta],\, \delta>0,\, \eps>0\\
\implies\quad &\eps_{x'}\in[\eps_x-\delta-\eps,\eps_x+\delta+\eps] \quad \forall\ x'\in[x-\delta,x+\delta],\, \delta>0,\, \eps>0 \\
\implies\quad &\eps_{x'}\in[\eps_x-\delta,\eps_x+\delta] \quad \forall\ x'\in[x-\delta,x+\delta],\, \delta>0.
}
Now, if $([x_1,x_2]\times\{y\})\cap\Supp(\mu) = \varnothing$, then $\eps_x>0$ for every $x\in[x_1,x_2]$.
By the continuity just observed, there is then some $\eps>0$ such that $\eps_x\geq \eps$ for all $x\in[x_1,x_2]$.
Consequently, $\mu([x_1,x_2]\times(y-\eps,y+\eps)) = 0$, which means $y\notin\Supp(\mu_{x_1,x_2})$.
We have now proved \eqref{containment_implication}.

Now suppose $\{U_i\}$ is a countable cover of $\Supp(\mu)$.
For each $i$, let
\eq{
\wt U_i \coloneqq \{y\in\R : ([x_1,x_2]\times\{y\}) \cap U_i \neq \varnothing\}.
}
By \eqref{containment_implication}, $\{\wt U_i\}$ is a cover of $\Supp(\mu_{x_1,x_2})$.
Furthermore, it is trivial that $\diam(\wt U_i) \leq \diam(U_i)$.
From the definitions \eqref{H_content_def} and \eqref{H_dim_def} of Hausdorff content and Hausdorff dimension, the inequality \eqref{dim_supp_2show} immediately follows.
\end{proof}


\subsection{Proof of dimension upper bound in bivariate case} \label{upper_bound_proof}
In this section, we prove the matching upper bound for $d_H(\DD)$. 



\subsubsection{Step 1: Reduce to bounded sets} \label{step_1}
Suppose we can show the following for any $R>0$.

\begin{claim} \label{upper_bd_R}
We almost surely have $d_H(\DD\cap[-R,R]^2) \leq \frac{1}{2}$.
\end{claim}

Proposition~\ref{upper_bd_thm} then immediately follows by taking a countable sequence $R_j\nearrow\infty$ and using the fact that if $\XX\subset\R^2$ satisfies $d_H(\XX\cap[-R,R]^2) \leq d$ for every $R$, then $d_H(\XX) \leq d$.
So let us fix the value of $R$ and aim simply to prove Claim~\ref{upper_bd_R}.

\subsubsection{Step 2: Relate the exceptional sets to pairs of disjoint geodesics}


Recall the event $\EE$ 
from \eqref{existence_event} 
guaranteeing geodesic existence. 
Let $\MDG(A,B) = \MDG_{0,1}^\R(A,B)$ denote the maximum size of a collection of disjoint geodesics whose endpoints lie in $A\times\{0\}$ and $B\times\{1\}$,
and consider the event
%
%
%
\eeq{ \label{event_definitions}
\WW_{z,w}^\eps &\coloneqq \{\MDG((z-\eps,z+\eps),(w-\eps,w+\eps))\geq2\}, \quad z,w\in\R,\, \eps>0. 
}

\begin{claim} \label{set_containments} 

On the event $\EE$, we have
\eq{ 
\big\{\DD \cap \big([z,z+\eps)\times[w,w+\eps)\big) \neq \varnothing\big\}\subset\WW^\eps_{z,w} \quad \text{for all $z,w\in\R$, $\eps>0$}.
}

%
\end{claim}

\begin{proof}
The following argument is illustrated in Figure~\ref{upper_bound_fig}.
Suppose $(x,y)\in\DD \cap \big([z,z+\eps)\times[w,w+\eps)\big)$.
That is, there are $\gamma_1^*,\gamma_2^*\in G_{x,y}$ such that $\gamma_1(r)<\gamma_2(r)$ for all $r\in(0,1)$.
Take any $x_1\in(x-\eps,z)$, $y_1\in(y-\eps,w)$, and set $x_2 = x_1+\eps$, $y_2 = y_1+\eps$.
We then have
\eq{
z-\eps<x_1<z<x_2<z+\eps \quad \text{and} \quad w-\eps<y_1<y<y_2<w+\eps.
}
By Lemma~\ref{reordering_lemma} and the assumed occurrence of $\EE$, there are $\gamma_1\in G_{x_1,y_1}$ and $\gamma_2\in G_{x_2,y_2}$ such that
\eq{ 
\gamma_1(r) \leq \gamma_1^*(r) < \gamma_2^*(r) \leq \gamma_2(r) \quad \text{for all $r\in(0,1)$}.
}
Of course, we also know $\gamma_1(0)=x_1<x_2=\gamma_2(0)$ and $\gamma_1(1)=y_1<y_2=\gamma_2(1)$, and so $\gamma_1$ and $\gamma_2$ are disjoint.
By our choice of endpoints, $\WW^\eps_{z,w}$ has occurred.
%
\end{proof}

\begin{figure}[]
\centering
\includegraphics[trim=0.5in 0.5in 0.5in 0.5in, clip, width=0.48\textwidth]{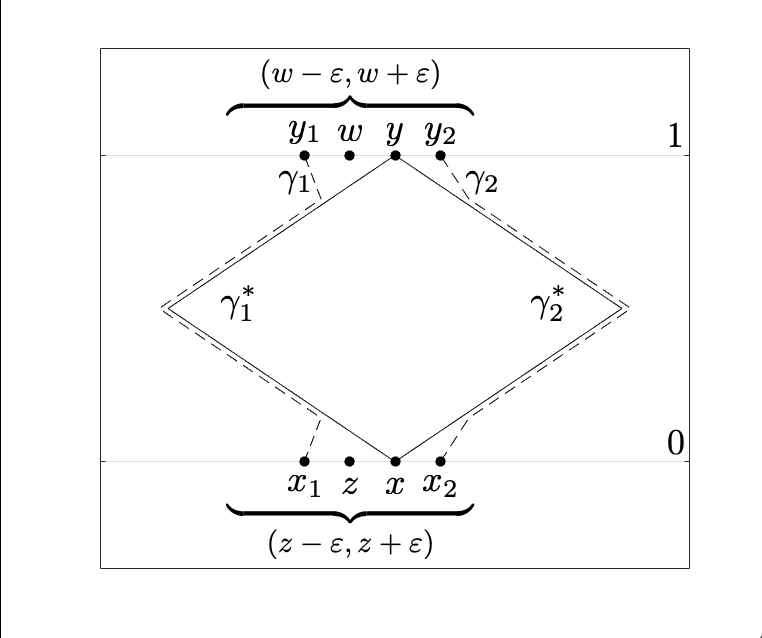}
\caption{Geodesics considered in the proof of Claim~\ref{set_containments}, when the pair $(x,y)$ belongs to $\DD\cap\big([z,z+\eps)\times[w,w+\eps)\big)$. The
disjointness of the solid geodesics implies the disjointness of the dashed geodesics.}
\label{upper_bound_fig}
\end{figure}

\subsubsection{Step 3: Use tail estimate to deduce dimension upper bound} \label{step_3}
Suppose we can show the following for any $d>\frac{1}{2}$.

\begin{claim} \label{upper_bd_d}
We almost surely have $H^d(\DD\cap[-R,R]^2)=0$. 
\end{claim}

Claim~\ref{upper_bd_R} immediately follows by taking a countable sequence $d_j\searrow\frac{1}{2}$.
Therefore, let us fix $d>\frac{1}{2}$ and complete the proof of Proposition~\ref{upper_bd_thm} by verifying Claim~\ref{upper_bd_d}.
Let $\eta \coloneqq (2d-1)/6 > 0$.

\begin{proof}[Proof of Claim~\ref{upper_bd_d}]
Let $G$ be the constant from Theorem~\ref{disjoint_geo_rarity}, and choose $\eps'\in(0,1]$ sufficiently small that the following inequalities hold for all $\eps\in(0,\eps']$:
\begin{subequations}
\label{eps_assumptions}
\begin{linenomath}\postdisplaypenalty=0
\begin{align}
\eps \leq G^{-16}, \quad  2R \leq \eps^{-1/2}\Big(\log \frac{1}{\eps}\Big)^{-2/3}G^{-2}, \label{eps_assumption_1} \\
G^{8}\exp\Big\{G^2\Big(\log\frac{1}{\eps}\Big)^{5/6}\Big\}\eps^{3/2} \leq \eps^{3/2-\eta}, \label{eps_assumption_2} \\
\eps^\eta \leq \frac{1}{2(2R+1)^2}. \label{eps_assumption_3}
\end{align} 
\end{linenomath}
\end{subequations}
The assumption \eqref{eps_assumption_1} is technical and allows us to apply Corollary~\ref{disjoint_2geo_rarity} whenever the relevant spatial coordinates belong to $[-R,R]$, and \eqref{eps_assumption_2} merely makes the resulting estimate easier to write:
\eeq{ \label{single_W}
\P(\WW^\eps_{z,w}) \leq \eps^{3/2-\eta} \quad \text{for all $z,w\in[-R,R]$, $\eps\in(0,\eps']$}.
}
Now take any summable sequence $\delta_j\searrow0$. 
Because $d-1/2-2\eta=\eta>0$, we can subsequently choose a sequence $\eps_j\searrow0$ such that 
\eeq{ \label{eps_sequence_choice}
\lim_{j\to\infty} \eps_j^{d-1/2-2\eta}\delta_j^{-1} = 0.
}
For convenience, let us always choose $\eps_j$ so that $R/\eps_j\in\Z$.
As soon as $\eps_j\leq\eps'$, the estimate \eqref{single_W} leads to
\eeq{ \label{sum_indicator}
\E\bigg[\sum_{z,w\in\eps_j\Z\cap[-R,R]} \one_{\WW^{\eps_j}_{z,w}}\bigg] 
\leq \Big(\frac{2R}{\eps_j}+1\Big)^2\eps_j^{3/2-\eta}
\leq (2R+1)^2\eps_j^{-1/2-\eta}
&\stackref{eps_assumption_3}{\leq} \eps_j^{-1/2-2\eta}.
}
Applying Markov's inequality to \eqref{sum_indicator} results in
\eeq{\label{for_borel_cantelli_1}
\P\bigg(\sum_{z,w\in\eps_j\Z\cap[-R,R]} \one_{\WW^{\eps_j}_{z,w}} \geq \eps_j^{-1/2-2\eta}\delta_j^{-1}\bigg)&\leq \delta_j. 
}
Our final step will be to use this inequality to deduce that the $d$-dimensional Hausdorff content of $\DD\cap[-R,R]^2$ is zero.

If the event appearing in \eqref{for_borel_cantelli_1} does not occur, then Claim~\ref{set_containments} implies that $\DD \cap \big([z,z+\eps_j)\times[w,w+\eps_j)\big)$ is nonempty for at most $\eps_j^{-1/2-2\eta}\delta_j^{-1}$ values of $(z,w)\in (\eps_j\Z\cap[-R,R])^2$.
In this case, $\DD \cap [-R,R]^2$ can be covered by $\eps_j^{-1/2-2\eta}\delta_j^{-1}$ rectangles of diameter less than $2\eps_j$, meaning that 
\eq{
H^d(\DD \cap [-R,R]^2) \leq 2^d\eps_j^{d-1/2-2\eta}\delta_j^{-1}.
}
Since $\sum_{j}\delta_j<\infty$, it follows from \eqref{for_borel_cantelli_1} and Borel--Cantelli that with probability one, the above display is true for all large $j$.
Because of \eqref{eps_sequence_choice}, this implies $H^d(\DD \cap [-R,R]^2)=0$ with probability one.
\end{proof}

\bibliography{airy}

\end{document}